\documentclass{amsart}

\usepackage{amsthm}
\usepackage{newlfont}
\usepackage{amsmath}
\usepackage{amssymb}
\usepackage{amsfonts}
\usepackage[mathscr]{euscript}

\usepackage{mathrsfs}

\usepackage{fancybox}
\usepackage{nicefrac}

\usepackage[arrow,curve,matrix,tips]{xy}
\CompileMatrices 

\usepackage{stmaryrd}

\usepackage{ifthen}

\usepackage{hyperref}

\usepackage{enumitem}
\setenumerate[1]{label=(\alph*)}

\hypersetup{
 pdfauthor = {Olaf~M.~Schn{\"u}rer},
 pdftitle = {Equivariant Sheaves on Flag Varieties},
 pdfsubject = {Mathematics},
 pdfkeywords = {Equivariant Derived Category, Flag Variety, Formality, Differential Graded Module, DG Module, t-Structure},
}

\newtheorem{theorem}{Theorem}
\newtheorem{lemma}[theorem]{Lemma}
\newtheorem{corollary}[theorem]{Corollary}

\newtheorem{proposition}[theorem]{Proposition}

\theoremstyle{definition}

\theoremstyle{remark}
\newtheorem{remark}[theorem]{Remark}

\newcommand{\op}{\operatorname}

\newcommand{\ra}{\rightarrow}

\newcommand{\la}{\leftarrow}

\newcommand{\sra}{\twoheadrightarrow}
\newcommand{\hra}{\hookrightarrow}

\newcommand{\hla}{\hookleftarrow}

\newcommand{\xra}[1]{\xrightarrow{#1}}
\newcommand{\xla}[1]{\xleftarrow{#1}}

\newcommand{\sira}{\xra{\sim}}

\newcommand{\sila}{\overset{\sim}{\leftarrow}}

\newcommand{\silongra}{\overset{\sim}{\longrightarrow}}

\newcommand{\ol}[1]{{\overline{#1}}}
\newcommand{\ul}[1]{{\underline{#1}}}

\newcommand{\bl}{\backslash}

\newcommand{\sar}{\ar@{->>}}
\newcommand{\iar}{\ar@{^{(}->}}
\newcommand{\gar}{\ar@{=}}

\newcommand{\decal}{\mu}

\newcommand{\Bl}[1]{{\mathbb{#1}}}
\newcommand{\DZ}{\Bl{Z}}
\newcommand{\DN}{\Bl{N}}

\newcommand{\DNinfty}{{\DN\cup \{\infty\}}}

\newcommand{\DR}{\Bl{R}}
\newcommand{\DC}{\Bl{C}}

\newcommand{\DP}{{\Bl{P}}}

\newcommand{\hyper}{{\Bl{H}}}

\newcommand{\DD}{{\Bl{D}}}
\newcommand{\HH}{{\Bl{H}}}

 \newcommand{\Hom}{\op{Hom}}
 \newcommand{\complexHom}{{\mkern3mu\mathcal{H}{\op{om}}\mkern3mu}}
 \newcommand{\complexhom}{{\complexHom}}
\newcommand{\complexEnd}{{\mkern3mu\mathcal{E}{\op{nd}}\mkern3mu}}
\newcommand{\cHom}{\complexHom}
\newcommand{\cEnd}{\complexEnd}
\newcommand{\Ext}{\op{Ext}}

\newcommand{\sHom}{{\mkern-1mu\mathscr{H}\mkern-3mu{\op{om}}\mkern3mu}}

\newcommand{\Lotimes}{\overset{L}{\otimes}}
\newcommand{\Lotimesover}[1]{\underset{#1}{\Lotimes}}
\newcommand{\otimesover}[1]{\underset{#1}{\otimes}}

\newcommand{\modover}[1]{{{#1}\text{-}\op{mod}}}

\newcommand{\gmodover}[1]{{#1}\text{-}\op{gmod}}

\newcommand{\gMod}{\op{gMod}}

\newcommand{\End}{\op{End}}

\newcommand{\Der}{\op{Der}}
\newcommand{\Derb}{{\op{Der}^{\bounded}}}

\newcommand{\id}{\op{id}}

\newcommand{\pr}{\op{pr}}

\newcommand{\length}{\lambda}

\newcommand{\ind}{\op{ind}}

\newcommand{\leftadjointtores}{\op{prod}}
\newcommand{\pro}{\leftadjointtores}
\newcommand{\Lpro}{{\pro}}

\newcommand{\GL}{\op{GL}}

\newcommand{\kernel}{\op{ker}}
\newcommand{\Ker}{\op{ker}}
\newcommand{\Kern}{\op{ker}}
\newcommand{\Cok}{\op{cok}}

\newcommand{\Bild}{\op{im}}

\newcommand{\bild}{\op{im}}

\newcommand{\inv}{^{-1}}
\newcommand{\can}{\op{can}}

\newcommand{\Sh}{{\op{Sh}}}

\newcommand{\mar}{{\ar@{|->}}}

\newcommand{\derivcat}{\mathcal{D}}

\newcommand{\dc}{\derivcat}
\newcommand{\dcb}{\derivcat^{\bounded}}
\newcommand{\dbc}{\derivcat^{\bounded}}
\newcommand{\dbcc}{\derivcat^{\bounded}_{\op{c}}}

\newcommand{\Ket}{\op{Ket}}
\newcommand{\Ketb}{\op{Ket}^{\bounded}}
\newcommand{\Hot}{\op{Hot}}

\newcommand{\Perv}{\op{Perv}}

\newcommand{\p}[1]{{{}^{{p}}{#1}}}

\newcommand{\Verd}{\op{thick}}
\newcommand{\PraeVerd}{\op{tria}}

\newcommand{\dgDer}{\op{dgDer}}
\newcommand{\dgHot}{\op{dgHot}}
\newcommand{\dgHotproj}{{\op{dgHotp}}}
\newcommand{\dgMod}{\op{dgMod}}
\newcommand{\dggMod}{\op{dggMod}}

\newcommand{\dgPerDer}{{\op{dgPer}}}

\newcommand{\dgPraeDer}{\op{dgPrae}}

\newcommand{\dgFilMod}{\op{dgFlag}}

\newcommand{\dgFiltDer}{\op{dgFilt}}

\newcommand{\Nat}{\op{Nat}}

\newcommand{\bounded}{\op{b}}
\newcommand{\bd}{{\bounded}}

\newcommand{\point}{\text{pt}}

\newcommand{\Cy}{\op{Z}}

\newcommand{\Ho}{\op{H}}

\newcommand{\Sub}{\op{Sub}}

\newcommand{\Formy}{\op{Form}}
\newcommand{\Forget}{\op{For}}

\renewcommand{\tilde}[1]{\widetilde{#1}}

\renewcommand{\hat}[1]{\widehat{#1}}

\newcommand{\IC}{{\mathcal{I}\mathcal{C}}}

\newcommand{\comp}{\circ}

\newcommand{\invlim}{\varprojlim}

\newcommand{\opp}{{\op{op}}}
\newcommand{\opposite}{\opp}

\newcommand{\gr}{\op{gr}}

\newcommand{\MHM}{\op{MHM}}
\newcommand{\MHS}{\op{MHS}}

\newcommand{\dgMHS}{{\op{dgMHS}}}

\newcommand{\mathovalbox}[1]{{\text{\ovalbox{${#1}$}}}}

\newcommand{\acyclic}[1]{{$#1$}-acyc\-lic}

\newcommand{\asystem}{ap\-proxi\-ma\-tion}

\newcommand{\rat}{\op{rat}}
\newcommand{\real}{\op{real}}
\newcommand{\rera}{{v}}

\newcommand{\define}[1]{{\textbf{#1}}}
\newcommand{\stress}[1]{{\textit{#1}}}

\newcommand{\os}[2]{\overset{#1}{#2}}

\newcommand{\neuerabschnitt}{\medskip}

\begin{document}

\title{Equivariant Sheaves on Flag Varieties}
\author{Olaf M.\ Schn{\"u}rer}
\address{Mathematisches Institut, Universit{\"a}t Bonn, 
Beringstra\ss{}e 1, D-53115 Bonn, Germany}
\email{olaf.schnuerer@math.uni-bonn.de}
\thanks{Supported by the state of Baden-W{\"u}rttemberg}

\date{September 2008.}

\keywords{Equivariant Derived Category, Flag Variety, Formality,
  Perfect Derived Category,
  Differential Graded Module, DG Module, t-Structure.}

\subjclass[2000]{14M15, 18E30}

\begin{abstract}
We show that the Borel-equivariant derived category of sheaves on the
flag variety of a complex reductive group is equivalent to the
perfect derived category of dg modules over the extension algebra of the
direct sum of the simple equivariant perverse sheaves. This proves a
conjecture of Soergel and Lunts in the case of flag varieties.
\end{abstract}

\maketitle

\tableofcontents

\section{Introduction}
\label{cha:introduction}

Let $G$ be a complex connected reductive affine algebraic group and 
$B \subset P \subset G$ a Borel and a parabolic subgroup. 
The main result of this article is an
algebraic description of the $B$-equivariant (bounded, constructible)
derived category $\dcb_{B,c}(X)$  (see \cite{BL}) of sheaves of real vector 
spaces on the partial flag variety $X:=G/P$.

Let $\mathcal{S}$ be the stratification of $X$ into $B$-orbits and
$\IC_B(S) \in \dcb_{B,c}(X)$ the equivariant intersection
cohomology complex of the closure of the stratum $S \in \mathcal{S}$.
The $(\IC_B(S))_{S \in \mathcal{S}}$ are the simple equivariant
perverse sheaves on $X$. Let $\IC_B(\mathcal{S})$ be their direct sum 
and 
$\mathcal{E}=\Ext(\IC_B(\mathcal{S}))$ its graded algebra of
self-extensions 
in ${\dcb_{B,c}(X)}$.
We consider $\mathcal{E}$ as a differential graded (dg) algebra with
differential $d=0$. 
Let $\dgDer(\mathcal{E})$ be the derived category of (right) dg
$\mathcal{E}$-modules (see \cite{Keller-deriving-dg-cat}) and
$\dgPerDer(\mathcal{E})$ the perfect derived category, i.\,e.\ the
smallest strict full triangulated subcategory of $\dgDer(\mathcal{E})$
containing $\mathcal{E}$ and closed under forming direct summands.
We give alternative descriptions of $\dgPerDer(\mathcal{E})$ below.
\begin{theorem}[{cf.\ Theorem~\ref{t:formality-equivariant-partial-flag-variety}}]
  \label{t:einleitung-equivariant-partial}
  There is an equivalence of triangulated categories
  \begin{equation*}
    \dcb_{B,c}(X)
    \cong \dgPerDer(\Ext(\IC_B(\mathcal{S}))).
  \end{equation*}
\end{theorem}

Similar equivalences between equivariant derived categories and
categories of dg modules over the extension algebra of the simple
equivariant perverse sheaves 
are known for a connected Lie group acting on
a point (\cite[12.7.2]{BL}), for a torus acting on an affine or
projective normal toric variety (\cite{Luntstoric}),
and for a complex semisimple adjoint group acting on a smooth
complete symmetric variety (\cite{guillermou}).
The key point in the proof of these equivalences is the formality
of some dg algebra whose cohomology is the extension algebra.

Conjecturally (\cite[0.1.3]{Luntstoric}, \cite[4]{Soergel-langlands}), 
the analog of Theorem
\ref{t:einleitung-equivariant-partial} should hold for 
the equivariant derived category of a complex reductive group acting
on a projective variety with a finite number of orbits.
(Theorem \ref{t:einleitung-equivariant-partial} is a
special case of this 
conjecture since $\dcb_{G,c}(G \times_B X)$ and $\dcb_{B,c}(X)$ are
equivalent by the
induction equivalence.)

Let $\dcb(X)$ be the (bounded) derived category of sheaves of real
vector spaces on $X=G/P$, and $\dcb(X, \mathcal{S})$ the full
subcategory of  $\mathcal{S}$-constructible objects.
Let $\IC(\mathcal{S})$ be the direct sum of the
(non-equivariant) simple $\mathcal{S}$-constructible perverse sheaves on $X$, and
$\mathcal{F}=\Ext(\IC(\mathcal{S}))$ its graded algebra of self-extensions in
$\dcb(X)$. The category $\dgPerDer(\mathcal{F})$ is defined similarly as
$\dgPerDer(\mathcal{E})$ above.
In the course of the proof of Theorem 
\ref{t:einleitung-equivariant-partial}
we obtain the following non-equivariant analog.
\begin{theorem}[{cf.\ Theorem~\ref{t:formality-partial-flag-variety}}]
  \label{t:einleitung-partial}
  There is an equivalence of triangulated categories
  \begin{equation*}
    \dbc(X, \mathcal{S}) 
    \cong \dgPerDer(\Ext(\IC(\mathcal{S}))).
  \end{equation*}
\end{theorem}

The category $\Perv_{B}(X)$ of equivariant
perverse
sheaves on $X$ is the heart of the perverse t-structure on
$\dbc_{B,c}(X)$,
and similarly for $\Perv(X, \mathcal{S}) \subset \dbc(X,
\mathcal{S})$. 
The corresponding t-structure on
$\dgPerDer(\mathcal{E})$ and $\dgPerDer(\mathcal{F})$ can be
defined for a more general class of dg algebras; we explain this below.
It turns out that the heart of such a t-structure is equivalent to 
a full abelian subcategory $\dgFilMod$ of the abelian category of dg
modules.
The equivalences in Theorems 
\ref{t:einleitung-equivariant-partial} and
\ref{t:einleitung-partial} are in fact t-exact and
induce equivalences
\begin{align*}
  \Perv_{B}(X) & \cong \dgFilMod(\Ext(\IC_B(\mathcal{S}))),\\
  \Perv(X, \mathcal{S}) & \cong \dgFilMod(\Ext(\IC(\mathcal{S}))),
\end{align*}
i.\,e.\ algebraic descriptions of the categories of (equivariant)
perverse sheaves. The simple object $\IC_B(S)$ is mapped
to $e_S \mathcal{E}$ 
(where $e_S \in \mathcal{E}$ is the projector from $\IC_B(\mathcal{S})$ onto the
direct summand $\IC_B(S)$), which is an indecomposable projective dg
$\mathcal{E}$-module.
This seems to be part of a Koszul duality
(cf.\ \cite[1.2.6]{BGS}).

The forgetful functor $\Forget:\dcb_{B,c}(X) \ra
\dcb(X,\mathcal{S})$ induces a surjective morphism 
$\mathcal{E} \ra \mathcal{F}$ of dg algebras and an extension of
scalars functor 
$(?\Lotimes_{\mathcal{E}}\mathcal{F}): \dgPerDer(\mathcal{E}) \ra
\dgPerDer(\mathcal{F})$.
These two functors provide a connection between the equivalences in
Theorems \ref{t:einleitung-equivariant-partial} and
\ref{t:einleitung-partial}, 
i.\,e.\ there is a commutative (up to natural isomorphism) square
(see Remark \ref{rem:formality-equivariant-partial-flag-variety})
\begin{equation*}
  \xymatrix{
    {\dcb_{B,c}(X)} \ar[r]^-\sim
    \ar[d]^{\Forget}
    &
    {\dgPerDer(\Ext(\IC_B(\mathcal{S})))}
    \ar[d]^{(?\Lotimes_{\mathcal{E}} \mathcal{F})}
    \\
    {\dcb(X, \mathcal{S})} \ar[r]^-\sim
    &
    {\dgPerDer(\Ext(\IC(\mathcal{S}))).}
  }
\end{equation*}

Let us comment on some purely algebraic results concerning
certain perfect derived categories of dg modules
mentioned above (see \cite{OSdiss-perfect-dg-arXiv}). 
Let $\mathcal{A}=(A=\bigoplus_{i \geq 0} A^i, d)$ be 
a positively graded dg algebra with $A^0$ a semisimple ring and
$d(A^0)=0$ (i.\,e. $A^0$ is a dg subalgebra).
Let $(L_x)_{x \in W}$ be the finite collection of non-isomorphic simple
(right) $A^0$-modules, and $\dgPraeDer(\mathcal{A})$ the smallest
strict full triangulated subcategory of the derived category $\dgDer(\mathcal{A})$ of dg
$\mathcal{A}$-modules that contains all 
$\hat{L}_x:=L_x \otimes_{A^0}\mathcal{A}$ (where $L_x$ is concentrated 
in degree zero).
Let $\dgMod(\mathcal{A})$ be the abelian category of dg
$\mathcal{A}$-modules,
and $\dgFilMod(\mathcal{A})$ the full subcategory of $\dgMod(\mathcal{A})$
consisting of objects that have an 
$\hat{L}_x$-flag, i.\,e.\ a
finite filtration with subquotients isomorphic to objects of
$\{\hat{L}_x\}_{x \in W}$ (without shifts). 
Then $\dgPraeDer(\mathcal{A})$ coincides with $\dgPerDer(\mathcal{A})$ and
carries a natural bounded t-structure.
Moreover 
$\dgFilMod(\mathcal{A})$ is a full abelian subcategory 
of $\dgMod(\mathcal{A})$ and naturally equivalent to the heart of this
t-structure.
Let us note that there is another equivalent
full subcategory of $\dgPerDer(\mathcal{A})$ consisting of certain
filtered dg modules that is quite accessible
to computations (cf.\ Theorem \ref{t:form-filt-iso-perfect}). 

These remarks apply in particular to the
dg algebras $\mathcal{E}$ 
and $\mathcal{F}$ defined above and make the categories of dg modules
appearing in our main equivalences quite explicit.
They also show that the categories of dg modules appearing in the main equivalences 
of \cite{Luntstoric} and \cite{guillermou} are in fact 
of the form $\dgPerDer$.

Assume for this paragraph that we work with sheaves of complex vector
spaces. Our main Theorems \ref{t:einleitung-equivariant-partial} and
\ref{t:einleitung-partial} remain true 
(see subsection \ref{sec:complex-coefficients} and Remark \ref{rem:equivariant-complex}).
Assume now in addition that $G$ is
semisimple and that $P=B$. Then the extension algebras
are isomorphic to morphism spaces of Soergel's bimodules (see 
\cite{Soergel-langlands, Soergel-combi-HC, Soe-Kat}). 
These bimodules are isomorphic to the ($B$-equivariant)
intersection cohomologies of Schubert varieties and
can be described using the
moment graph picture (see \cite{braden-mcpherson-moment-graphs}).
Thus, if $T \subset B$ is a maximal torus,
the $B$-equivariant derived category of the flag
variety $G/B$ only depends on 
the moment graph associated to $T$ acting on $G/B$.

\neuerabschnitt

Let us describe in more detail our approach to prove Theorem
\ref{t:einleitung-equivariant-partial}. 
We use notation
from subsequent sections without further explanation.

\textbf{Step 1}
(see section \ref{cha:form-deriv-categ}).
Let $X$ be a complex variety with a stratification $\mathcal{T}$ into
cells (i.\,e.\ $T\cong \DC^{d_T}$ for each $T \in \mathcal{T}$).
Under some purity assumptions explained below we will establish
an equivalence
\begin{equation}
  \label{eq:step-1}
  \dbc(X, \mathcal{T}) 
  \cong \dgPerDer(\Ext(\IC(\mathcal{T})))
\end{equation}
of triangulated categories, of which 
Theorem \ref{t:einleitung-partial} is a special case.
Note that we could write equivalently $\dgPraeDer$ on the right hand side.
The proof works as follows.
Since $\mathcal{T}$ is a cell-stratification,
there is an equivalence  
\begin{equation*}
  \dbc(\Perv(X,\mathcal{T})) \sira \dbc(X, \mathcal{T}).
\end{equation*}
There are enough projective objects in $\Perv(X, \mathcal{T})$,
so we find
projective resolutions $P_T \ra \IC(T)$ of finite length
\begin{equation*}
  \dots \ra P_T^{-2} \ra P_T^{-1} \ra P_T^0 \ra \IC(T)\ra 0.
\end{equation*}
Let $P \ra \IC(\mathcal{T})$ be the direct
sum of these resolutions and 
$\mathcal{B}=\cEnd(P)$ the dg
algebra of endomorphisms of $P$. The functor $\cHom(P,?)$ induces an
equivalence
\begin{equation*}
  \dcb(\Perv(X, \mathcal{T})) \sira \dgPraeDer_{\mathcal{B}}(\{e_T\mathcal{B}\}_{T \in \mathcal{T}}).
\end{equation*}
Note that the cohomology of $\mathcal{B}$ is isomorphic to $\Ext(\IC(\mathcal{T}))$.
Thus we obtain equivalence \eqref{eq:step-1} if $\mathcal{B}$ is formal. 
In order to prove formality, we need to choose the resolutions 
$P_T \ra \IC(T)$ more carefully.

Each $\IC(T)$ is the underlying perverse sheaf of a mixed Hodge module
$\tilde{\IC}(T)$ that is pure of weight $d_T$. 
We construct resolutions $\tilde P_T \ra
\tilde{\IC}(T)$ in the category of mixed Hodge modules
so that the underlying resolutions $P_T \ra \IC(T)$ are projective
resolutions as considered above. 
From these resolutions we get a dg algebra 
of
mixed Hodge structures with underlying dg algebra
$\mathcal{B}=\cEnd(P)$.
If each $\tilde{\IC}(T)$ is $\mathcal{T}$-pure of weight $d_T$ (i.\,e.\ all
restrictions to strata in $\mathcal{T}$ are pure of weight $d_T$),
this additional structure on $\mathcal{B}$
enables us to construct  
a dg subalgebra $\Sub(\mathcal{B})$ of
$\mathcal{B}$ and quasi-isomorphisms 
\begin{equation*}
  \mathcal{B} \hla \Sub(\mathcal{B}) \sra \Ho(\mathcal{B})
\end{equation*}
of dg algebras, establishing the formality of $\mathcal{B}$. 

We will need the following slightly more general statement than equivalence
\eqref{eq:step-1}, with essentially the same proof. 
 
\begin{theorem}[{cf.\ Theorem \ref{t:formality-ic}}]
  \label{t:einleitung-formality-ic}
  Let $(X, {\mathcal{S}})$ be a stratified complex variety with
  irreducible and simply connected 
  strata. Let $\mathcal{T}$ be a cell-stratification refining
  ${\mathcal{S}}$. 
  If $\tilde{\IC}({S})$ is $\mathcal{T}$-pure of weight
  $d_{S}$ for each $S \in {\mathcal{S}}$, there is a triangulated
  equivalence 
  \begin{equation*}
    \dcb(X, {\mathcal{S}})
    \cong \dgPerDer(\Ext(\IC({\mathcal{S}}))).
  \end{equation*}
\end{theorem}

\textbf{Step 2}
(see section \ref{cha:form-clos-embedd}).
Let $(X, \mathcal{S})$ and $(Y, \mathcal{T})$ be stratified complex varieties
with irreducible and simply connected strata.
Let $i:Y \ra X$ be a closed embedding so that
$\mathcal{S} \ra \mathcal{T}$, $S \mapsto S\cap Y$, is
bijective and $i|_{\ol S\cap Y}: \ol S \cap Y \ra \ol S$ is a
normally nonsingular inclusion of a fixed codimension $c$ for all $S \in
\mathcal{S}$. 
Then 
\begin{equation}
  \label{eq:einl-normally-nonsingular-ic}
  [-c]i^*(\tilde{\IC}(S)) \sira \tilde{\IC}(S\cap Y)
\end{equation}
for all $S \in \mathcal{S}$.
If both stratifications $\mathcal{S}$ and $\mathcal{T}$ have
compatible refinements by cell-stratifications satisfying the purity conditions of
Theorem \ref{t:einleitung-formality-ic}, we obtain the vertical
equivalences in the following diagram.
\begin{equation}
  \label{eq:einleitung-the-goal}
  \xymatrix@C+0.25cm{
    {\dcb(X, {\mathcal{S}})}
    \ar[rr]^{[-c]i^*}
    \ar[d]^\sim
    &&
    {\dcb(Y, {\mathcal{T}})}
    \ar[d]^\sim
    \\
    {\dgPerDer(\Ext(\IC({\mathcal{S}})))}
    \ar[rr]_-{? \Lotimesover{\Ext(\IC({\mathcal{S}}))} \Ext(\IC({\mathcal{T}}))}
    &&
    {\dgPerDer(\Ext(\IC({\mathcal{T}})))}
  }
\end{equation}
The extension of scalars functor in the lower row is induced by the
isomorphisms \eqref{eq:einl-normally-nonsingular-ic}. 
This diagram is
commutative (up to natural isomorphism).
Unfortunately the proof is rather technical.

\textbf{Step 3}
(see section \ref{cha:form-equiv-deriv}).
Let $X=G/P$ be a partial flag variety
with stratification $\mathcal{S}$ into $B$-orbits as before. 
We construct a sequence 
\begin{equation*}
  E_0 \xra{f_0} E_1 \xra{f_1} 
  E_2 \ra 
  \dots 
  \ra 
  E_n \xra{f_n} E_{n+1} \ra \dots
\end{equation*}
of resolutions $p_n:E_n \ra X$ of $X$ satisfying several nice
properties. For example, each $p_n$ is smooth
and \acyclic{n} (in the classical topology), the quotient
morphisms $q_n:E_n \ra \ol{E}_n:= B\bl E_n$ are Zariski locally
trivial fiber bundles, and each
$\mathcal{S}_n :=\{q_n(p_n\inv(S)) \mid S \in \mathcal{S}\}$ is a
stratification of $\ol{E}_n$. 
The induced morphisms
$\ol{f}_n:\ol{E}_n \ra \ol{E}_{n+1}$
define functors 
$\ol{f}_n^*:\dcb(\ol{E}_{n+1}, \mathcal{S}_{n+1}) \ra 
\dcb(\ol{E}_{n}, \mathcal{S}_{n})$, and we obtain a sequence of
categories whose inverse limit is equivalent to the
category we want to describe,
\begin{equation*}
  \dcb_{B,c}(X) \cong \invlim \dcb(\ol{E}_{n}, \mathcal{S}_{n}).
\end{equation*}
Moreover, the morphisms
$\ol{f}_n$
satisfy the assumptions of Step 2
(in particular, the stratifications $\mathcal{S}_n$ admit refinements where the
purity conditions hold),
and the obtained commutative diagrams
of the form \eqref{eq:einleitung-the-goal} provide
an equivalence 
\begin{equation*}
  \invlim \dcb(\ol{E}_{n}, \mathcal{S}_{n}) \sira \invlim \dgPerDer(\Ext(\IC({\mathcal{S}_n}))).
\end{equation*}
Finally, the obvious morphisms 
$\mathcal{E}=\Ext(\IC_B(\mathcal{S}))\ra
\Ext(\IC({\mathcal{S}_n}))$ of dg algebras induce an equivalence
\begin{equation*}
  \dgPerDer(\Ext(\IC_B({\mathcal{S}}))) \sira \invlim \dgPerDer(\Ext(\IC({\mathcal{S}_n}))).
\end{equation*}
This finishes the sketch of proof of
Theorem \ref{t:einleitung-equivariant-partial}.

\neuerabschnitt

It would be nice to know whether the analog of Theorem
\ref{t:einleitung-equivariant-partial}
is true for $\dcb_{Q,c}(G/P)$ if $Q$ is a parabolic subgroup of $G$
containing $B$. 
Theorem \ref{t:formality-partial-flag-variety} shows
that the non-equivariant version holds, i.\,e.\ we can
replace the
stratification $\mathcal{S}$ in Theorem \ref{t:einleitung-partial} by
the stratification into $Q$-orbits.
We expect that our methods can be generalized to affine flag varieties.

This
article is organized as follows:
In section~\ref{sec:differential-graded-modules}
we introduce the main categories of dg modules, show how dg modules
can be used to describe certain triangulated categories and prove an
elementary but crucial result establishing the formality of some dg
algebras with an additional grading.
Sections~\ref{cha:form-deriv-categ},~\ref{cha:form-clos-embedd} and \ref{cha:form-equiv-deriv} contain
essentially the results explained above
in Steps 1, 2 and 3 respectively.
However the methods are developed in a broader context and may be applied to other situations.
Section~\ref{sec:inverse-limits} contains some results on inverse limits of categories (of
dg modules) used in Step 3.

\subsection*{Acknowledgments}
\label{sec:ack}
This article contains the main results of my thesis
\cite{OSdiss} written at the University of Freiburg. I am very
grateful to my advisor Wolfgang Soergel for all his advice and enthusiasm.
I would like to thank Peter Fiebig, Geordie Williamson, and Anne and Martin Balthasar
for useful comments and discussions.

\section{Differential Graded Modules}
\label{sec:differential-graded-modules}

\subsection{DG Modules}
\label{sec:review-dg-modules}

We review the language of differential graded
(dg) modules over a dg algebra (see \cite{Keller-deriving-dg-cat,
  Keller-construction-of-triangle-equiv, BL}).

Let $k$ be a commutative ring and 
$\mathcal{A}=(A=\bigoplus_{i \in \DZ} A^i, d)$ 
a differential graded $k$-algebra (= dg algebra). 
A dg (right) module over $\mathcal{A}$ will also be called an
$\mathcal{A}$-module or a dg module if there is no doubt about the dg algebra.  
We often write $M$ for a dg module $(M, d_M)$.
We consider the category $\dgMod(\mathcal{A})$ of dg modules, the
homotopy category $\dgHot(\mathcal{A})$ and the derived category
$\dgDer(\mathcal{A})$ of dg modules.  
We denote the shift functor on all these categories 
by $M \mapsto \{1\}M$, e.\,g.\
$(\{1\}M)^i=M^{i+1}$, $d_{\{1\}M}= -d_M$. 
We define $\{n\}=\{1\}^n$ for $n
\in \DZ$.
Both $\dgHot(\mathcal{A})$ and $\dgDer(\mathcal{A})$ are triangulated categories.

A dg module ${P}$ is called \define{homotopically projective}
(\cite{Keller-construction-of-triangle-equiv}),
if it satisfies one of
the following equivalent conditions
(\cite[10.12.2.2]{BL}):
\begin{enumerate}
\item $\Hom_{\dgHot}({P}, ?) = \Hom_{\dgDer}({P},?)$, i.\,e.\ 
  for all dg modules $M$, the canonical map $\Hom_{\dgHot}({P}, M) \ra
  \Hom_{\dgDer}({P},M)$ is an isomorphism.
\item 
  $\Hom_{\dgHot}({P}, {M})=0$ for each acyclic dg module ${M}$.
\end{enumerate}
In \cite[3.1]{Keller-deriving-dg-cat} such a module is said to have \stress{property (P)},
in \cite[10.12.2]{BL} the term \stress{$\mathcal{K}$-projective} is
used.
For example, $\mathcal{A}$ and each direct summand
of $\mathcal{A}$ is homotopically projective.

Let $\dgHotproj(\mathcal{A})$ be the full subcategory of $\dgHot(\mathcal{A})$
consisting of homotopically projective dg modules.
The quotient functor
$\dgHot(\mathcal{A}) \ra \dgDer(\mathcal{A})$
induces a triangulated
equivalence (\cite[3.1, 4.1]{Keller-deriving-dg-cat}) 
\begin{equation}
  \label{eq:dgHotp-equiv-dgDer}
  \dgHotproj(\mathcal{A}) \sira \dgDer(\mathcal{A}).
\end{equation}

Let $\dgPerDer(\mathcal{A})$ be the perfect derived category,
i.\,e.\ the smallest strict (= closed under isomorphisms) full triangulated subcategory of
$\dgDer(\mathcal{A})$ containing $\mathcal{A}$ and closed under
forming direct summands.

Each morphism of dg algebras (\define{dga-morphism}) $f:\mathcal{A} \ra \mathcal{B}$ induces on
cohomology a
dga-morphism $\Ho(f):\Ho(\mathcal{A}) \ra
\Ho(\mathcal{B})$. If $\Ho(f)$ is an isomorphism, $f$ is called a
\define{dga-quasi-isomorphism}.
Two dg algebras $\mathcal{A}$ and $\mathcal{B}$ are \define{equivalent} 
if there is a sequence
$  \mathcal{A} \la \mathcal{C}_1 \ra \mathcal{C}_2 \la \dots \ra \dots \la
  \mathcal{C}_n \ra \mathcal{B}$
of dga-quasi-isomorphisms. A dg algebra $\mathcal{A}$ is \define{formal} if it is
equivalent to a dg algebra with differential $d=0$. In this case,
$\mathcal{A}$ is equivalent to $\Ho(\mathcal{A})$.

If $\mathcal{A} \ra \mathcal{B}$ is a morphism of dg algebras,
we have the  \define{extension of scalars} functor
(\cite[10.11]{BL})
\begin{equation*}
  \pro_{\mathcal{A}}^\mathcal{B} = (? \otimes_{\mathcal{A}}
  \mathcal{B}): \dgMod(\mathcal{A}) \ra \dgMod(\mathcal{B}).   
\end{equation*}
It descends to a triangulated functor
$\pro_{\mathcal{A}}^\mathcal{B} = (? \otimes_{\mathcal{A}}
\mathcal{B})$ between the homotopy categories 
and has the left derived functor $\Lpro_{\mathcal{A}}^\mathcal{B} = (? \Lotimes_{\mathcal{A}}
  \mathcal{B})$ on the level of derived categories.
This left derived functor is an equivalence if $\mathcal{A} \ra
\mathcal{B}$ is a dga-quasi-isomorphism
(\cite[6.1]{Keller-deriving-dg-cat}).

\subsection{Differential Graded Graded Algebras and Formality}
\label{sec:diff-grad-grad}

We show that some dg algebras with an extra grading are formal.
Let $k$ be a commutative ring and 
$\mathcal{R}$
a differential
graded graded (dgg) algebra, i.\,e.\ a $\DZ^2$-graded associative
$k$-algebra $R=\bigoplus_{i,j \in \DZ}R^{ij}$ endowed with a
$k$-linear differential $d: R \ra R$ that is homogeneous of degree
$(1,0)$ and satisfies the Leibniz rule
$d(a b)=(da) b + (-1)^i a db$
for all $a \in R^{ij}$, $b \in R^{kl}$.
A dgg module $\mathcal{M}=(M, d)$ over
$\mathcal{R}$ is a $\DZ^2$-graded right $R$-module 
$M=\bigoplus_{i,j \in \DZ}M^{ij}$ with a $k$-linear differential 
$d: M \ra M$ of degree $(1, 0)$ satisfying
$d(m a)=(dm)a + (-1)^i m da$
for all $m \in M^{ij}$, $a \in R^{kl}$. 
Morphisms of dgg modules are morphisms of the underlying
$\DZ^2$-graded $R$-modules of degree $(0, 0)$ that commute with the
differentials. 
We denote the category of dgg modules over $\mathcal{R}$ 
by $\dggMod(\mathcal{R})$.

The cohomology of a dgg module over a dgg algebra $\mathcal{R}$ is a
dgg module over the the dgg algebra $\Ho(\mathcal{R})$.
Morphisms of dgg algebras (\define{dgga-morphisms}) are algebra homomorphisms
that are morphisms of dgg modules.
The meanings of \define{dgga-quasi-isomorphism}, \define{equivalent}
and \define{formal} are the obvious generalizations from dg algebras.

A $\DZ^2$-graded $k$-module $M = \bigoplus M^{ij}$ is \define{pure of weight} $w$, if $M^{ij}\not= 0$
implies $j=i+w$. 
A dgg module or algebra is pure of weight $w$ if the underlying bigraded module
is pure of weight $w$.
Every pure dgg algebra $\mathcal{R}$ of weight $w\not=0$ is the zero algebra, since $1
\in R^{00}$; hence it is also pure of weight $0$.

Let $M$ be in $\dggMod(\mathcal{R})$. We define a bigraded
$k$-submodule $\Gamma(M)$ of $M$ 
by 
\begin{equation}
  \label{eq:def-gamma}
  \Gamma(M)^{ij}=
  \begin{cases}
    M^{ij} & \text{if $i<j$,}\\
    \kernel(d^{ij}: M^{ij} \ra M^{i+1,j}) & \text{if $i=j$,}\\
    0 & \text{if $i>j$.}
  \end{cases}
\end{equation}
The differential of $M$ restricts to a differential of $\Gamma(M)$. 
The multiplication on $\mathcal{R}$ restricts to a multiplication on
$\Gamma(\mathcal{R})$ and $\Gamma(\mathcal{R})$ becomes a dgg
algebra. Similarly, $\Gamma(M)$ is a dgg module over
$\Gamma(\mathcal{R})$.
In fact, we obtain a functor
\begin{equation}
  \label{eq:functor-gamma}
  \Gamma: \dggMod(\mathcal{R}) \ra \dggMod(\Gamma(\mathcal{R})).
\end{equation}

\begin{proposition}\label{p:formality-of-pure-dgg-algebras}
  If the cohomology $\Ho(\mathcal{R})$ of a dgg algebra $\mathcal{R}$
  is pure of weight $0$,
  then $\mathcal{R}$ is formal. 
  More precisely,
  $\mathcal{R} \hla \Gamma(\mathcal{R}) \sra \Ho(\mathcal{R})$
  are dgga-quasi-isomorphisms where $\Gamma(\mathcal{R}) \hra \mathcal{R}$ is the obvious
  inclusion and $\Gamma(\mathcal{R}) \sra \Ho(\mathcal{R})$ the componentwise
  projection.
\end{proposition}

\begin{proof}
  Let $\mathcal{R}$ be an arbitrary dgg algebra.
  We include the following picture illustrating the 
  morphisms  $\mathcal{R} \hla \Gamma(\mathcal{R}) \sra
  \Ho(\mathcal{R})$.
  The differentials go to the right, the cocycle and cohomology modules
  of the complexes $R^{*j}$ are denoted by
  $\Cy^{ij}$ and $\Ho^{ij}$ respectively.  
  \begin{equation*}
    \xymatrix{
      {\begin{vmatrix}
          R^{01}& \ra & R^{11} \\
          R^{00}& \ra & R^{10}   
        \end{vmatrix}} 
      &
      \ar@{_{(}->}[l]
      {\begin{vmatrix}
          R^{01}  & \ra & \Cy^{11}  \\
          \Cy^{00} & \ra & 0 
        \end{vmatrix}}
      \ar@{->>}[r]
      &
      {\begin{vmatrix}
          0        & \ra & \Ho^{11} \\
          \Ho^{00} & \ra & 0       
        \end{vmatrix}}
    }
  \end{equation*}
  The proposition results from
  the following evident statements.
  \begin{enumerate}
  \item The dgga-inclusion $\Gamma(\mathcal{R}) \hra \mathcal{R}$ induces on cohomology an
    isomorphism in degrees $(i,j)$
    with $i \leq j$ (above the diagonal).
  \item If $\Ho(\mathcal{R})$ vanishes in degrees $(i,j)$ with $i < j$ (the
    cohomology lives below the diagonal), componentwise projection
    $\Gamma(\mathcal{R}) \ra \Ho(\mathcal{R})$ is a well-defined dgga-morphism and induces
    on cohomology an isomorphism in degrees $(i,i)$ (on the diagonal).
  \end{enumerate}
\end{proof}

We generalize Proposition \ref{p:formality-of-pure-dgg-algebras} slightly
to dgg algebras that look like matrices.
Let $\mathcal{R}$ be a dgg algebra and $\{e_\alpha\}_{\alpha \in I}$ a
finite set of orthogonal idempotent elements of $R^{00}$ satisfying 
$1=\sum_{\alpha \in I} e_\alpha$ and $d(e_\alpha)=0$ 
for all $\alpha \in I$. 
We get a direct sum decomposition 
$R=\bigoplus R_{\alpha\beta}$ where 
$R_{\alpha\beta}:=e_\alpha Re_\beta$ for $\alpha$, $\beta \in I$.
The differential of $\mathcal{R}$ induces differentials on each
component $R_{\alpha\beta}$. In particular, we can consider the
cohomologies $\Ho(R_{\alpha\beta})$.

\begin{proposition}\label{p:formality-of-pure-dgg-algebras-with-objects}
  Let $\mathcal{R}$ and $\{e_\alpha\}_{\alpha \in I}$ be as above.
  If there are integers $(n_\alpha)_{\alpha \in I}$
  such that each $\Ho(\mathcal{R}_{\alpha\beta})$ is pure of weight 
  $n_\alpha - n_\beta$, 
  then $\mathcal{R}$ is formal. More precisely, 
  there are a dgg subalgebra $\mathcal{S}$ of $\mathcal{R}$ containing
  all $\{e_\alpha\}_{\alpha \in I}$ and quasi-isomorphisms 
  $\mathcal{R} \hla \mathcal{S} \sra \Ho(\mathcal{R})$
  of dgg algebras.
\end{proposition}

\begin{proof}
  Define $S=\bigoplus S^{ij}_{\alpha\beta} \subset R$ by
  \begin{equation*}
    S^{ij}_{\alpha\beta}=
    \begin{cases}
      R^{ij}_{\alpha\beta} & \text{if $i+n_\alpha-n_\beta<j$,}\\
      \kernel(d^{ij}_{\alpha\beta}:R^{ij}_{\alpha\beta} \ra
      R^{i+1,j}_{\alpha\beta}) & \text{if $i+n_\alpha-n_\beta=j$,}\\
      0 & \text{if $i+n_\alpha-n_\beta>j$.}
    \end{cases}
  \end{equation*}
  With the induced multiplication and differential, $S$ becomes a dgg
  subalgebra $\mathcal{S}$ of $\mathcal{R}$.
  The inclusion $\mathcal{S} \hra \mathcal{R}$ and the obvious projection $\mathcal{S} \sra \Ho(\mathcal{R})$
  are easily seen to be quasi-isomorphisms of dgg algebras. 
\end{proof}

\subsection{Subcategories of Triangulated Categories}
\label{sec:subc-triang-categ}

Let $\mathcal{T}$ be a triangulated category, with shift 
$X \mapsto [1]X$. 
If $M$ is a set of
objects of $\mathcal{T}$, we denote by 
$\PraeVerd(M)=\PraeVerd(M, \mathcal{T})$
the smallest strict full triangulated
subcategory of $\mathcal{T}$ that contains 
all objects of $M$, and by 
$\Verd(M)=\Verd(M, \mathcal{T})$
the closure of $\PraeVerd(M)$ under taking direct summands.
If $X$ is an object of $\mathcal{T}$, 
we abbreviate $\PraeVerd(\{X\})$ by $\PraeVerd(X)$, and similarly for $\Verd$.

\begin{lemma}[Beilinson's Lemma]\label{l:praeverdier-cat-iso}
  Let $F:\mathcal{T} \ra \mathcal{T}'$ be a triangulated functor between
  triangulated categories, and let $M$ be a set of objects of
  $\mathcal{T}$. If $F$ induces isomorphisms
  \begin{equation*}
    \Hom_{\mathcal{T}}(X, [i]Y) \sira \Hom_{\mathcal{T}'}(F(X), [i]F(Y)),
  \end{equation*}
  for all $X$, $Y$ in $M$ and all $i \in \DZ$, it induces a triangulated
  equivalence 
  \begin{equation*}
    \PraeVerd(M) \sira \PraeVerd(F(M)),
  \end{equation*}
  where $F(M)=\{F(X) \mid X \in M\}$.
\end{lemma}

\begin{proof}
  This follows by a standard d{\'e}vissage argument.
\end{proof}

\subsection{Derived Categories and DG Modules}
\label{sec:deriv-categ-dg}

Let $\mathcal{A}$ be an abelian category. We denote by
$\Ket(\mathcal{A})$, $\Hot(\mathcal{A})$ and $\Der(\mathcal{A})$ (or $\dc(\mathcal{A})$) the
category of (cochain) complexes in $\mathcal{A}$, the homotopy
category of complexes in $\mathcal{A}$ and the derived category of
$\mathcal{A}$ respectively, with shift functor $A \mapsto [1]A$.
We often
consider $\mathcal{A}$ as a full subcategory of these categories, consisting
of complexes (with cohomology) concentrated in degree zero.
If $I$ is a subset of $\DZ$, we write $\Der^I(\mathcal{A})$ for the
full subcategory of $\Der(\mathcal{A})$ with objects whose cohomology
vanishes in degrees outside $I$.
For objects $A$ and $B$ in the derived category of $\mathcal{A}$, we
write $\Ext^n_\mathcal{A}(A,B)$ for
$\Hom^n_{\Der(\mathcal{A})}(A,B) :=\Hom_{\Der(\mathcal{A})}(A,[n]B)$
and 
$\Ext_\mathcal{A}(A, B)$ for the direct sum of all
$\Ext_\mathcal{A}^n(A, B)$, $n \in \DZ$.
We call 
\begin{equation}
  \label{eq:ext-algebra}
  \Ext(A) := \Ext_{\mathcal{A}}(A,A) =
  \bigoplus_{n \in \DZ}\Ext_\mathcal{A}^n(A, A)\\.
\end{equation}
the \define{extension algebra} of $A$.

If $M$, $N$ are
complexes in $\mathcal{A}$, let
$\complexHom(M,N)$ 
or $\complexHom_{\mathcal{A}}(M,N)$
denote the complex of abelian groups with $n$-th component 
\begin{equation*}
  \complexHom^n(M,N)
  =\prod_{i+j=n}\Hom_{\mathcal{A}}(M^{-i}, N^j)
\end{equation*}
and differential 
$df = d \comp f - (-1)^n f \comp d$
for each homogeneous $f$ of degree $n$. 
The $n$-th cohomology of this complex is
$\Hom_{\Hot(\mathcal{A})}(M,[n]N)$.
With the obvious composition, $\complexHom(M,M)$ becomes a dg
algebra that we denote by $\complexEnd(M)$.
The functor  
\begin{equation*}
  \complexhom(M, ?):\Ket(\mathcal{A}) \ra \dgMod(\complexEnd(M)),
\end{equation*} 
induces a triangulated functor between the homotopy categories.

Recall the category
$\dgPerDer(\mathcal{R})$
for a dg algebra $\mathcal{R}$ from subsection 
\ref{sec:review-dg-modules}. By definition, it is equal to 
$\Verd(\mathcal{R},\dgDer(\mathcal{R}))$.
If $M$ is a set of $\mathcal{R}$-modules, we define
\begin{equation*}
  \dgPraeDer_{\mathcal{R}}(M) := \PraeVerd(M,\dgDer(\mathcal{R})).
\end{equation*}

\begin{proposition}\label{p:hom-p-dot-iso}
  Let $\mathcal{A}$ be an abelian category, and
  $\{P_\alpha\}_{\alpha \in I}$ a finite set of complexes in
  $\mathcal{A}$ such that the canonical maps
  \begin{equation}
    \label{eq:endazyclic}
    \Hom_{\Hot(\mathcal{A})}(P_\alpha,[n]P_\beta) \ra \Hom_{\dc(\mathcal{A})}(P_\alpha,[n]P_\beta)
  \end{equation}
  are isomorphisms for all $n \in \DZ$ and all $\alpha$, $\beta \in I$.
  (For example, all $P_\alpha$ could be bounded
  above complexes of projective objects of $\mathcal{A}$.) 
  Define $P=\bigoplus P_\alpha$ and $\mathcal{R}= \cEnd(P)$.
  Let $e_\alpha \in \mathcal{R}^0$ be
  the projector from $P$ onto its direct summand $P_\alpha$.
  Then the functor $\complexHom(P, ?)$ induces a triangulated equivalence
  \begin{equation}
    \label{eq:chompequi}
    \PraeVerd(\{P_\alpha\}_{\alpha \in I}, \dc(\mathcal{A})) \sira
    \dgPraeDer_{\mathcal{R}}(\{e_\alpha\mathcal{R}\}_{\alpha \in I}).
  \end{equation}
\end{proposition}

\begin{proof}
  Consider the diagram
  \begin{equation*}
    \xymatrix{
      {\PraeVerd(\{P_\alpha\}_{\alpha \in I}, \Hot(\mathcal{A}))}
      \ar[d]
      \ar[rr]^-{\cHom(P, ?)} &&
      {\PraeVerd(\{e_\alpha\mathcal{R}\}_{\alpha \in I}, \dgHot({\mathcal{R}}))}
      \ar[d] \\
      {\PraeVerd(\{P_\alpha\}_{\alpha \in I}, \dc(\mathcal{A}))} &&
      {\dgPraeDer_{\mathcal{R}}(\{e_\alpha\mathcal{R}\}_{\alpha \in I})}
    }
  \end{equation*}
  with obvious vertical functors. We claim that all arrows are
  equivalences.
  For the arrow on the left this follows from
  \eqref{eq:endazyclic} and Lemma \ref{l:praeverdier-cat-iso}.
  Since all $e_\alpha\mathcal{R}$ are homotopically projective dg
  modules, equivalence \eqref{eq:dgHotp-equiv-dgDer} restricts to the
  equivalence on the right.
  Since $\Hom_{\Hot(\mathcal{A})}(P_\alpha,[n]P_\beta)$ and
  $\Hom_{\dgHot(\mathcal{R})}(e_\alpha
  \mathcal{R},[n]e_\beta\mathcal{R})$ 
  are both naturally identified
  with 
  $\Ho^n(e_\beta\mathcal{R}e_\alpha)$ and these identifications are
  compatible with the functor $\complexhom(P,?)$,
  Lemma \ref{l:praeverdier-cat-iso} proves our claim.
\end{proof}

\begin{remark}\label{rem:hom-p-dot-iso}
    Let $P$ be a complex in an abelian category $\mathcal{A}$ with
    endomorphism complex $\mathcal{R}=\cEnd(P)$. 
    If the composition 
    \begin{equation}
      \label{eq:qchomp}
      \Hot(\mathcal{A}) \xra{\cHom(P, ?)}
      \dgHot(\mathcal{R}) \ra \dgDer(\mathcal{R})
    \end{equation}
    vanishes on acyclic complexes,
    it
    factors through $q:\Hot(\mathcal{A}) \ra \Der(\mathcal{A})$ to a
    triangulated functor
    \begin{equation}
      \label{eq:chomp}
      \cHom(P, ?): \Der(\mathcal{A}) \ra \dgDer(\mathcal{R}).
    \end{equation}
    This is the case, for example, if $P$ is a bounded
    above complex of projective objects of $\mathcal{A}$.

    If we keep the assumptions of Proposition \ref{p:hom-p-dot-iso} and
    assume that the composition \eqref{eq:qchomp} vanishes on acyclic
    complexes, then the restriction of \eqref{eq:chomp} yields directly equivalence 
    \eqref{eq:chompequi}.
\end{remark}

\subsection{Perfect DG Modules}
\label{sec:perfect-dg-modules}

We recall some results from \cite{OSdiss-perfect-dg-arXiv}.
We assume in this subsection that $\mathcal{A}=(A,d)$ is a dg algebra
satisfying the following conditions:
\begin{enumerate}[label={(P\arabic*)}]
\item 
\label{enum:form-pg}
$A$ is positively graded, i.\,e.\ $A^i=0$ for $i < 0$;
\item 
\label{enum:form-ss}
$A^0$ is a semisimple ring;
\item 
\label{enum:form-sdga}
the differential of $\mathcal{A}$ vanishes on $A^0$, i.\,e.\ $d(A^0)=0$. 
\end{enumerate}

The semisimple ring $A^0$ has only a finite number
of non-isomorphic simple (right) modules $(L_x)_{x \in W}$.
We view $A^0$ as a dg subalgebra $\mathcal{A}^0$ of $\mathcal{A}$ and the
$L_x$ as $\mathcal{A}^0$-modules concentrated in degree zero.
Extension of scalars yields 
$\mathcal{A}$-modules $\hat L_x:= L_x \otimes_{\mathcal{A}^0}
\mathcal{A}$.
Define 
\begin{equation*}
  \dgPraeDer(\mathcal{A})
:=\dgPraeDer_{\mathcal{A}}(\{L_x\}_{x \in W}).
\end{equation*}

Let $\dgPerDer^{\leq 0}$ (and $\dgPerDer^{\geq 0}$ resp.) be the full
subcategories of $\dgPerDer(\mathcal{A})$ consisting of objects
$\mathcal{M}$ such that $\Ho^i(\mathcal{M} \Lotimes_{\mathcal{A}} \mathcal{A}^0)$
vanishes for $i > 0$ (for $i < 0$ respectively).
Let $\dgFilMod(\mathcal{A}) \subset \dgMod(\mathcal{A})$ be 
the full subcategory
consisting of objects that have an 
$\hat{L}_x$-flag, i.\,e.\ a
finite filtration with subquotients isomorphic to objects of $\{\hat{L}_x\}_{x \in W}$
(without shifts). 

\begin{theorem}[\cite{OSdiss-perfect-dg-arXiv}]
  \label{t:t-structure-auf-perf}
  Let $\mathcal{A}$ be a dg algebra satisfying 
  \ref{enum:form-pg}-\ref{enum:form-sdga}.
  \begin{enumerate}[label={(\arabic*)}]
  \item 
    Then $\dgPraeDer(\mathcal{A})=\dgPerDer(\mathcal{A})$, i.\,e.\
    $\dgPraeDer(\mathcal{A})$ is closed under taking direct summands.
  \item 
    $(\dgPerDer^{\leq 0}, \dgPerDer^{\geq 0})$ defines a bounded 
    (hence non-de\-gen\-er\-ate) 
    t-struc\-ture on
    $\dgPerDer(\mathcal{A})$.
  \item 
  Its heart $\dgPerDer^0$ 
  is
  equivalent to $\dgFilMod(\mathcal{A})$. More precisely,
  $\dgFilMod(\mathcal{A})$ is a full abelian subcategory
  of $\dgMod(\mathcal{A})$ and the obvious functor
  $\dgMod(\mathcal{A}) \ra \dgDer(\mathcal{A})$ 
  induces an equivalence 
  $\dgFilMod(\mathcal{A}) \sira \dgPerDer^0$.
  \item 
    Any object in the heart $\dgPerDer^0$ has finite length, and  
    the simple objects in $\dgPerDer^0$ are (up to isomorphism) the
    $\{\hat{L}_x\}_{x \in W}$.
  \end{enumerate}
\end{theorem}

\section{Formality of Derived Categories}
\label{cha:form-deriv-categ}

\subsection{Sheaves and Perverse Sheaves}
\label{sec:perv-sheav-strat}

We only consider complex (algebraic) varieties. 
Let $X$ be a variety. 
We denote by $\Sh(X)$ the abelian
category of sheaves of real vector spaces with respect to the
classical topology on $X$
and by 
$\dcb(X)=\Der^\bd(\Sh(X))$ its bounded derived category. 
Let $\dbcc(X)$ be the full triangulated subcategory of $\dbc(X)$,
consisting of complexes with algebraically constructible cohomology
(\cite[2.2.1]{BBD}).

Any morphism $f:X \ra Y$ of varieties gives rise to functors
$f^*$, $f_*$, $f_!$, $f^!$ relating $\dbc(X)$ and $\dbc(Y)$.
These functors would classically be written $f\inv$, $R\! f_*$,
$R\! f_!$ and $f^!$ respectively. Similarly we write $\otimes$ and
$\sHom$ for the derived functors of tensor product and sheaf
homomorphisms.
We denote the constant sheaf with stalk $\DR$ on $X$ by $\ul{X}$.
Verdier duality is defined by $\DD =\DD_X =\sHom(?, c^!(\ul{\point}))$
where $c:X \ra \point$ is the unique map to the final object $\point$
in the category of varieties. We have
$\DD f_* = f_! \DD$, $\DD f^* = f^! \DD$, and $\DD^2=\id$ on $\dbcc(X)$.

An algebraic stratification of $X$ is a
finite partition $\mathcal{S}$ of $X$ into non-empty locally closed
subvarieties, called strata, such that the closure of each stratum is
a union of strata. 
If $S \in \mathcal{S}$ is a stratum, we denote by
$l_S$ the inclusion of $S$ in $X$. 
From now on, if we speak about stratifications, we always mean
algebraic Whitney stratifications.
In particular, all strata are nonsingular.
We assume in the following that all strata are irreducible varieties.
The (complex) dimension of a stratum $S$ is denoted by $d_S$.
A \define{cell-stratification} is a stratification such that each
stratum $S$ is isomorphic to an affine linear space,
so $S \cong \DC^{d_S}$. 

A sheaf $F \in \Sh(X)$ is called 
\define{smooth (along a stratification $\mathcal{S}$)}
or \define{$\mathcal{S}$-con\-struct\-ible},
if $l_S^*(F)$ is a
local system on $S$, for all $S \in \mathcal{S}$.
Let $\Sh(X, \mathcal{S}) \subset \Sh(X)$ be the full subcategory of
such sheaves.
An object $F$ of $\dcb(X)$ is called
\define{smooth (along $\mathcal{S}$)} or
\define{$\mathcal{S}$-constructible}, if all $\Ho^i(F)$ are  
in $\Sh(X, \mathcal{S})$. 

Let $(X, \mathcal{S})$ be a stratified variety.
The full subcategory $\dcb(X,\mathcal{S}) \subset \dcb(X)$ of
$\mathcal{S}$-constructible objects is a triangulated subcategory and
closed under taking direct summands.
Middle perversity defines perverse t-structures on 
$\dcb(X, \mathcal{S})$ and $\dbcc(X)$, see \cite[2.1, 2.2]{BBD}.
Their hearts $\Perv(X, \mathcal{S})$
and $\Perv(X)$ are the categories of smooth perverse sheaves and of
perverse sheaves respectively.
We have $\Perv(X, \mathcal{S})=\Perv(X)\cap \dcb(X, \mathcal{S})$.
Since any object of $\dcb(X, \mathcal{S})$ has perverse cohomology
in finitely many degrees only (\cite[2.1.2.1]{BBD}), 
perverse truncation shows the non-trivial inclusion in
\begin{equation}
  \label{eq:derbxs-praeverd-verd}
  \dbc(X, \mathcal{S}) 
  = \PraeVerd(\Perv(X, \mathcal{S}), \dbc(X)).
\end{equation}

There is a triangulated equivalence of categories (see
\cite{Beilinson, BBD})
\begin{equation}
  \label{eq:beilinsonequi}
  \real=\real_X: \dbc(\Perv(X)) \sira \dbcc(X).
\end{equation}
We denote this functor often by $A \mapsto \ul{A} := \real(A)$.
In particular,
\begin{equation}
  \label{eq:realextiso}
  \real:\Ext^n_{\Perv(X)}(A,B) 
  \sira \Ext^n_{\Sh(X)}(\ul A, \ul B) 
\end{equation}
is an isomorphism for all $A$, $B \in \dcb(\Perv(X))$.
The corresponding statement for sheaves that are smooth along a fixed
stratification $\mathcal{S}$ is usually false. 

If $\mathcal{S}$ is a cell-stratification and $S \in \mathcal{S}$ a stratum, we
define $\Delta_S=l_{S*}([d_S]\ul{S})$. Since $l_S$ is affine, $\Delta_S$
belongs to $\Perv(X,\mathcal{S})$. The objects isomorphic to some
$\Delta_S$ are called \define{standard objects}.

\begin{theorem}[{\cite[3.2, 3.3]{BGS}}]
  \label{t:bgsequi}
  Let $(X, \mathcal{S})$ be a cell-stratified variety.
  Then the category $\Perv(X,\mathcal{S})$ is artinian and has enough projective and injective
  objects. Each projective object has a finite filtration with standard subquotients.
  Each object has a projective resolution of finite length. 
  There is a triangulated equivalence
  \begin{equation}
    \label{eq:real-X-S}
    \real=\real_{X,\mathcal{S}}: \dbc(\Perv(X,\mathcal{S})) \sira \dbc(X, \mathcal{S});
  \end{equation}
  (this functor is constructed in \cite[3.1]{BBD}) we denote it by $A \mapsto \ul{A}$.
\end{theorem}

\subsection{Mixed Hodge Structures}
\label{sec:MHS}

The following definitions and results are
taken 
from \cite{HodgeII, DeligneSHMR,DeligneMilneTannakian}.

A (real) mixed Hodge structure $M$ consists of
\begin{enumerate}
\item a real vector space $M_\DR$ of finite dimension,
\item a finite increasing filtration $W$ on $M_\DR$, called 
  weight filtration,
\item a finite decreasing filtration $F$ on the complexification
  $M_\DC=\DC\otimes_\DR M_\DR$, called Hodge filtration,
\end{enumerate}
such that the filtration $W_\DC$, obtained by extension of scalars,
the filtration $F$ and its complex conjugate filtration $\ol{F}$
form a system of three opposed filtrations on $M_\DC$,
i.\,e.\ $\gr^p_F\gr^q_{\ol{F}}\gr^{W_{\DC}}_n (M_{\DC})=0$ if $n \not=
p+q$.
A morphism $f:M \ra N$ of mixed Hodge structures is an
$\DR$-linear map $f_\DR:M_\DR \ra N_\DR$ that is compatible with the
weight filtrations and whose complexification $f_\DC$ is compatible
with the Hodge filtration.

A mixed Hodge structure $M$ has weights $\leq n$ (resp.\
$\geq n$), if $\gr^W_{j,\DR} (M) := \gr^W_j (M_{\DR})=W_j M_\DR /W_{j-1} M_\DR= 0$ for $j > n$ (resp.\ $j<n$).
It is pure of weight $n$, if it is of weight $\leq n$ and of weight $\geq n$.

Let $\DR(n)$ be the Tate structure of weight $-2n$. It is a pure Hodge
structure of weight $-2n$, with $\DR(n)_\DR = (2 \pi i)^n\DR
\subset \DC = \DR(n)_\DC$.

The category $\MHS$ of mixed Hodge structures is a rigid abelian
$\DR$-linear tensor category. It admits the fiber functor ``underlying
vector space'' $\omega_0:\MHS \ra \modover{\DR}$ to the category of
finite dimensional real vector spaces and is hence neutral tannakian.
A mixed Hodge structure $M$ is
polarizable, if each graded piece $\gr^W_n(M)$ is a polarizable Hodge
structure (\cite[2.1.16]{HodgeII}). The polarizable mixed Hodge 
structures 
are a rigid tensor subcategory of $\MHS$.

The functor $\gr^W_{\DR}:\MHS \ra \gmodover{\DR}$, $M \mapsto \bigoplus_{n \in
  \DZ} \gr^W_{n, \DR}(M)$, is an exact faithful $\DR$-linear tensor
functor 
to the category of
finite dimensional graded real vector spaces.
We denote the composition of $\gr^W_\DR$ with the functor ``underlying
vector space'' $\eta:\gmodover{\DR} \ra \modover{\DR}$ by $\omega_W$. This
functor 
$\omega_W:\MHS \ra \modover{\DR}$ is a fiber functor and there is an
isomorphism of fiber functors (\cite[p.~513]{DeligneSHMR})
\begin{equation}
  \label{eq:ffiso}
  a: \omega_0 \silongra \omega_W.
\end{equation}

\subsection{Mixed Hodge Modules}
\label{sec:MHM}

We denote by $\MHM(X)$ the abelian category of mixed Hodge modules
(over $\DR$) on a complex variety $X$ (see
\cite{SaitoOntheTheoryofMHMs, SaitoIntrotoMHMs,BGS}). Instead of mixed
Hodge module we also say Hodge sheaf.
There is a faithful and exact functor $\rat: \MHM(X) \ra \Perv(X)$.
It induces a triangulated functor $\rat: \dbc(\MHM(X))\ra
\dbc(\Perv(X))$. 
Objects and morphisms in $\MHM(X)$ or in $\dbc(\MHM(X))$ are sometimes
denoted by a letter with a tilde, and omission of the tilde means
application of $\rat$, 
e.\,g.\ $\tilde M \mapsto M=\rat(\tilde M)$.

There are functors $\sHom$, $\otimes$ and
Verdier duality
$\DD$.
For $f:X \ra Y$ a morphism of
complex varieties, we have
functors $f^*$, $f_*$, $f_!$, $f^!$ relating
$\dbc(\MHM(X))$ and $\dbc(\MHM(Y))$.
We have the usual adjunctions $(f^*,f_*)$ and
$(f_!, f^!)$ between these functors, and $\DD f_* = f_! \DD$, $\DD f^*
= f^! \DD$ and $\DD^2=\id$. 
All these functors ``commute'' with the composition
\begin{equation*}
  \rera := \real \comp \rat : \dbc(\MHM(X)) \ra \dbc(\Perv(X))\sira
\dbcc(X),
\end{equation*}
where $\real$ is the equivalence \eqref{eq:beilinsonequi}.

The Hodge sheaves on the point $\point$ are the
polarizable mixed Hodge structures (\cite[1.4]{SaitoIntrotoMHMs}). 
Each Tate structure $\DR(n)$ is in $\MHM(\point)$.

Each Hodge sheaf $M \in \MHM(X)$ has a finite
increasing filtration $W$ in $\MHM(X)$ called weight filtration. This filtration is
functorial, and $M
\mapsto \gr^W_n(M)$ is an exact functor (\cite[1.5]{SaitoIntrotoMHMs}).
A Hodge sheaf $M$ has weights $\leq n$ (resp.\
$\geq n$), if $\gr^W_j(M) = 0$ for $j > n$ (resp.\ $j<n$).
More generally, a complex of Hodge sheaves $M$ has weights $\leq n$ (resp.\
$\geq n$), if each $\Ho^i(M)$ has weights $\leq n+i$ (resp.\ $\geq n+i$).
It is called pure of weight $n$, if it has weights $\leq n$ and $\geq n$.

We give some properties of mixed Hodge modules.

\begin{enumerate}[label={(M\arabic*)}]
\item\label{enum:weights-f-upper-star} If $M \in \dbc(\MHM(X))$ is of
  weight $\leq w$ (resp.\ $\geq w$), so are $f_!M$, $f^*M$ (resp.\
  $f_*M$, $f^!M$) (\cite[1.7]{SaitoIntrotoMHMs}).
\item\label{enum:weight-verdier} 
  $M$ is of weight $\leq w$ if and only
  if $\DD M$ is of weight $\geq -w$.
\item\label{enum:GrW-semisimple} For any $M \in \MHM(X)$, every
  $\gr^W_n(M)$ is a semisimple object of $\MHM(X)$
  (\cite[1.9]{SaitoIntrotoMHMs}).
\item\label{enum:pure-cplxs-are-stalk} If $M \in \dbc(\MHM(X))$ is
  pure of weight $n$, we have a noncanonical isomorphism $M \cong
  \bigoplus_{j \in \DZ}[-j]\Ho^j(M)$ (\cite[1.11]{SaitoIntrotoMHMs}).
\end{enumerate}
In the following, $f:X\ra Y$ is a morphism of complex varieties, $M$,
$N$, $A$,
$B$, $C$, $D$ are objects of $\dbc(\MHM(X))$ or $\dbc(\MHM(Y))$, and
$c:X \ra \point$ is the constant map.
\begin{enumerate}[label={(M\arabic*)}, resume]
\item\label{enum:f-upper-star-tensor}
  We have $f^*(A \otimes B) = f^*A \otimes f^*B$. 
\item \label{enum:tate-def} The Tate twist $M(n)$ of $M$ is defined by
  $M(n)=M\otimes c^*(\DR(n))$ (\cite[1.15]{SaitoIntrotoMHMs}),
  satisfies $M(0)=M$ and 
  commutes with all functors $f^*$, $f^!$, $f_*$, $f_!$. 
\item \label{enum:weight-under-shifts} 
  If $M$ is of weight $\leq w$,
  then $M(n)$ is of weight $\leq w - 2n$, and $[n]M$ is of weight $\leq
  w+n$. The same statement with $\leq$ replaced by $\geq$.
\item\label{enum:composition-and-tensor}
  The adjunction $(? \otimes B, \sHom(B, ?))$
  (\cite[2.9]{SaitoExtensionofMHMs})
  yields the composition morphism
  \begin{equation*}
    \sHom(B,C) 
    \otimes 
    \sHom(A,B) 
    \ra \sHom(A, C)
  \end{equation*}
  and in combination with the symmetry of the tensor product a morphism
  \begin{equation*}
    \sHom(A,B) \otimes \sHom(C,D) \ra \sHom(A\otimes C, B
    \otimes D).
  \end{equation*}
\item\label{enum:sheafhom-via-verdier}
  We have 
  (\cite[2.9.3]{SaitoExtensionofMHMs})
   \begin{equation*}
     \sHom(A,B) = \DD(A \otimes \DD B).
   \end{equation*}
\item\label{enum:f-upper-shriek-hom} 
  From \ref{enum:sheafhom-via-verdier} and
  \ref{enum:f-upper-star-tensor} we get
  \begin{equation*}
    f^!\sHom(A,B) 
    = \DD (f^*A \otimes \DD f^!B)
    = \sHom(f^*A, f^! B).
  \end{equation*}
\item \label{enum:smooth-iso} 
  If $f$ is smooth of relative (complex) dimension $n$, we have
   \begin{equation*}
     [2n]f^*(M)(n)=f^!(M).
   \end{equation*}
\end{enumerate}

Let $(X,\mathcal{S})$ be a stratified variety and
$\MHM(X, \mathcal{S})$
the full
abelian subcategory of $\MHM(X)$ consisting of Hodge sheaves $M$ satisfying
$\rat(M) \in \Perv(X, \mathcal{S})$. 
We denote by $\dcb(\MHM(X), \mathcal{S})$ the full subcategory of
$\dcb(\MHM(X))$ consisting of complexes $M$ satisfying $\rera(M) \in
\dcb(X, \mathcal{S})$ (or, equivalently $\Ho^i(M) \in \MHM(X, \mathcal{S})$, for all $i \in \DZ$). 
Objects of $\MHM(X, \mathcal{S})$ and $\dcb(\MHM(X), \mathcal{S})$ are
called smooth (along $\mathcal{S})$.

\begin{proposition}\label{p:hodge-on-Cn}
  Let $S=\DC^n$ for some $n \in \DN$ and $c:S \ra \point$ the constant map.
  If $M \in \MHM(S, \{S\})$ is a
  pure Hodge sheaf of weight $w$ and
  smooth along the trivial stratification, there is a pure
  Hodge structure $E \in \MHM(\point)$ of weight $w-n$ such that $M
  \cong [n]c^*(E)$.
\end{proposition}

\begin{proof}
By \cite[2.2
Theorem]{SaitoIntrotoMHMs}, $M$ corresponds to a polarizable variation
$V$ of Hodge structure of weight $w-n$ on $S=\DC^n$. The fiber $V_0$
of $V$ at $0 \in \DC^n$ is a polarizable Hodge structure of weight
$w-n$. We denote its constant extension to $\DC^n$ by $\ul{V_0}$.
Obviously, there is an isomorphism $V \sira \ul{V_0}$ of the
underlying local systems that respects the Hodge filtration at $0 \in
\DC^n$. By the Rigidity Theorem (\cite[7.24]{Schmid}, see also
\cite[13.1.9, 13.1.10]{CMP}), this isomorphism is an isomorphism of
polarizable variations of Hodge structures of weight $w-n$.  We obtain
$M \cong [n]c^*(V_0)$, where we now consider $V_0$ as a polarizable
Hodge structure of weight $w-n$ on $\point$, in particular as an
element of $\MHM(\point)$.
\end{proof}

If $Y$ is a variety, we define $\tilde{\ul{Y}} = c^*(\DR(0)) \in
\dbc(\MHM(Y))$, so $\rera(\tilde{\ul{Y}})=\ul{Y}$.
Let $X$ be an irreducible variety of dimension $d_X$ and $j: U \ra X$ the inclusion of a
nonsingular affine open dense 
subset. The intersection cohomology complexes of $X$ are defined by (\cite[1.13]{SaitoIntrotoMHMs})
\begin{align*}
  \IC(X) & := \bild(j_! ([d_X]\ul U) \ra j_*  ([d_X]\ul U)) \in
  \Perv(X) \text{ and}\\
  \tilde{\IC}(X) & := \bild(j_! ([d_x]\tilde{\ul U}) \ra j_* ([d_x]\tilde{\ul U})) \in \MHM(X).
\end{align*}
This definition does not depend on the choice of $U$,
$\tilde{\IC}(X)$ is simple and pure of weight $d_X:= \dim_\DC X$ and
satisfies $\rat(\tilde{\IC}(X))=\IC(X)$.

If $l_{\ol S}:\ol S \ra X$ is the inclusion of the closure of a stratum
$S$ in a stratified variety $(X, \mathcal{S})$, we denote $l_{\ol{S}*}
(\IC(\ol{S}))$ by 
$\IC_S$, and similarly for $\tilde{\IC}_S$. 
These
objects are smooth, $\tilde{\IC}_S$ is simple and pure of weight 
$d_S$, and we have $\rat(\tilde{\IC}_S)=\IC_S$.
If $\mathcal{S}$ is a cell-stratification, the 
$(\IC_S)_{S \in \mathcal{S}}$ are precisely the simple
objects of $\Perv(X, \mathcal{S})$.

In the introduction, we wrote $\IC(S)$ and $\tilde{\IC}(S)$ instead of
$\IC_S$ and $\tilde{\IC}_S$. We will use this notation later on again.

\subsection{Construction of Epimorphisms from Projective Objects}
\label{sec:constructingepis}

In this subsection we describe an algorithm for
constructing
an epimorphism from a projective object onto a given object. This algorithm
will be used in subsection \ref{sec:perverseprojective} in order to
show that there are enough perverse-projective mixed Hodge modules.

Let $\mathcal{A}$ be an artinian $k$-category, where $k$ is a field.
We write $\Hom$, $\End$, $\Ext$, $\otimes$ instead of $\Hom_\mathcal{A}$,
$\End_\mathcal{A}$, $\Ext_\mathcal{A}$, $\otimes_k$, respectively.
We make the following assumptions:
\begin{enumerate}[label={(E\arabic*)}]
\item\label{en:endo-simple} $\End(L)=k$ for all simple objects 
  $L$ in $\mathcal{A}$.
\item\label{en:enough-projectives} There are enough projective objects in $\mathcal{A}$.
\end{enumerate}
Note that \ref{en:endo-simple} implies that $\Hom(M,N)$ is finite dimensional,
for all $M$, $N \in \mathcal{A}$. Then \ref{en:enough-projectives} shows that 
$\Ext^1(M, N)$ is finite dimensional, for all $M$, $N \in \mathcal{A}$. 

The following algorithm keeps extending simple objects to a given
object until this is no longer possible. In doing this, only
non-trivial extensions are used.

\begin{enumerate}
\item[Step 1:] Take an object $A \in \mathcal{A}$ as input datum.
\item[Step 2:] Set $i= 0$ and $A_0=A$.
\item[Step 3:] While there is a simple object $L \in \mathcal{A}$ with
  $\Ext^1(A_i,L)\not=0$ 
  \begin{enumerate}
  \item[Step 3.1:] Take a simple object $L \in \mathcal{A}$ with $E:=\Ext^1( A_i,L)\not=0$. 
  \item[Step 3.2:] The element $\id \in E^* \otimes E= \Ext^1(A_i,E^*\otimes L)$ 
    gives rise\footnote{Perhaps we should explain what we mean by the object $E^* \otimes L$.  Let
$\Nat$ be the following category: Its objects are the natural numbers
$\DN$; if $m$, $n$ are objects of $\Nat$, we define $\Hom_{\Nat}(m,n)$
to be the set of $n\times m$-matrices over $k$; composition 
is matrix multiplication.  We fix an equivalence of
categories $\phi: \modover{k} \sira \Nat$ between the category
of finite dimensional vector spaces over $k$ and $\Nat$.
There is an obvious functor $(? \otimes ?): \Nat \times
\mathcal{A} \ra \mathcal{A}$, 
$(n, M) \mapsto n \otimes M := M^{\oplus n}$. If
$V$ is a finite dimensional vector space and $M$ is in $\mathcal{A}$, we
define $V \otimes M:= \phi(V) \otimes M$.}
    to an extension $E^* \otimes L \hra A_{i+1} \sra A_i$.
  \item[Step 3.3:] Increase $i$ by $1$.
  \end{enumerate}
\item[Step 4:] Define $Q=A_i$ and return the epimorphism $Q=A_i \sra A_0=A$.
\end{enumerate}

\begin{proposition}\label{p:algoterminatesprojective}
  Given any $A \in \mathcal{A}$, the above algorithm terminates after
  finitely many steps and returns an epimorphism $Q \ra A$ from a
  projective object $Q$.
\end{proposition}
\begin{proof}
  We denote the length of an object $X \in  \mathcal{A}$ by $\length(X)$. 
  Assume that our algorithm does not stop. Then it constructs a 
  sequence $\dots \sra A_2 \sra A_1 \sra A_0$ of objects and
  epimorphisms with $\length(A_0) < \length(A_1) < \length(A_2) < \dots$. 
  By \ref{en:enough-projectives}, there are a projective object $P$ and an
  epimorphism $\pi_0:P \ra A_0=A$.  
  Since $A_1 \ra A_0$ is epimorphic, there is a
  morphism $\pi_1:P \ra A_1$ lifting $\pi_0:P \ra A_0$. Proceeding in this
  manner we obtain liftings $\pi_i:P \ra A_i$ of $\pi_{i-1}$ for all $i > 0$.
  Now Lemma \ref{l:length-bounded} below shows that all these $\pi_i$ are
  epimorphisms. In particular, we get the contradiction $\length(A_i) \leq \length(P)$ for all $i$.
  So our algorithm stops. The returned object $Q$ 
  is projective since $\Ext^1(Q,L)=0$ for all simple objects $L \in
  \mathcal{A}$.
\end{proof}

\begin{lemma}\label{l:length-bounded}
  Let $\pi:P \sra A$ be an epimorphism from an object $P$ (not
  necessarily projective) onto
  $A$. Let $L$ be a simple object, $E=\Ext^1(A,L)$ and
  \begin{equation}
    \label{eq:lemma-extension}
    0 \ra E^*\otimes L \xra{i} M \xra{c} A \ra 0
  \end{equation} 
  the extension defined by $\id \in E^* \otimes E= \Ext^1(A,
  E^*\otimes L)$. Let $\tilde \pi:P \ra M$ be a morphism such that
  $c \comp \tilde{\pi} =\pi$. 
  Then $\tilde \pi$ is an epimorphism.
\end{lemma}
\begin{proof}
  We have to show that 
  \begin{enumerate}
  \item[$(*)_N$] The map $\tilde{\pi}^*:\Hom(M,N) \ra \Hom(P,N)$ is injective.
  \end{enumerate}
  holds for all objects $N$.
  If $(N', N, N'')$ is a short exact sequence and
  $(*)_{N'}$ and $(*)_{N''}$ hold, then it is easy to see that
  $(*)_N$ is satisfied. So it is enough to prove $(*)_N$
  for simple objects $N$.

  Let $N$ be simple. By applying $\Hom(?, N)$ to the exact
  sequence \eqref{eq:lemma-extension}, we obtain the exact sequence
  \[0 \ra \Hom(A,N) \xra{c^*} \Hom(M,N) \xra{i^*} \Hom(E^*\otimes L,N)
  \xra{\delta} \Ext^1(A,N).\]
  Our claim is that $c^*$ is bijective.
  For $N\not\cong L$ this is clear, since $\Hom(E^*\otimes L,N)$
  vanishes.
  If $N=L$ there is an obvious map $\can:E \ra {\Hom(E^* \otimes L, L)}$
  such that $\delta \comp \can =\id$.
  Since $\Hom(L,L)=k$ by \ref{en:endo-simple}, $\can$ and hence $\delta$
  are isomorphisms. This shows that $c^*$ is an isomorphism if $N=L$
  or $N\cong L$. 
  
  We now apply $\Hom(?, N)$ to 
  $c \comp \tilde{\pi} =\pi$ and obtain $\tilde{\pi}^* \comp c^* = \pi^*$.
  Since $c^*$ is bijective and $\pi^*$ is injective, $\tilde{\pi}^*$
  is injective and $(*)_N$ is true.
\end{proof}

\subsection{Existence of Enough Perverse-Projective Hodge Sheaves}
\label{sec:perverseprojective}

Let $(X, \mathcal{S})$ be a cell-stratified complex variety.
A smooth Hodge sheaf $\tilde P \in \MHM(X, \mathcal{S})$ is called
\define{perverse-projective}, if the underlying perverse sheaf
$P=\rat(\tilde P)$ is a projective object of $\Perv(X, \mathcal{S})$.
A complex in $\MHM(X, \mathcal{S})$ is called
\define{perverse-projective}, if all its components are perverse-projective.

\begin{proposition}
  \label{p:enoughperverseprojective}
  If $(X, \mathcal{S})$ is a cell-stratified complex variety,
  there are enough perverse-projective objects in $\MHM(X,
  \mathcal{S})$, i.\,e.\ 
  for every smooth Hodge sheaf $\tilde A \in \MHM(X, \mathcal{S})$,
  there is a
  perverse-projective smooth Hodge sheaf $\tilde P \in \MHM(X, \mathcal{S})$ and an
  epimorphism $\tilde P \sra \tilde A$.
\end{proposition}

\begin{proof}
  Postponed to the end of this subsection.
\end{proof}

\begin{corollary}\label{cd:ppreso}
  Every smooth Hodge sheaf $\tilde A \in \MHM(X, \mathcal{S})$ has a
  \define{perverse-projective resolution} $\tilde P \ra \tilde A$, i.\,e.\ 
  there is a perverse-projective complex $\tilde {P} = (\tilde{P}^n,
  d^n)$ in $\MHM(X, \mathcal{S})$ with $\tilde{P}^n = 0$
  for $n > 0$ and a quasi-isomorphism
  $\tilde P \ra \tilde A$
  in  $\Ket(\MHM(X, \mathcal{S}))$. 
  Moreover, we can assume that this resolution is \define{of finite length}, i.\,e.\ $\tilde{P}^n=0$ for $n \ll 0$.
\end{corollary}
\begin{proof}
  The first statement is obvious from the proposition, the second one
  follows from Theorem \ref{t:bgsequi}.
\end{proof}

  If $\tilde{M}$, $\tilde{N}\in \MHM(X)$ are Hodge sheaves 
  on $X$ with underlying perverse sheaves $M$ and $N$,
  there is a (polarizable) mixed Hodge structure on all
  $\Ext^i_{\Perv(X)}(M,N)$, defined as follows.
  Let $c: X \ra \point$ be the constant map and $\ul{M}=\real(M)$,
  $\ul{N}=\real(N)$. On a point, perverse cohomology  
  and ordinary cohomology coincide, and we get
  \begin{equation*}
    \rera(\Ho^i c_*\sHom(\tilde{M},
    \tilde{N})) 
    = \Ho^i c_* \sHom(\ul{M},\ul{N})
    = \Ext_{\Sh(X)}^i(\ul{M},\ul{N}).
  \end{equation*}
  Thus, we obtain 
  a natural mixed Hodge structure on
  $\Ext_{\Sh(X)}^i(\ul{M},\ul{N})$.
  We transfer this structure to
  $\Ext^i_{\Perv(X)}(M,N)$ (using \eqref{eq:realextiso}) and denote it
  by $\Ext^i_{\Perv(X)}(\tilde M, \tilde N)$
  (and by $\Hom_{\Perv(X)}(\tilde M, \tilde N)$ for $i=0$).
  If $\tilde{M}$, $\tilde{N}$ are smooth along our
  cell-stratification, the analogous argument 
  equips $\Ext^i_{\Perv(X,\mathcal{S})}(M,N)$
  with a mixed Hodge structure
  $\Ext^i_{\Perv(X,\mathcal{S})}(\tilde{M}, \tilde{N})$.
  
  \begin{remark}\label{rem:extnice}
    Our construction defines bifunctors
    \[\Ext_{\Perv(X)}^i(?, ?):\MHM(X)^\opposite \times \MHM(X) \ra \MHS.\]
    The usual long exact $\Ext$-sequences in both
    variables underlie exact sequences of mixed Hodge structures.
    Furthermore, if $\tilde A$, $\tilde B$ and $\tilde C$ are Hodge
    sheaves, composition defines a morphism of mixed Hodge structures
    \begin{equation}
      \label{eq:extcompo}
      \Ext_{\Perv(X)}^i(\tilde B, \tilde C)
      \otimes
      \Ext_{\Perv(X)}^j(\tilde A, \tilde B) 
      \ra
      \Ext^{i+j}_{\Perv(X)}(\tilde A, \tilde C).
    \end{equation}
    This can be seen as follows. If $F$, $G$ are in $\dbc(\MHM(X))$ 
    there is a natural  morphism 
    $c_*F \otimes
    c_*G \ra c_*(F \otimes G)$. We compose this morphism for
    $F=\sHom(\tilde B, \tilde C)$ and $G=\sHom(\tilde A,
    \tilde B)$ with $c_*$ of the morphism
    $ \sHom(\tilde B, \tilde C) \otimes \sHom(\tilde A, \tilde B)
    \ra \sHom(\tilde A, \tilde C)$
    (cf.~\ref{enum:composition-and-tensor})
    and get a morphism
    \begin{equation*}
      c_*\sHom(\tilde B, \tilde C) 
      \otimes 
      c_*\sHom(\tilde A, \tilde B) 
      \ra c_*\sHom(\tilde A, \tilde C).
    \end{equation*}
    Now we take the $(i+j)$-th cohomology and use the obvious morphism
    of mixed Hodge structures 
    $\Ho^i (c_*F) \otimes \Ho^j (c_*G) \ra \Ho^{i+j}(c_*F \otimes c_*G)$
    in order to get morphism \eqref{eq:extcompo}.
    The analogous remarks are valid for $\Ext_{\Perv(X, \mathcal{S})}^i$.
  \end{remark}

\begin{lemma}\label{l:isorat}
  Let $F: \mathcal{A} \ra \mathcal{B}$ be an exact faithful 
  functor between abelian categories, $F:
  \dc(\mathcal{A}) \ra \dc(\mathcal{B})$ its derived functor, and
  $f$ a morphism in $\dc(\mathcal{A})$. 
  Then $f$ is an isomorphism if and only if $F(f)$ is an isomorphism.
\end{lemma}

\begin{proof}
  This is an easy exercise.
\end{proof}

\begin{lemma}
Let ${A}$, ${B} \in \MHM(X)$ be Hodge sheaves on
$X$, let ${M} \in
\MHM(\point)$ be a polarizable mixed Hodge
structure and $c:X \ra \point$ the constant map.
Then there are natural isomorphisms 
\begin{align}
  \label{eq:isodirectimagetensor}
  M\otimes c_*A & \sira c_*(c^*M \otimes A)\\
  \label{eq:homandconstcommute}
  c^* {M} \otimes \sHom({A},{B})
  & \sira \sHom({A},c^* {M} \otimes{B}). 
\end{align}
in $\dbc(\MHM(\point))$ and $\dbc(\MHM(X))$ respectively.
\end{lemma}

\begin{proof}
The morphism \eqref{eq:isodirectimagetensor} is the image of the
identity morphism of $c^*M \otimes A$ under the chain of obvious morphisms
\begin{align*}
\Hom(c^*M \otimes A,c^*M \otimes A)
 & \ra \Hom(c^*M \otimes c^*c_*A,c^*M \otimes A)\\
 & = \Hom(c^*(M \otimes c_*A),c^*M \otimes A)\\
 & = \Hom(M \otimes c_*A, c_*(c^*M \otimes A)),
\end{align*}
where $\Hom=\Hom_{\dcb(\MHM)}$.
Since $v(M)$ is a finite dimensional vector space, 
$v \eqref{eq:isodirectimagetensor}$ is an
isomorphism, and Lemma \ref{l:isorat}, applied to $\rat$, 
shows that \eqref{eq:isodirectimagetensor} is an isomorphism.

The morphism \eqref{eq:homandconstcommute} comes from the
identifications (\ref{enum:tate-def},
\ref{enum:sheafhom-via-verdier}) 
\[c^*M = \DD (c^*\DR(0) \otimes \DD c^*M) = \sHom(c^*\DR(0),c^*M)\]
and the morphism (\ref{enum:composition-and-tensor}, \ref{enum:tate-def}) 
\begin{align*}
\sHom(c^*\DR(0),c^*M) \otimes \sHom(A,B) & \ra
\sHom(c^*\DR(0)\otimes A, c^*M \otimes B)\\ 
& =\sHom(A,c^*M \otimes B).
\end{align*}
All these morphisms
are mapped to isomorphisms by the functor $v$, so
\eqref{eq:homandconstcommute} is an isomorphism by Lemma
\ref{l:isorat} and equivalence \eqref{eq:beilinsonequi}.
\end{proof}

\begin{lemma}\label{l:otimestildeHompervers}
  Let $\tilde{A}$, $\tilde{B} \in \MHM(X, \mathcal{S})$, $\tilde{M} \in
  \MHM(\point)$, and let $c:X \ra \point$ be the constant map.
  Then $c^*\tilde{M}\otimes \tilde{B} \in \MHM(X,
  \mathcal{S})$, and there is a natural isomorphism
  \[\tilde{M}\otimes
  {\Ext}^i_{\Perv(X, \mathcal{S})}(\tilde{A},\tilde{B})\sira
  {\Ext}^i_{\Perv(X, \mathcal{S})}(\tilde{A}, c^*
  \tilde{M}\otimes \tilde{B})\]
  of (polarizable) mixed Hodge structures, for all $i \in \DZ$.
\end{lemma}
\begin{proof}
  The first statement is obvious. Isomorphisms \eqref{eq:isodirectimagetensor} and
  \eqref{eq:homandconstcommute} yield an isomorphism
  \begin{equation*}
    \tilde{M} \otimes c_* \sHom(\tilde{A},\tilde{B})
    \sira c_*\sHom(\tilde{A},c^* \tilde{M} \otimes\tilde{B}) 
  \end{equation*}
  Taking the $i$-th cohomology and using the exactness of the functor
  $(\tilde{M} \otimes ?)$ finishes the proof.
\end{proof}

\begin{proof}[Proof of Proposition \ref{p:enoughperverseprojective}]

  If $\tilde M$, $\tilde N \in \MHM(X)$ are Hodge sheaves, there is
  a short exact sequence (see \cite{SaitoExtensionofMHMs})
  \begin{equation*}
    \begin{split}
    0 \ra \Ho^1_{\MHM(\point)} (\Hom_{\Perv(X)}&(\tilde M, \tilde N))
    \ra \Ext^1_{\MHM(X)}(\tilde M, \tilde N) \\
    \ra & \Ho^0_{\MHM(\point)}(\Ext^1_{\Perv(X)}(\tilde M, \tilde N)) \ra
    0, 
    \end{split}
  \end{equation*}
  where $\Ho^i_{\MHM(\point)}$ is the absolute Hodge cohomology
  functor: 
  For $A \in \MHM(\point)$,
  it is defined by 
  $\Ho^i_{\MHM(\point)}(A):=\Ext^i_{\MHM(\point)}(\DR(0), A)$. 
  The categories $\MHM(X, \mathcal{S})$
  and $\Perv(X, \mathcal{S})$ are closed under extensions in
  $\dbc(\MHM(X))$ and $\dbc(X)$
  (\cite[1.3.6, 3.1.17]{BBD}). Thus, for smooth $\tilde M$, $\tilde N \in
  \MHM(X, \mathcal{S})$, there is a short
  exact sequence 
  \begin{equation}
    \label{eq:hodgestratiextension}
    \begin{split}
    0 \ra \Ho^1_{\MHM(\point)} (\Hom_{\Perv(X, \mathcal{S})}&(\tilde M, \tilde N))
    \ra \Ext^1_{\MHM(X,\mathcal{S})}(\tilde M, \tilde N) \\
    \ra & \Ho^0_{\MHM(\point)}(\Ext^1_{\Perv(X, \mathcal{S})}(\tilde M, \tilde N)) \ra
    0.
    \end{split}
  \end{equation}
  
  Let $\tilde M$, $\tilde N \in \MHM(X, \mathcal{S})$ and consider
  the polarizable mixed Hodge structure $\tilde E =
  \Ext^1_{\Perv(X,\mathcal{S})}(\tilde M, \tilde N)$. The map
  $\can:\DR(0) \ra \tilde{E}^* \otimes \tilde E$, $1 \mapsto
  \id_{\tilde E}$, is a
  morphism of polarizable mixed Hodge structures, i.\,e.\ an element
  $\can \in \Ho^0_{\MHM(\point)}(\tilde{E}^* \otimes \tilde E)$. 
  Lemma \ref{l:otimestildeHompervers} yields an isomorphism
  \begin{equation*}
    \tilde{E}^* \otimes \tilde E=\tilde{E}^*\otimes {\Ext}^1_{\Perv(X, \mathcal{S})}(\tilde{M},\tilde{N})\sira
    {\Ext}^1_{\Perv(X, \mathcal{S})}(\tilde{M}, c^*(\tilde{E}^*)\otimes \tilde{N})
  \end{equation*}
  of polarizable mixed Hodge structures. The exact sequence
  \eqref{eq:hodgestratiextension} shows that there is an 
  extension of smooth Hodge sheaves 
  \begin{equation*}
    c^*(\tilde {E}^*) \otimes \tilde{N} \hra \tilde{K} \sra
    \tilde{M}
  \end{equation*}
  such that the underlying
  extension of perverse sheaves is given by the element $\id_E \in E^*
  \otimes E \sira \Ext^1_{\Perv(X, \mathcal{S})}(M, c^*(E^*) \otimes N)$.
    
  We now use the following algorithm in order to prove our
  proposition. 
    
  \begin{enumerate}
  \item[Step 1:] Take an object $\tilde A \in \MHM(X, \mathcal{S})$ as input datum.
  \item[Step 2:] Set $i= 0$ and $\tilde{A}_0=\tilde A$.
  \item[Step 3:] While there is a stratum $S \in \mathcal{S}$ with
    $\Ext^1_{\Perv(X,\mathcal{S})}(\tilde{A}_i,\tilde{\IC}_S)\not=0$
    \begin{enumerate}
    \item[Step 3.1:] Take a stratum $S \in \mathcal{S}$ with
      $\tilde{E}=\Ext^1_{\Perv(X,\mathcal{S})}(\tilde{A}_i,\tilde{\IC}_S)\not=0$.
    \item[Step 3.2:] Choose, as explained above, an extension
      $c^*(\tilde {E}^*) \otimes \tilde{\IC}_S \hra \tilde{A}_{i+1} \sra
      \tilde{A}_i$ of smooth Hodge sheaves such that the underlying
      extension of perverse sheaves is given by $\id \in E^* \otimes
      E$.
    \item[Step 3.3:] Increase $i$ by $1$.
    \end{enumerate}
  \item[Step 4:] Define $\tilde P=\tilde{A}_i$ and return the
    epimorphism $\tilde{P}=\tilde{A}_i \sra \tilde{A}_0=\tilde{A}$
    of smooth Hodge sheaves.
  \end{enumerate}
  The underlying algorithm is the algorithm from subsection
  \ref{sec:constructingepis} for $\mathcal{A}=\Perv(X, \mathcal{S})$.
  Since $\mathcal{S}$ is a cell-stratification, this choice of
  $\mathcal{A}$ is justified by Theorem \ref{t:bgsequi}. 
  Thus, Proposition \ref{p:algoterminatesprojective} shows that our
  algorithm terminates and returns an 
  epimorphism $\tilde P \sra 
  \tilde A$ from a perverse-projective smooth Hodge sheaf
  $\tilde P$.
\end{proof}

\subsection{Comparison of  Mixed Hodge Structures}
\label{sec:comparing}

Let $\tilde A$, $\tilde B \in \MHM(X,\mathcal{S})$ be smooth Hodge
sheaves on a cell-stratified variety $(X, \mathcal{S})$, with
underlying smooth perverse sheaves $A$ and $B$.
As explained in subsection \ref{sec:perverseprojective}, there is a
(polarizable) mixed Hodge structure 
$\Ext^i_{\Perv(X, \mathcal{S})}(\tilde A, \tilde B)$ on 
$\Ext^i_{\Perv(X, \mathcal{S})}(A, B)$.

Now assume that $\tilde{P} \ra \tilde A$ and $\tilde{Q}
\ra \tilde B$ are perverse-projective resolutions of finite length
(cf.\ Corollary and Definition \ref{cd:ppreso}), with underlying
projective resolutions $P \ra A$ and $Q \ra B$. We apply
$\Hom_{\Perv(X,\mathcal{S})}(?, ?)$ to $\tilde{P}$ and
$\tilde{Q}$ and get a double complex 
of (polarizable) mixed Hodge structures
(see Remark \ref{rem:extnice}) with $(i,j)$-component $\Hom_{\Perv(X,
  \mathcal{S})}(\tilde{P}^{-i}, \tilde{Q}^j)$. 
We denote its simple complex by
$\complexHom_{\Perv(X,\mathcal{S})}(\tilde{P}, \tilde{Q})$
or simply by $\cHom(\tilde P, \tilde Q)$.
(We use this notation also for arbitrary complexes $\tilde P$ and
$\tilde Q$ in $\MHM(X, \mathcal{S})$.)
The underlying complex of real vector spaces is
the complex
$\cHom(P,Q)=
\cHom_{\Perv(X, \mathcal{S})}(P,Q)$ from 
subsection \ref{sec:deriv-categ-dg}.
The $n$-th cohomology $\Ho^n(\complexHom_{\Perv(X,\mathcal{S})}(\tilde{P},
\tilde{Q}))$ is a mixed Hodge structure, and its underlying vector
space is 
\begin{equation}
  \label{eq:coho-exts}
  \begin{split}
    \Ho^n(\complexHom_{\Perv(X,\mathcal{S})}(P, Q))
    = \Hom_{\Hot(\Perv(X, \mathcal{S}))}(P, [n]Q)\\
    \notag = \Ext_{\Perv(X, \mathcal{S})}^n(P, Q)
    \notag \sira 
    \Ext^n_{\Perv(X,\mathcal{S})}(P,B) \\
    \notag
    \sila \Ext^n_{\Perv(X,\mathcal{S})}(A,B).
  \end{split}
\end{equation}

\begin{proposition}\label{p:coincide}
  The (polarizable) mixed Hodge structures
  \begin{equation*}
    \Ext^n_{\Perv(X, \mathcal{S})}(\tilde A, \tilde B) \text{ and }
    \Ho^n(\complexHom_{\Perv(X,\mathcal{S})}(\tilde{P},\tilde{Q}))
  \end{equation*}
  with underlying vector space $\Ext^n_{\Perv(X, \mathcal{S})}(A,B)$
  are isomorphic.
\end{proposition}

\begin{proof}
  We write $\Hom$, $\complexHom$ and $\Ext$ instead
  of $\Hom_{\Perv(X, \mathcal{S})}$, $\complexHom_{\Perv(X,
    \mathcal{S})}$ and $\Ext_{\Perv(X, \mathcal{S})}$ respectively,
  and show the existence of isomorphisms of mixed Hodge structures
  \begin{equation}
    \label{eq:two-isos}
    \Ext^n(\tilde{A}, \tilde{B})
    \sila \Ho^n(\complexHom(\tilde{P},\tilde{B}))
    \sila \Ho^n(\complexHom(\tilde{P},\tilde{Q})).
  \end{equation}

  Let us construct the isomorphism on the left in \eqref{eq:two-isos}.
  We decompose $\tilde{P}\ra\tilde A$ into short exact sequences
  as follows. For $i \leq 0$, let $\tilde{K}^{i}$ be the image of the
  differential $\tilde{P}^{i-1} \ra \tilde{P}^i$, and define
  $\tilde{K}^1 = \tilde{A}$. For each $i \leq 0$, we get a short exact sequence
  $(\tilde{K}^{i}, \tilde{P}^{i}, \tilde{K}^{i+1})$.
  The associated long exact $\Ext$-sequence in the first variable gives an exact
  sequence of mixed Hodge modules (see Remark \ref{rem:extnice})
  \begin{align}
    \label{eq:longext}
    0 & \ra \Hom(\tilde{K}^{i+1}, \tilde{B}) \ra \Hom(\tilde{P}^{i},
    \tilde{B})) \ra \Hom(\tilde{K}^{i}, \tilde{B}) \\
    & \ra \Ext^1(\tilde{K}^{i+1}, \tilde{B}) \ra 0 \notag
  \end{align}
  and isomorphisms
  \begin{equation}
    \label{eq:higher-ext}
    \Ext^j(\tilde{K}^{i}, \tilde{B}) \sira
    \Ext^{j+1}(\tilde{K}^{i+1}, \tilde{B}) 
  \end{equation}
  for all $j \geq 1$. 
  The $0$-th cohomology of the complex $\tilde{C} :=
  \complexHom(\tilde{P},\tilde{B})$ is
  \begin{equation*}
    \Ho^0(\tilde{C}) = \Ker \big(\Hom(\tilde{P}^{0},
    \tilde{B}) \ra \Hom(\tilde{K}^{0}, \tilde{B})\big)
    \sira \Hom(\tilde{A}, \tilde{B})
  \end{equation*}
  by \eqref{eq:longext} for $i=0$.
  For $m \geq 0$ we have
  \begin{align*}
    \Ho^{m+1}(\tilde{C}) & = \Cok \big(\Hom(\tilde{P}^{-m},
    \tilde{B}) \ra \Hom(\tilde{K}^{-m}, \tilde{B})\big)\\
    & \sira \Ext^1(\tilde{K}^{-m+1}, \tilde{B})\\
    & \sira \Ext^{m+1}(\tilde{A}, \tilde{B})
  \end{align*}
  by \eqref{eq:longext} and repeated use of \eqref{eq:higher-ext}. 
  This establishes the isomorphism on the left in \eqref{eq:two-isos}.
  The isomorphism on the right 
  is a consequence of the following Lemma \ref{l:qiso-hom}
  applied to the quasi-isomorphism 
  $\tilde{Q} \ra \tilde{B}$.
\end{proof}

\begin{lemma}
  \label{l:qiso-hom}
  Let $\tilde{P}$, $\tilde{Q}$ and $\tilde{R}$ be
  complexes in $\MHM(X, \mathcal{S})$, and assume that
  $\tilde{P}$ is perverse-projective and bounded above.
  Then any quasi-isomorphism 
  $f:\tilde{Q} \ra \tilde{R}$ in $\Ket(\MHM(X, \mathcal{S}))$ induces
  a quasi-iso\-mor\-phism
  $\complexHom(\tilde{P},\tilde{Q}) \ra \complexHom(\tilde{P},\tilde{R})$.
\end{lemma}

\begin{proof}
  It suffices to prove that $\rat(f): Q \ra R$ in $\Ket(\Perv(X, \mathcal{S}))$
  induces a quasi-isomorphism
  $\complexHom({P},{Q}) \ra \complexHom({P},{R})$.
  But this follows from Remark \ref{rem:hom-p-dot-iso}.
\end{proof}

\begin{remark}\label{rem:componice}
Assume that $\tilde{P} \ra \tilde A$, 
$\tilde{Q} \ra \tilde B$ 
and $\tilde{R} \ra \tilde C$ 
are perverse-projective resolutions of finite length.
Then $\complexHom(\tilde{P}, \tilde{Q})$ and
$\complexHom(\tilde{Q}, \tilde{R})$ are complexes
of (polarizable) mixed Hodge structures. The obvious composition map 
\begin{equation*}
  \complexHom(\tilde{Q}, \tilde{R})
  \otimes
  \complexHom(\tilde{P}, \tilde{Q})
  \ra \complexHom(\tilde{P}, \tilde{R})
\end{equation*}
is a morphism of complexes of mixed Hodge structures. It induces a
morphism of mixed Hodge structures
\begin{equation*}
\Ho^i(\complexHom(\tilde{Q}, \tilde{R}))
\otimes
\Ho^j(\complexHom(\tilde{P}, \tilde{Q}))
\ra \Ho^{i+j}(\complexHom(\tilde{P}, \tilde{R})).
\end{equation*}
Under the identifications from (the proof of) Proposition
\ref{p:coincide}, this morphism corresponds to the morphism
\eqref{eq:extcompo} in Remark \ref{rem:extnice}. 
The reason for this fact is that both morphisms are just the composition of
morphisms in the derived category of perverse sheaves, if we forget
about the mixed Hodge structures. 
\end{remark}

\subsection{Local-to-Global Spectral Sequence}

If $X$ is a complex variety and $M$ a complex of sheaves 
(of Hodge sheaves respectively) on $X$,
the hypercohomology $\hyper(M) := \Ho(X;M):=\Ho(c_*M)$ is
a complex (with differential zero) of vector spaces (of mixed Hodge
structures respectively). 
Here $c:X \ra \point$ is the constant map.
If $l: Y \hra X$ is a locally closed subvariety,
the local hypercohomology of $M$ along $Y$ is
$\hyper_Y(M) := \hyper(l^!M)  = \hyper(l_*l^! M)$.

Consider now a complex variety $X$ filtered by closed subvarieties
$X=X_0 \supset X_1 \supset \dots \supset X_r = \emptyset$.
If $M\in\dbc(X)$ is a complex of sheaves, there is a local-to-global
spectral sequence with $E_1$-term
$E_1^{p,q}=\hyper^{p+q}_{X_p-X_{p+1}}(M)$ converging to
$E_\infty^{p,q}=\gr^p(\hyper^{p+q}(M))$. 
It can be constructed from an injective resolution of $M$, 
cf. \cite[3.4]{BGS}. 

Even though there are not enough injective Hodge sheaves, we shall
construct a similar spectral sequence of mixed (polarizable) Hodge
structures, if $M$ is a complex of Hodge sheaves on our filtered
variety $X$. In order to do so, we need the following technical proposition.

\begin{proposition}\label{p:baue-filtrierten-komplex}
  Let $\mathcal{A}$ be an abelian category and 
  \begin{equation*}
    \xymatrix{
      {A^0} 
      & \ar[l]_-{a_1} {A^1} 
      & \ar[l]_-{a_2} {\dots}
      & \ar[l]_-{a_{r-1}} {A^{r-1}} 
      & \ar[l]_-{a_r} {A^r=0}
    }
  \end{equation*}
  a finite sequence of objects and morphisms in
  $\dbc(\mathcal{A})$. Then there is a bounded 
  complex $K$ in $\mathcal{A}$ with a finite filtration $F$ by
  subcomplexes, $K=F^0 K \supset F^1 K \supset \dots \supset F^{r-1} K
  \supset F^r K=0$, and quasi-isomorphisms $u_p:F^p K \ra A^p$ in
  $\Ket(\mathcal{A})$ such that the diagram
  \begin{equation*}
    \xymatrix{
      {A^0} 
      & {A^1} \ar[l]_-{a_1}
      & {A^2} \ar[l]_-{a_2}
      & \ar[l]_{a_3} {\dots}
      & {A^{r-1}} \ar[l]_-{a_{r-1}}
      & {A^r} \ar[l]_-{a_r}\\
      {F^0 K} \ar[u]_-{u_0}
      & {F^1 K} \ar[l]_-{{k}_1} \ar[u]_-{u_1} 
      & {F^2 K} \ar[l]_-{{k}_2} \ar[u]_-{u_2} 
      & \ar[l]_{k_3} {\dots}
      & {F^{r-1} K} \ar[l]_-{{k}_{r-1}} \ar[u]_-{u_{r-1}} 
      & {F^{r} K} \ar[l]_-{{k}_{r}} \ar[u]_-{u_{r}} 
    }
  \end{equation*}
  commutes in $\dbc(\mathcal{A})$. Here $k_p: F^pK \hra F^{p+1}K$
  denotes the inclusion.
\end{proposition}

\begin{proof}
  This seems to be well known, cf.\ the similar statement given in
  \cite[3.1.2.7]{BBD} without proof.
  For an explicit proof see \cite[Prop.~26]{OSdiss}
\end{proof}

Now let $M \in \dbc(\MHM(X))$ be a complex of Hodge sheaves on $X$,
where $X$ is a complex variety, filtered by closed subvarieties
$X=X_0 \supset X_1 \supset \dots \supset X_r = \emptyset$.
If we let $i_p:X_p \hra X$ denote the inclusion, the adjunctions
$(i_{p*}=i_{p!},i_p^!)$ yield a sequence in $\dbc(\MHM(X))$
\begin{equation*}
  \xymatrix{
    {M=i_{0*}i_0^! M} 
    & \ar[l] {i_{1*}i_1^! M} 
    & \ar[l] {\dots} 
    & \ar[l] {i_{r-1*}i_{r-1}^! M} 
    & \ar[l] {i_{r*}i_{r}^! M=0.}
  }
\end{equation*}
We apply $c_*$ to this sequence, where $c:X \ra \point$. Proposition
\ref{p:baue-filtrierten-komplex} shows that there is a diagram 
\begin{equation*}
  \xymatrix{
    {c_*M} 
    & {c_*i_{1*}i_1^! M} \ar[l]
    & {c_*i_{2*}i_2^! M} \ar[l]
    & \ar[l] {\dots}
    & 0 \ar[l]\\
    {K=F^0 K} \ar[u]
    & {F^1 K} \ar[l]
\ar[u]
    & {F^2 K} \ar[l]
\ar[u]
    & \ar[l]
{\dots}
    & {F^{r} K = 0,} \ar[l]
\ar[u]
  }
\end{equation*}
where the lower horizontal row is a finite filtration on a complex
$K$ in $\MHS$, the vertical maps are quasi-isomorphisms in
$\Ketb(\MHS)$, and the diagram commutes in $\dbc(\MHS)$.
(We could write $\MHM(\point)$ instead of $\MHS$.)

By \cite[Proposition XX.9.3]{Lang-Algebra}, there exists a spectral
sequence $(E_r, d_r)_{r \geq 0}$ with 
$E_1^{p,q}  = \Ho^{p+q}(\gr_F^p (K))$ and 
$E_\infty^{p,q}  = \gr^p(\Ho^{p+q}(K))$.
Using standard techniques it is easy to identify the $E_1$-term
of this spectral sequence with
$\hyper^{p+q}_{X_p-X_{p+1}}(M)$
(for details see \cite[2.11]{OSdiss}). 
This proves

\begin{proposition}\label{p:hodge-lgss}
Let $X$ be a complex variety, filtered by closed subvarieties
$X=X_0 \supset X_1 \supset \dots \supset X_r = \emptyset$, and 
$M \in \dbc(\MHM(X))$ a complex of Hodge sheaves on $X$.
Then there is a spectral sequence $(E_r, d_r)_{r \geq 0}$ of
(polarizable) mixed Hodge structures with $E_1$-term  
$E_1^{p,q}=\hyper^{p+q}_{X_p-X_{p+1}}(M)$ 
that converges to $E_\infty^{p,q} =
\gr^p(\hyper^{p+q}(M))$
(where $\hyper(M)$ is filtered by the images of the obvious maps
$\hyper_{X_p}(M) \ra \hyper(M)$).
\end{proposition}

\subsection{Purity}
\label{sec:purity}
Let $(X, \mathcal{S})$ be a stratified complex variety, $\tilde M
\in \MHM(X)$ and $w \in \DZ$. We say that $\tilde M$ is
\define{$\mathcal{S}$-$*$-pure of weight $w$}, if, for all strata $S
\in \mathcal{S}$, the restrictions 
$l_S^*\tilde M$ are pure of weight $w$.
It is \define{$\mathcal{S}$-$!$-pure of weight $w$}, if all restrictions
$l_S^!\tilde M$ are pure of weight $w$,
and \define{$\mathcal{S}$-pure of weight $w$}, if it is 
$\mathcal{S}$-$*$-pure and 
$\mathcal{S}$-$!$-pure of weight $w$.

\begin{theorem}\label{t:purity-for-Ext}
Let $(X, \mathcal{S})$ be a cell-stratified complex variety,
$\tilde M$, $\tilde N \in \MHM(X,\mathcal{S})$ 
smooth Hodge sheaves, and $m$, $n \in \DZ$.
If $\tilde M$ is $\mathcal{S}$-$*$-pure of weight $m$ and $\tilde N$ is
$\mathcal{S}$-$!$-pure of weight $n$,
the complex (with
differential zero) of (polarizable) mixed Hodge structures 
\[\Ext_{\Perv(X,\mathcal{S})}(\tilde M, \tilde N) = \bigoplus_{i \in
  \DN} \Ext^i_{\Perv(X,\mathcal{S})}(\tilde M, \tilde N)\]
is pure of weight $n-m$.
\end{theorem}
\begin{proof}
Recall that the mixed
Hodge structure 
$\Ext^i_{\Perv(X,\mathcal{S})}(\tilde M, \tilde N)$ 
was defined from 
$\Ho^i(c_*\sHom(\tilde{M}, \tilde N)) =  
\hyper^i(\sHom(\tilde{M}, \tilde N))$.
We define $X_p$ to be the union of all strata whose codimension in $X$
is greater or equal to $p$, $X_p = \bigcup_{\text{$S \in \mathcal{S}$, $d_S + p \leq \dim_\DC X$}} S$.
This defines a filtration of $X$ by closed subvarieties, $X=X_0 \supset X_1 \supset \dots \supset X_r
= \emptyset$, where $r = \dim_\DC X +1$. Proposition
\ref{p:hodge-lgss} shows that there is a spectral sequence of mixed
Hodge structures with $E_1$-term  
$E_1^{p,q}=\hyper^{p+q}_{X_p-X_{p+1}} (\sHom(\tilde M, \tilde N))$
converging to 
$E_\infty^{p,q} =
\gr^p(\hyper^{p+q}(\sHom(\tilde M, \tilde N)))$.
Lemma \ref{l:E1-pure} below shows
that $E_1^{p,q}$ is a pure Hodge structure of weight $p+q+n-m$. 
There are no non-zero morphisms between pure Hodge structures of different
weights, hence our spectral sequence degenerates at the
$E_1$-term, i.\,e.\ $E_1=E_2=\ldots=E_\infty$.
Furthermore, $\hyper^{p+q}(\sHom(\tilde M, \tilde N))$ is
pure of weight $p+q+n-m$, since it has a finite filtration with
successive subquotients that are pure and of the same weight
(it is in fact isomorphic to the direct sum of these subquotients, see
\ref{enum:GrW-semisimple}).
\end{proof}

\begin{lemma}\label{l:E1-pure}
  Under the assumptions of Theorem \ref{t:purity-for-Ext} and with the
  notation introduced in its proof,
  $E_1^{p,q}=\hyper^{p+q}_{X_p-X_{p+1}}(\sHom(\tilde M, \tilde
  N))$ is a pure Hodge structure of weight $p+q+n-m$. 
\end{lemma}

\begin{proof}
  The decomposition $X_p-X_{p+1}= \bigcup_{\text{$S \in \mathcal{S}$, $d_S + p = \dim_\DC X$}} S$
  into strata of codimension $p$ is the decomposition of $X_p-X_{p+1}$
  into connected components. Therefore, we have
  \begin{equation*}
    \label{eq:39}
    \hyper^{p+q}_{X_p-X_{p+1}}(\sHom(\tilde M, \tilde N)) =
    \bigoplus_{\text{$S \in \mathcal{S}$, $d_S + p = \dim_\DC X$}} \hyper^{p+q}(l_S^! \sHom(\tilde M, \tilde N)),
  \end{equation*}
  where $l_S:S \hra X$ is the inclusion of the stratum $S \in
  \mathcal{S}$. For each $S \in \mathcal{S}$, we have by property \ref{enum:f-upper-shriek-hom}
  \begin{equation*}
    l_S^! \sHom(\tilde M, \tilde N)  =
    \DD(l_S^* \tilde M \otimes \DD l_S^! \tilde N).
  \end{equation*}
  
  The restrictions $l_S^* \tilde M$ and $l_S^! \tilde N$ are pure,
  so \ref{enum:pure-cplxs-are-stalk} yields 
  isomorphisms
  $l_S^* \tilde M \cong \bigoplus_{i \in \DZ} [-i]\Ho^i(l_S^*\tilde M)$
  and $l_S^! \tilde N \cong \bigoplus_{j \in \DZ} [-j]\Ho^j(l_S^! \tilde N)$.

  Fix $i$, $j \in \DZ$.
  Since $\Ho^i(l_S^* \tilde M)$ is in $\MHM(S, \{S\})$ and pure of
  weight $m+i$, Proposition \ref{p:hodge-on-Cn} shows that there is
  $A' \in \MHM(\point)$ pure of weight $m+i-d_S$ such that 
  $\Ho^i(l_S^* \tilde M) \cong [d_S]c^*A'$. 
  Then $A:= [d_S-i] A' \in \dcb(\MHM(\point))$ is pure of weight $m$
  and we have $[-i]\Ho^i(l_S^* \tilde M) \cong c^*(A)$. 
  If we proceed similarly and use \ref{enum:weight-under-shifts} and
  \ref{enum:smooth-iso}, we 
  find $B \in \dcb(\MHM(\point))$ pure of weight $n$ such that 
  $[-j]\Ho^j(l_S^! \tilde N) \cong c^!(B)$.
  Note that $A$ and $B$ are up to shift objects of $\MHM(\point)$.
  Using \ref{enum:f-upper-shriek-hom}, \ref{enum:smooth-iso} and 
  the adjunction isomorphism $\id \sira c_*c^*$, we obtain
  \begin{align*}
    \HH^{p+q}\big(\DD(c^*(A)\otimes \DD(c^!B))\big)
    &
    = \Ho^{p+q}\big(c_*c^!\sHom(A, B)\big)\\
    &
    = \Ho^{p+q}\big([2d_S]c_*c^*\sHom(A, B)(d_S)\big)\\
    &
    \sila \Ho^{p+q}\big([2d_S]\sHom(A, B)(d_S)\big),
  \end{align*}
  and this is pure of weight $p+q+n-m$ by \ref{enum:weight-verdier},
  \ref{enum:weight-under-shifts} and \ref{enum:sheafhom-via-verdier}.
\end{proof}

\subsection{Formality of some DG Algebras}
\label{sec:formality-of-some-dga}

Let $X$ be a complex variety with a cell-stratification $\mathcal{S}$,
and $\tilde M \in \MHM(X,\mathcal{S})$ a smooth
Hodge sheaf. By Corollary \ref{cd:ppreso}, there 
is a perverse-projective resolution $\tilde{P} \ra \tilde M$
of finite length, with underlying projective resolution $P \ra M$.
As in subsection \ref{sec:comparing}, we consider the complex
\begin{equation*}
\tilde{A} := \complexEnd(\tilde{P}) := \complexHom_{\Perv(X, \mathcal{S})}(\tilde{P},\tilde{P}).
\end{equation*}
of (polarizable) mixed Hodge structures. Remark \ref{rem:componice}
shows that the multiplication (= composition map)
$\tilde A \otimes \tilde A \ra \tilde A$
is a morphism of complexes of mixed Hodge
structures. Note that a complex of mixed Hodge structures is the same as an
object of the tensor category $\dgMHS$ of differential graded
mixed Hodge structures. So $\tilde A$ is a (unital) ring object in
$\dgMHS$ (a ``dg algebra of mixed Hodge structures'').

The exact faithful $\DR$-linear tensor functors ``underlying vector
space'' $\omega_0$, ``associated graded vector space'' $\gr^W_{\DR}$,
``underlying vector space'' $\eta$ and
``underlying vector space of the associated graded vector space''
$\omega_W=\eta \comp \gr^W_{\DR}$ (see subsection \ref{sec:MHS}) induce
tensor functors (denoted by the same symbol):
\begin{equation}
  \label{eq:omega-eta}
  \xymatrix{
    {\dgMHS} \ar[d]^{\omega_0} \ar[rd]^{\omega_W} \ar[r]^{\gr^W_{\DR}}
    & {\dggMod(\DR)} \ar[d]^{\eta}\\
    {\dgMod(\DR)} & {\dgMod(\DR)} 
  }
\end{equation}
Here we consider $\DR$ as a dg $\DR$-algebra concentrated in degree
$0$ and as a dgg $\DR$-algebra concentrated in degree $(0,0)$ 
(see subsections \ref{sec:review-dg-modules} and \ref{sec:diff-grad-grad}).
More elementary, $\dgMod(\DR)$ and $\dggMod(\DR)$ are the categories of 
dg real vector spaces and dg graded real vector spaces respectively.
The isomorphism $a$ from \eqref{eq:ffiso} induces an isomorphism
\begin{equation}
  \label{eq:a-omega-iso}
  a:\omega_0 \sira \omega_W  
\end{equation}
between the induced functors.
Then $A=\omega_0(\tilde A)$ is the dg algebra
$\cEnd(P)$.
Its cohomology is the extension algebra 
$\Ext_{\Perv(X, \mathcal{S})}(P)$ and isomorphic to  
$\Ext_{\Perv(X,\mathcal{S})}(M)$.

\begin{theorem}
\label{t:formality-of-endo-complex}
Let $(X, \mathcal{S})$, $\tilde{P} \ra \tilde M$, $\tilde A$
and $A$ be as above, and $w$ an integer. 
If $\tilde M$ is $\mathcal{S}$-pure of weight $w$,
then $A$ is formal. More precisely, there are 
a dg subalgebra $\Sub(A)$ of $A$ and 
dga-quasi-isomorphisms $A \hla \Sub(A) \sra \Ho(A)$.
\end{theorem}

\begin{proof}
  Consider the dgg algebra $\tilde{R} :=\gr^W_{\DR} (\tilde A)$.
  Its graded components are 
  $\tilde{R}^{ij} = \gr_{j, \DR}^W(\tilde{A}^i)$.
  By Proposition \ref{p:coincide} and Theorem \ref{t:purity-for-Ext},
  the complexes (with differential zero) of mixed Hodge structures 
  $\Ho(\tilde A)$ and $\Ext_{\Perv(X, \mathcal{S})}(\tilde M, \tilde M)$
  are isomorphic and pure of weight $0$. So
  \begin{equation*}
    \gr^W_{j, \DR} (\Ho^i(\tilde A))=\Ho^i(\gr^W_{j, \DR}(\tilde A))
    =\Ho^i(\tilde{R}^{*j})
    =(\Ho(\tilde{R}))^{ij}
  \end{equation*}
  vanishes for $i \not = j$, and the dgg algebra $\Ho(\tilde{R})$ is pure of weight
  $0$.
  Proposition \ref{p:formality-of-pure-dgg-algebras} shows the
  existence of dgga-quasi-isomorphisms
  $\tilde R \hla \Gamma(\tilde R) \sra \Ho(\tilde R)$.
  We define $R:= \eta(\tilde{R})$ and $\Gamma(R) := \eta(\Gamma(\tilde
  R))$, apply $\eta$ to 
  $\tilde R \hla \Gamma(\tilde R) \sra \Ho(\tilde R)$
  and obtain
  the dga-quasi-isomorphisms in the first row in the following diagram
  \begin{equation}
    \label{eq:sub-def}
    \xymatrix{
      {\omega_W(\tilde{A})}
      =
      R 
      & 
      \ar@{_{(}->}[l] 
      {\Gamma(R)}
      \ar@{-{>>}}[r] 
      &
      {\Ho(R)}
      =
      {\omega_W(\Ho(\tilde{A}))}
      \\
      {\omega_0(\tilde{A})=A}
      \ar[u]^{\sim}_{a_{\tilde{A}}}
      &
      {\Sub(A)}
      \ar@{_{(}->}[l] 
      \ar[u]^{\sim}
      &
      {{\Ho(A)={\omega_0(\Ho(\tilde{A}))}},}
      \ar[u]^{\sim}_{a_{\Ho(\tilde{A})}}
    }
  \end{equation}
  where 
  the vertical isomorphisms come from the natural isomorphism \eqref{eq:a-omega-iso} and 
  the dg subalgebra
  $\Sub(A)\subset A$ is defined as the pull-back, i.\,e.\ 
  $\Sub(A) := a\inv(\Gamma(R))$. 
  All vertical (horizontal) morphisms in this diagram are dga-(quasi-)iso\-mor\-phisms.
\end{proof}

We generalize Theorem \ref{t:formality-of-endo-complex} slightly as
follows. Let $(X, \mathcal{S})$ be as above, $I$ a finite set
and $\tilde{P}_\alpha \ra \tilde{M}_\alpha$ perverse-projective
resolutions of finite length of smooth Hodge sheaves 
$\tilde{M}_\alpha$ ($\alpha \in I$).
Let $\tilde{P}$ be the direct sum of the 
$(\tilde{P}_\alpha)_{\alpha \in I}$, 
$\tilde{A} = \complexEnd(\tilde{P})$,
$A=\omega_0(\tilde{A})=\cEnd(P)$
the
dg algebra underlying the 
``dg algebra of mixed Hodge structures'' $\tilde A$, and $e_\alpha \in
A^0$ the projector from $P$ onto the 
direct summand ${P}_\alpha$.
The cohomology $\Ho(A)$ of $A$ is isomorphic to the extension algebra
$\Ext_{\Perv(X,\mathcal{S})}(M)$, 
where $M$
is the direct sum of the underlying perverse sheaves 
$(M_\alpha)_{\alpha \in I}$.

\begin{theorem}
  \label{t:formality-of-endo-complex-with-objects}
  Let $X$, $\mathcal{S}$, $I$, $\tilde{P}_\alpha \ra
  \tilde{M}_\alpha$, $\tilde A$, $A$, $e_\alpha$ be as above.  Let $w_\alpha
  \in \DZ$ ($\alpha \in I$) be integers. 
  If $\tilde{M}_\alpha$ is $\mathcal{S}$-pure of weight $w_\alpha$,
  for all $\alpha \in I$, 
  then $A$ is formal. More precisely, there are
  a dg subalgebra $\Sub(A)$ of $A$ containing 
  all $(e_\alpha)_{\alpha \in I}$ and
  quasi-isomorphisms $A \hla \Sub(A) \sra \Ho(A)$ of 
  dg algebras.
\end{theorem}

\begin{proof}
  The proof is very similar to that of Theorem
  \ref{t:formality-of-endo-complex}. 
  Mainly, we use Proposition
  \ref{p:formality-of-pure-dgg-algebras-with-objects} instead of
  Proposition \ref{p:formality-of-pure-dgg-algebras}.
\end{proof}

\subsection{Formality of Cell-Stratified Varieties}
\label{sec:form-cell-strat}

\begin{theorem}
  \label{t:formality-for-cell-stratified-varieties}
  Let $X$ be a complex variety with a
  cell-stratification $\mathcal{S}$, 
  $(\tilde{M}_\alpha)_{\alpha \in I}$ a finite number of
  smooth Hodge sheaves $\tilde{M}_\alpha \in \MHM(X, \mathcal{S})$ 
  with direct sum $\tilde M = \bigoplus \tilde M_\alpha$.
  Denote by 
  $\Ext(\ul M):=\Ext_{\Sh(X)}(\ul M)$
  the extension algebra of $\ul{M}=\rera(\tilde M)$, a dg algebra with differential $d=0$.
  Let $e_\alpha \in \Ext^0(\ul M)=\End_{\Sh(X)}(\ul M)$ be the
  projector from $\ul M$ onto the direct summand 
  $\ul{M}_\alpha=\rera(\tilde{M}_\alpha)$.
  If there are integers $(w_\alpha)_{\alpha \in I}$, such that
  $\tilde{M}_\alpha$ is $\mathcal{S}$-pure of weight $w_\alpha$, for
  all $\alpha \in I$,
  there is an
  equivalence of triangulated categories
  \begin{equation}
    \label{eq:praeverd-dgprae}
    \PraeVerd(\{\ul M_\alpha\}_{\alpha \in I}, \dcb(X))
    \sira \dgPraeDer_{\Ext(\ul M)}(\{e_\alpha\Ext(\ul M)\}_{\alpha \in I}).
  \end{equation}
\end{theorem}

    Under the equivalence we construct in the proof, the objects
    $\ul{M}_\alpha$ and $e_\alpha\Ext(\ul M)$ correspond.
    We do not emphasize similar obvious correspondences in the
    following.
 
    Due  to the equivalences \eqref{eq:beilinsonequi} and \eqref{eq:real-X-S}
    we can replace $\Ext(\ul M)$
    by $\Ext_{\Perv(X, \mathcal{S})}(M)$ or 
    $\Ext_{\Perv(X)}(M)$,
    and also the left hand side of
    \eqref{eq:praeverd-dgprae} by
    $\PraeVerd(\{M_\alpha\}_{\alpha \in I})$ formed in $\dcb(\Perv(X,
    \mathcal{S})))$ 
    or
    in $\dcb(\Perv(X))$.

\begin{proof}
Let $\tilde{P}_\alpha \ra \tilde{M}_\alpha$ be 
perverse-projective resolutions of finite length (Corollary
\ref{cd:ppreso}) and 
$P_\alpha \ra M_\alpha$ the underlying
projective resolutions in $\Perv(X, \mathcal{S})$.
Let $\tilde P=\bigoplus \tilde{P}_\alpha$ and $P = \bigoplus
P_\alpha$.
As in the second part of subsection
\ref{sec:formality-of-some-dga}, we define 
$\tilde{A} =
\complexEnd(\tilde{P})$
and $A=\omega_0(\tilde{A})
=\cEnd(P)$.
Theorem \ref{t:formality-of-endo-complex-with-objects} yields 
a dg subalgebra $\Sub(A)$ of $A$ and dga-quasi-isomorphisms
$A \hla \Sub(A) \sra \Ho(A)$.
We claim that
\begin{align}
  \PraeVerd(\{\ul M_\alpha\}, \dcb(X)) 
  \label{eq:the-equivalence-i}
  & \xla{\real_{X, \mathcal{S}}} 
  \PraeVerd(\{M_\alpha\}, \dcb(\Perv(X, \mathcal{S})))\\
  \label{eq:the-equivalence-ii}
  & \xra{\complexHom(P, ?)}
  \dgPraeDer_A(\{e_\alpha A\})\\
  \label{eq:the-equivalence-iii}
  & \xla{? \Lotimesover{\Sub(A)} A}
  \dgPraeDer_{\Sub(A)}(\{e_\alpha \Sub(A)\})\\
  \label{eq:the-equivalence-iv}
  & \xra{? \Lotimesover{\Sub(A)} \Ho(A)}
  \dgPraeDer_{\Ho(A)}(\{e_\alpha \Ho(A)\})\\
  \label{eq:the-equivalence-v}
  & \xra{? \Lotimesover{\Ho(A)} \Ext(\ul M)}
  \dgPraeDer_{{\Ext(\ul M)}}(\{e_\alpha {\Ext(\ul M)}\}).
\end{align}
is a sequence of triangulated equivalences. 
By Theorem \ref{t:bgsequi}, \eqref{eq:the-equivalence-i} is an
equivalence.
The isomorphisms $P_\alpha \ra
M_\alpha$ in $\dcb(\Perv(X, \mathcal{S}))$,
Proposition \ref{p:hom-p-dot-iso} and Remark \ref{rem:hom-p-dot-iso}
show that \eqref{eq:the-equivalence-ii} is an equivalence.
(The $e_\alpha$ here are in fact representatives of the $e_\alpha$ in
the theorem, but we do not care about this too much.)
The dga-quasi-isomorphisms 
$A \hla \Sub(A) \sra \Ho(A)$
induce 
equivalences 
between the respective derived 
categories $\dgDer$ and restrict to the equivalences
 \eqref{eq:the-equivalence-iii} and \eqref{eq:the-equivalence-iv}.
The last equivalence is similarly induced by the dga-isomorphism
(cf.\ \eqref{eq:coho-exts})
\begin{equation*}
  \Ho(A)
   = \Ho(\cEnd(P))  {=} \Ext(P) 
   \xrightarrow[\sim]{P \sira M} \Ext(M) \xra{\real} \Ext(\ul{M}).
\end{equation*}
\end{proof}

\begin{remark}
  \label{rem:formy-pm}
  We use the notation 
  $\Formy_{\tilde{P} \ra \tilde{M}}$
  for the
  ``formality equivalence''
  \eqref{eq:praeverd-dgprae} constructed in the proof of 
  Theorem \ref{t:formality-for-cell-stratified-varieties} to indicate
  that it mainly depends on 
  the perverse-projective resolutions
  $\tilde{P}_\alpha \ra \tilde{M}_\alpha$. 
\end{remark}

\subsection{Formality and Intersection Cohomology Complexes}
\label{sec:form-inters-cohom}

Let $(X, \mathcal{T})$ be a cell-stratified complex variety,
$\mathcal{E}=\Ext(\IC(\mathcal{T}))$ the extension
algebra of the direct sum 
$\IC(\mathcal{T})$
of the $({{\IC}}_T)_{T \in \mathcal{T}}$, and $e_T$ the
projector from this direct sum onto $\IC_T$.
Then the dg algebra $\mathcal{E}$ satisfies the conditions
\ref{enum:form-pg}-\ref{enum:form-sdga}, hence 
\begin{equation*}
  \dgPraeDer_{\mathcal{E}}(\{e_T\mathcal{E}\}_{T \in \mathcal{T}})=\dgPraeDer(\mathcal{E})=\dgPerDer(\Ext(\IC(\mathcal{T}))
\end{equation*}
thanks to Theorem \ref{t:t-structure-auf-perf}.
If $\tilde{\IC}_T$ is
$\mathcal{T}$-pure of weight $d_T$, for all $T \in I$, these equalities, Theorem
\ref{t:formality-for-cell-stratified-varieties} and 
\eqref{eq:derbxs-praeverd-verd} yield an equivalence
\begin{equation*}
  \dbc(X, \mathcal{T}) 
  \cong \dgPerDer(\Ext(\IC(\mathcal{T}))).
\end{equation*}

Similarly, if $\mathcal{T}'$ is a subset of $\mathcal{T}$
and all $\tilde{\IC}_{T'}$ are $\mathcal{T}$-pure of weight $d_{T'}$, for
$T' \in \mathcal{T}'$, we 
get 
by Theorem
\ref{t:formality-for-cell-stratified-varieties}
an equivalence
\begin{equation}
  \label{eq:T-prime}
  \PraeVerd(\{\IC_{T'}\}_{T' \in \mathcal{T}'}, \dcb(X))
  \cong 
  \dgPerDer(\Ext(\IC(\mathcal{T}'))).
\end{equation}

\begin{theorem}
  \label{t:formality-ic}
  Let $(X, \mathcal{S})$ be a stratified variety with simply connected
  strata. Let $\mathcal{T}$ be a cell-stratification refining
  $\mathcal{S}$. 
  If $\tilde{\IC}_S$ is $\mathcal{T}$-pure of weight $d_S$ for all $S
  \in \mathcal{S}$, there is a triangulated equivalence
  \begin{equation}
    \label{eq:simplyconn}
    \dcb(X, \mathcal{S})
    \sira \dgPerDer(\Ext(\IC(\mathcal{S}))).
  \end{equation}
  This equivalence is t-exact with respect to the perverse
  t-structure on $\dcb(X, \mathcal{S})$ and the 
  t-structure from Theorem \ref{t:t-structure-auf-perf} on
  $\dgPerDer$. Restriction to the respective hearts yields an equivalence
  \begin{equation*}
    \Perv(X, \mathcal{S}) \cong \dgFilMod(\Ext(\IC(\mathcal{S}))).
  \end{equation*}
\end{theorem}

\begin{proof}
  Since each $S \in \mathcal{S}$ is irreducible, it contains a (unique)
  dense stratum $T(S) \in \mathcal{T}$; then $\IC_{S}=
  \IC_{T(S)}$. 
  We apply equivalence \eqref{eq:T-prime} to the set $\mathcal{T}'$ of
  these dense strata. Then we use 
  \eqref{eq:derbxs-praeverd-verd} and the fact that 
  the $(\IC_S)_{S \in \mathcal{S}}$ are the simple objects of
  $\Perv(X, \mathcal{S})$. This show equivalence
  \eqref{eq:simplyconn}.
  
  Since
  $\IC(\mathcal{S})$ is mapped to $e_S\Ext(\IC(\mathcal{S}))$ (where
  $e_S$ is the obvious projector),
  the remaining statements follow from
  Theorem \ref{t:t-structure-auf-perf}.
\end{proof}

\begin{remark}
  \label{rem:formy-simplyconn}
  We use the notation 
  $\Formy_{\tilde{P} \ra \tilde{\IC}(\mathcal{S})}^{\mathcal{T}}$
  for equivalence \eqref{eq:simplyconn} to indicate its dependence on
  the refinement $\mathcal{T}$ and the perverse-projective resolutions
  $\tilde{P}_S \ra \tilde{\IC}_S$ (cf.\ Remark \ref{rem:formy-pm}).
\end{remark}

\begin{remark}
  \label{rem:T-pure}
  In Theorem \ref{t:formality-ic}, it is sufficient to require that
  each $\tilde{\IC}_S$ is $\mathcal{T}$-$*$-pure of weight
  $d_{S}$:
  If $S$ is a stratum of a stratified variety and $l$ the inclusion
  of a subvariety, we have
  $\DD(\tilde{\IC}_S)\cong \tilde{\IC}_S(d_{S})$ and obtain
  \begin{equation*}
      l^!(\tilde{\IC}_S)
      = \DD (l^* (\DD (\tilde{\IC}_S))) 
      \cong \DD (l^* (\tilde{\IC}_S(d_{S})))
      = \DD (l^* (\tilde{\IC}_S)) (-d_{S}).
  \end{equation*}
  So if $l^*(\tilde{\IC}_S)$ is pure of weight $d_S$, then
  $l^!(\tilde{\IC}_S)$ is pure of weight 
  $d_S$.
\end{remark}

\subsection{Formality of Partial Flag Varieties}
\label{sec:form-part-flag}

Let $G$ be a complex connected reductive
affine algebraic group and a $B \subset G$ Borel subgroup. Let $P$, $Q$ be parabolic
subgroups of $G$ containing $B$.
The Bruhat decomposition of the flag variety $G/B$ into $B$-orbits is a
cell-stratification. 
More generally, the $B$-orbits on the partial flag variety $G/P$ form
a cell-stratification, and the $Q$-orbits on $G/P$ form a
stratification.
(These stratifications are indeed Whitney stratifications thanks to \cite[Thm.~2]{kaloshin-whitney}.) 

\begin{proposition}
  \label{p:q-orbits-sc}
  The $Q$-orbits in $G/P$ are simply connected.
\end{proposition}

This is probably well-known but we could not find a proof in the literature.

\begin{proof}
  Let $Y \subset G/P$ be a $Q$-orbit and $T\subset B \subset G$ a
  maximal torus. 
  Then $Y = QwP/P$ for some $w$ in the normalizer of $T$ in $G$.
  The stabilizer $S$ of $wP/P$ in $Q$ is
  $Q \cap w P w\inv$ and is
  connected as the intersection of two parabolic subgroups of a connected reductive
  group (\cite[14.22]{borel-lag}). 
  The exact sequence of homotopy groups associated to the
  $S$-principal fiber bundle $Q \ra Y$, $q
  \mapsto qwP/P$, shows that it is sufficient to prove surjectivity of 
  $\pi_1(S) \ra \pi_1(Q)$.
  Note that $T \subset S$. 
  If $L_Q$ is a Levi subgroup of $Q$, then $T \subset L_Q$ and
  $\pi_1(Q)=\pi_1(L_Q)$. Hence surjectivity is a consequence of the
  following Lemma \ref{l:pi1tg}. 
\end{proof}

\begin{lemma}
  \label{l:pi1tg}
  If $T$ is a maximal torus in a connected reductive group $L$, then
  $\pi_1(T) \ra \pi_1(L)$ is surjective.
\end{lemma}

\begin{proof}
  Let $A \subset L$ be a Borel subgroup containing $T$. Then
  $\pi_1(T)=\pi_1(A)$ and the long exact
  homotopy sequence for the $A$-principal fiber bundle $L \ra L/A$
  yields an exact sequence
  $ \pi_1(A) \ra 
    \pi_1(L) \ra 
    \pi_1(L/A) \ra \pi_0(A)$.
  Since $A$ is connected and fundamental groups of topological groups
  are abelian, $\pi_1(L/A)$ is abelian and vanishes since
  the flag variety $L/A$ has only even cohomology.
  So $\pi_1(A) \ra \pi_1(L)$ is surjective.
\end{proof}

\begin{theorem}
  \label{t:bptgp}
  Let $F \in \dcb(\MHM(G/P))$
  be
  pure and smooth along the stratification by $B$-orbits. Let $l:Y
  \ra G/P$ be the inclusion of a $B$-orbit. Then $l^*(F)$ and $l^!(F)$
  are pure as well, of the same weight as $F$.
\end{theorem}

\begin{proof}
  Let $\pi:G/B \ra G/P$ be the obvious projection and $Z \subset
  \pi\inv(Y)$ the unique Bruhat cell such that $\pi$ induces an
  isomorphism $Z \sira Y$:
  \begin{equation*}
    \xymatrix{
      Z \ar[r] \ar[d]_\sim^{\pi} 
      & {\pi\inv(Y)} \ar[r] \ar[d]^{\pi}
      & {G/B} \ar[d]^{\pi}\\
      Y \ar@{=}[r]
      & Y \ar[r]^{l}
      & {G/P.}
    }
  \end{equation*}
  Since $\pi$ is smooth of relative dimension $n = \dim_\DC
  (P/B)$, $\pi^*(F)= [-2n]\pi^!(F)(-n)$ and $\pi^!(F)$ are pure of
  the same weight as $F$ 
  (use 
  \ref{enum:weights-f-upper-star}, 
  \ref{enum:weight-under-shifts} and 
  \ref{enum:smooth-iso}).
  Obviously, $\pi^*(F)$ is
  smooth along the stratification by Bruhat cells, and so is
  $\pi^!(F)$. Hence we deduce from 
  \cite[Parabolic Purity Theorem]{Soepurity}
  that $l^*(F)$ and $l^!(F)$ are pure of the same
  weight as $F$.
\end{proof}

\begin{theorem}
  \label{t:formality-partial-flag-variety}
  Let $\mathcal{Q}$ be the stratification of $G/P$ into
  $Q$-orbits. Then
  there is a t-exact equivalence
  $\dbc(G/P, \mathcal{Q}) 
  \cong 
  \dgPerDer(\Ext(\IC(\mathcal{Q})))$
  inducing an equivalence
  $\Perv(G/P, \mathcal{Q}) \cong \dgFilMod(\Ext(\IC(\mathcal{Q})))$.
\end{theorem}

\begin{proof}
  The strata of $\mathcal{Q}$ are simply connected by Proposition
  \ref{p:q-orbits-sc}.
  The cell-stra\-ti\-fi\-ca\-tion $\mathcal{T}$ of $G/P$ into
  $B$-orbits refines $\mathcal{Q}$, and
  every $\tilde{\IC}_Q$ is $\mathcal{T}$-pure of weight $d_Q$,
  for $Q \in \mathcal{Q}$ (Theorem \ref{t:bptgp}).
  Hence we can apply Theorem \ref{t:formality-ic}.
\end{proof}

\subsection{Complex coefficients}
\label{sec:complex-coefficients}

Let $(X, \mathcal{S})$ be a cell-stratified variety.
We denote the derived category of sheaves of complex vector
spaces on $X$ by $\dcb(X)_{\DC}$ and use similar notation in the
following. The obvious extension of
scalars functor $\dcb(X)\ra \dcb(X)_{\DC}$, $N \mapsto N_{\DC}$
restricts to a functor $\dcb(X, \mathcal{S}) \ra \dcb(X,
\mathcal{S})_{\DC}$. This functor is t-exact with respect to the perverse
t-structure and maps projective objects of $\Perv(X,
\mathcal{S})$ to projective objects of $\Perv(X, \mathcal{S})_{\DC}$.
If $M$, $N$ are in $\dbc(X)$ we have a
canonical isomorphism
\begin{equation}
  \label{eq:hom-complex}
  \DC \otimes_\DR 
  \Hom_{\dcb(X)}(M, N) \sira
  \Hom_{\dcb(X)_\DC}(M_\DC, N_\DC).
\end{equation}

Under the assumptions of Theorem 
\ref{t:formality-for-cell-stratified-varieties} the complexified
version 
\begin{equation*}
  \PraeVerd(\{\ul{M_\alpha}_\DC\}_{\alpha \in I}, \dcb(X)_{\DC})
  \sira \dgPraeDer_{\Ext(\ul{M}_{\DC})}(\{e_\alpha\Ext(\ul{M}_\DC)\}_{\alpha \in I}).
\end{equation*}
of equivalence \ref{eq:praeverd-dgprae} is true: With the notation
used in the proof of Theorem
\ref{t:formality-for-cell-stratified-varieties},
$(P_\alpha)_\DC \ra (M_\alpha)_\DC$ is a projective resolution in
$\Perv(X, \mathcal{S})_\DC$. 
From \eqref{eq:hom-complex} we see that $\DC \otimes_\DR A$ and
$\complexEnd(P_\DC)$ are isomorphic as dg algebras. Since $A$ is
formal, the same is true for $\complexEnd(P_\DC)$. Now it is easy
to adapt the sequence of equivalences 
\eqref{eq:the-equivalence-i}-\eqref{eq:the-equivalence-v} to the case
of complex coefficients.

In particular, Theorem \ref{t:formality-partial-flag-variety} is also true for
complex coefficients.

\section{Formality and Closed Embeddings}
\label{cha:form-clos-embedd}

We formulate in subsection \ref{sec:goal-section} the goal of this section
and explain in subsection \ref{sec:form-norm-nons} the main application.
The proof of the goal is divided into several parts and given in the following subsections.

\subsection{The Goal of the Section}
\label{sec:goal-section}

Let $(X, \mathcal{S})$ and $(Y, \mathcal{T})$ be cell-stratified
complex varieties, and $i: Y \ra X$ a closed embedding such that
$i(\mathcal{T}):=\{i(T)\mid T \in \mathcal{T}\} \subset \mathcal{S}$.
We say for short that $i: (Y, \mathcal{T}) \ra (X, \mathcal{S})$ is a
\define{closed embedding of cell-stratified varieties}.
(If $\mathcal{S}$ and $\mathcal{T}$ are merely stratifications, the
term \define{closed embedding of stratified varieties} is defined similarly.)

Assume that we are in the setting of Theorem
\ref{t:formality-for-cell-stratified-varieties}
on $X$ and on $Y$. More precisely, 
let 
$(\tilde{M}_\alpha)_{\alpha \in I}$ 
and $(\tilde{N}_\alpha)_{\alpha \in I}$ be finite collections of smooth
Hodge sheaves on $X$ and on $Y$. 
Assume that there are integers $(w_\alpha)_{\alpha \in I}$ (resp.\ $(v_\alpha)_{\alpha \in I}$) such that
$\tilde{M}_\alpha$ (resp.\ $\tilde{N}_\alpha$) is $\mathcal{S}$-pure
(resp.\ $\mathcal{T}$-pure) of weight $w_\alpha$ (resp.\ $v_\alpha$), for all
$\alpha \in I$. 

Let $\decal$ be an integer (in our applications, $\decal$ will be
the
negative complex codimension of the inclusion $i: Y \ra X$, and we
will have
$w_\alpha + \decal = v_\alpha$ for all $\alpha \in I$).
Suppose
that there are isomorphisms 
\begin{equation}
  \label{eq:assumption}
  \tilde\sigma_\alpha:[\decal]i^*(\tilde{M}_\alpha)\sira \tilde{N}_\alpha 
\end{equation}
in $\dcb(\MHM(Y))$, for all $\alpha \in I$. Let
$\tilde\sigma:[\decal]i^*(\tilde{M}) \sira \tilde{N}$ be the direct
sum of these isomorphisms, where $\tilde M = \bigoplus \tilde M_\alpha$ and 
$\tilde N = \bigoplus \tilde N_\alpha$. 

Let $\tilde \pi_\alpha:\tilde P_\alpha \ra \tilde M_\alpha$ and
$\tilde \rho_\alpha: \tilde Q_\alpha \ra \tilde N_\alpha$ 
be perverse-projective resolutions and 
$\tilde \pi:\tilde P \ra \tilde M$ and
$\tilde \rho: \tilde Q \ra \tilde N$ their direct sum.
The vertical
equivalences in 
\begin{equation}
  \label{eq:the-goal}
  \xymatrix{
    {\PraeVerd(\{\ul M_\alpha\},\dcb(X))}
    \ar[rr]^{[\decal]i^*}
    \ar[d]_{\Formy_{\tilde{P} \ra \tilde{M}}}^\sim
    &&
    {\PraeVerd(\{\ul N_\alpha\},\dcb(Y))}
    \ar[d]_{\Formy_{\tilde{Q} \ra \tilde{N}}}^\sim
    \\
    {\dgPraeDer_{\Ext(\ul M)}(\{e_\alpha \Ext(\ul M)\})}
    \ar[rr]_-{? \Lotimesover{\Ext(\ul M)} \Ext(\ul N)}
    \ar@{=>}[rru]^\sim
    &&
    {\dgPraeDer_{\Ext(\ul N)}(\{e_\alpha \Ext(\ul N)\})}
  }
\end{equation}
come from Theorem
\ref{t:formality-for-cell-stratified-varieties},
cf.\ Remark \ref{rem:formy-pm}.
The upper horizontal arrow is the restriction of 
$[\decal]i^*:\dcb(X) \ra \dcb(Y)$, the lower one is
the restriction of the
extension of scalars functor  
coming from the composition
\begin{equation}
  \label{eq:extmn}
  \Ext(\ul M) \xra{[\decal]i^*} \Ext([\decal]i^*(\ul M))
  \xrightarrow[\sim]{\sigma} \Ext(\ul N),
\end{equation}
where the $\sigma$ above the arrow indicates that the isomorphism is
constructed using the isomorphism
$\sigma:[\decal]i^*(\ul M) \sira \ul N$.
\begin{theorem}
  \label{t:closed-embedding-goal}
  Keep the above assumptions. Then diagram \eqref{eq:the-goal}
  commutes up to the indicated natural isomorphism, 
  i.\,e.\ there is a natural isomorphism (of triangulated functors)
  \begin{equation*}
    (? \Lotimesover{\Ext(\ul M)} \Ext(\ul N)) \comp \Formy_{\tilde{P} \ra
      \tilde{M}}
    \sira
    \Formy_{\tilde{Q} \ra \tilde{N}} \comp \; [\decal]i^*.
  \end{equation*}
\end{theorem}
\begin{proof}
  Let $\tilde A = \cEnd(\tilde P)$ and $\tilde B = \cEnd(\tilde Q)$.
  We define $A$ and $\Sub(A)$ as in the proof of Theorem 
  \ref{t:formality-for-cell-stratified-varieties}, and $B$ and
  $\Sub(B)$ accordingly.
  We only prove the theorem in the case that $I$ is a singleton, the
  general case being an obvious generalization.
  We expand diagram \eqref{eq:the-goal} according to the sequence of
  equivalences \eqref{eq:the-equivalence-i}-\eqref{eq:the-equivalence-v} to the
  following diagram.
\newpage
  \begin{equation}
    \label{eq:the-big}
    \xymatrix@C+1.5cm@R+2cm{
      {\PraeVerd(\ul{M},\dcb(X)) }
      \ar[r]^{[\decal]i^*}
      & 
      {\PraeVerd(\ul{N},\dcb(Y)) }
      \\
      {\PraeVerd(M, \dcb(\Perv(X, \mathcal{S})))}
      \ar[r]^{[\decal]L\p i^*} 
      \ar[u]^{\real_{X, \mathcal{S}}}_\sim
      \ar[d]_{\cHom(P, ?)}^\sim
      & 
      {\PraeVerd(N, \dcb(\Perv(Y, \mathcal{T})))}
      \ar[u]_{\real_{Y, \mathcal{T}}}^\sim
      \ar@{=>}[lu]^{\text{Proposition \ref{p:real-ius}}}_\sim
      \ar[d]^{\cHom(Q, ?)}_\sim \\
      {\dgPraeDer_A(A)}
      \ar@{=>}[ru]^\sim_{\text{Corollary \ref{c:LI-tensor}}}
      \ar[r]^-{?\Lotimesover{A} \cHom(Q, [\decal]\p i^* (P))}
      & 
      {\dgPraeDer_B(B)}
      \\
      {\dgPraeDer_{\Sub(A)}(\Sub(A))}
      \ar[u]^{?\Lotimesover{\Sub(A)}A}_{\sim}
      \ar[r]^(.55){?\Lotimesover{\Sub(A)} \Sub(\cHom(Q, [\decal]\p i^* (P)))}
      \ar[d]_{?\Lotimesover{\Sub(A)}\Ho(A)}^{\sim}
      & 
      {\dgPraeDer_{\Sub(B)}(\Sub(B))}
      \ar[u]_{?\Lotimesover{\Sub(B)}B}^{\sim}
      \ar[d]^{?\Lotimesover{\Sub(B)}\Ho(B)}_{\sim}
      \ar@{=>}[lu]_\sim^{\text{Subsection \ref{sec:pass-cohom-algebr}}} 
      \ar@{=>}[ld]_\sim^{\text{Subsection \ref{sec:pass-cohom-algebr}}} \\
      \dgPraeDer_{\Ho(A)}(\Ho(A))
      \ar[r]^-{?\Lotimesover{\Ho(A)} \Ho(\cHom(Q, [\decal]\p i^* (P)))}
      \ar[d]_{?\Lotimesover{\Ho(A)} \Ext(\ul M)}^{\sim}
      & 
      { \dgPraeDer_{\Ho(B)}(\Ho(B))}
      \ar[d]^{?\Lotimesover{\Ho(B)} \Ext(\ul N)}_{\sim}
      \ar@{=>}[ld]_\sim^{\text{Subsection \ref{sec:pass-geom-extens}}}
      \\
      {\dgPraeDer_{{\Ext(\ul M)}}({\Ext(\ul M)})}
      \ar[r]^-{? \Lotimesover{\Ext(\ul M)} \Ext(\ul N)}
      & 
      {\dgPraeDer_{{\Ext(\ul N)}}({\Ext(\ul N)})}
    }
  \end{equation}
  The horizontal functors in this diagram will be explained in the rest of this section.
  It will result from the specified Proposition, Corollary and subsections that 
  all squares commute up to natural isomorphism, as
  indicated by the diagonal arrows.
  Since all vertical arrows are equivalences, this proves the theorem.
\end{proof}

\subsection{Formality and Normally Nonsingular Inclusions}
\label{sec:form-norm-nons}

If $f: Y \ra X$ is a closed embedding of irreducible varieties and, in the
classical topology, a normally nonsingular inclusion of (complex) codimension
$c$ (cf.\ \cite[I.1.11]{gormac-stratified}), we have a canonical isomorphism (\cite[5.4.1]{gormac-IH-ii},
and \cite[0]{BBD} for the different normalization)
\begin{align}
  \label{eq:nnperv}
  [-c]f^*(\IC(X)) & \sira \IC(Y)
  && \text{in $\Perv(Y)$.}
\intertext{It comes from a canonical isomorphism}
  \label{eq:nnmhm}
  [-c]f^*(\tilde{\IC}(X)) & \sira \tilde{\IC}(Y)
  && \text{in $\MHM(Y)$.}
\end{align}

Let $i:Y \ra X$ be a closed embedding of varieties and
assume that we have stratifications $\mathcal{S}$ of $X$ and
$\mathcal{T}$ of $Y$ (with irreducible strata). 
Suppose that
$\mathcal{S} \ra \mathcal{T}$, $S \mapsto i\inv(S)=S\cap Y$, is
bijective and that $i|_{\ol S\cap Y}: \ol S \cap Y \ra \ol S$ is a
normally nonsingular inclusion of a fixed codimension $c$, for all $S \in \mathcal{S}$.
The isomorphisms \eqref{eq:nnperv} and \eqref{eq:nnmhm} induce
isomorphisms
\begin{align}
  \label{eq:nnpervstrat}
  [-c]i^*(\IC_S) & \sira \IC_{S\cap Y} 
  && \text{in $\Perv(Y, \mathcal{T})$ and}\\
  \label{eq:nnmhmstrat}
  [-c]i^*(\tilde{\IC}_S) & \sira \tilde{\IC}_{S\cap Y}
  && \text{in $\MHM(Y, \mathcal{T})$.}
\end{align}

\begin{theorem}
  \label{t:formality-ic-closed-embedding}
  Let $i:Y \ra X$ and $\mathcal{S}$, $\mathcal{T}$ be as above.
  Assume that
  \begin{enumerate}
  \item all strata in $\mathcal{S}$ and $\mathcal{T}$ are simply connected,
  \item there are cell-stratifications $\mathcal{S}'$ and
    $\mathcal{T}'$ refining $\mathcal{S}$ and $\mathcal{T}$
    such that
    \begin{itemize}
    \item $\tilde{\IC}_{S}$ is $\mathcal{S}'$-pure of weight $d_S$,
      for all $S \in \mathcal{S}$,
    \item $\tilde{\IC}_{T}$ is $\mathcal{T}'$-pure of weight $d_T$,
      for all $T \in \mathcal{T}$, and
    \item $i:(Y, \mathcal{T}') \ra (X, \mathcal{S}')$ is a closed
      embedding of cell-stratified varieties.
    \end{itemize}
  \end{enumerate}
  Let $\tilde{P}_S \ra \tilde{\IC}_S$ and $\tilde{Q}_T \ra
  \tilde{\IC}_T$ be perverse-projective resolutions
  (smooth along the cell-stratifications), for $S \in \mathcal{S}$ and $T \in \mathcal{T}$.
  Then diagram
  \begin{equation*}
    \xymatrix{
      {\dcb(X, {\mathcal{S}})}
      \ar[rrr]^{[-c]i^*}
      \ar[d]^\sim_{\Formy_{\tilde{P} \ra \tilde{\IC}(\mathcal{S})}^{\mathcal{S}'}}
      &&&
      {\dcb(Y, {\mathcal{T}})}
      \ar[d]_\sim^{\Formy_{\tilde{Q} \ra \tilde{\IC}(\mathcal{T})}^{\mathcal{T}'}}
      \\
      {\dgPerDer(\Ext(\IC({\mathcal{S}})))}
      \ar[rrr]^-{? \Lotimesover{\Ext(\IC({\mathcal{S}}))} \Ext(\IC({\mathcal{T}}))}
      &&&
      {\dgPerDer(\Ext(\IC({\mathcal{T}})))}
    }
  \end{equation*}
  is commutative (up to natural isomorphism). The vertical functors
  are given by Theorem \ref{t:formality-ic} (cf.\ Remark
  \ref{rem:formy-simplyconn}), the extension of scalars functor is induced by
  the isomorphisms \eqref{eq:nnpervstrat}.
  All functors in this diagram are t-exact (with respect to the
  perverse t-structure and the
  t-structure from Theorem \ref{t:t-structure-auf-perf}).
\end{theorem}

\begin{proof}
  Except for the t-exactness of the horizontal functors, 
  this is a consequence of Theorem
  \ref{t:closed-embedding-goal} and the results of subsection
  \ref{sec:form-inters-cohom}. 

  It is obvious that the extension of scalars functor is t-exact. The
  t-exactness of $[-c]i^*$ can be proved as follows: 
  Since all strata in $\mathcal{S}$ are simply connected,
  the $(\IC_S)_{S \in \mathcal{S}}$ are the simple objects of
  $\Perv(X,\mathcal{S})$. Every object of $\Perv(X, \mathcal{S})$ has
  finite length. Now the isomorphisms \eqref{eq:nnpervstrat} and the
  long exact perverse cohomology sequence show the t-exactness.
\end{proof}

\subsection{Closed Embeddings and (Perverse) Sheaves}
\label{sec:closed-embeddings}

Let $i: (Y, \mathcal{T}) \ra (X, \mathcal{S})$ be a closed embedding
of stratified varieties.
Since $i_*$ is perverse t-exact \cite[1.3.17, 1.4.16]{BBD}, 
$\p i_*: \Perv(Y, \mathcal{T}) \ra \Perv(X, \mathcal{S})$
is exact and induces the 
functor $\p i_*$ in diagram
\begin{equation}
  \label{eq:ils-real}
  \xymatrix{
    {\dcb(\Perv(Y, \mathcal{T}))} \ar[rr]^-{\real_{Y,\mathcal{T}}} 
    \ar[d]^{\p i_*}
    & & {\dcb(Y, \mathcal{T})}
    \ar[d]^{i_*}\\
    {\dcb(\Perv(X, \mathcal{S}))} \ar[rr]^-{\real_{X,\mathcal{S}}} 
    & & {\dcb(X, \mathcal{S}).}
  }
\end{equation}
The horizontal functors were introduced in \eqref{eq:real-X-S}.

\begin{proposition}
  \label{p:ils-real}
  Keep the above assumptions. Then diagram \eqref{eq:ils-real}
  commutes up to natural isomorphism.
\end{proposition}

\begin{proof}
  This follows from the definition of the 
  functor $\real_{X,\mathcal{S}}$ given in 
  \cite[3.1]{BBD}; for details see \cite[3.3, 3.4]{OSdiss}.
\end{proof}

Let $i:(Y, \mathcal{T}) \hra (X, \mathcal{S})$ be a closed embedding
of cell-stratified varieties. 
Theorem \ref{t:bgsequi} shows that the
right exact functor $\p i^*: \Perv(X, \mathcal{S}) \ra \Perv(Y,
\mathcal{T})$ has a left derived functor $L \p i^*$ between the bounded derived categories.

\begin{proposition}
  \label{p:real-ius}
  Let $i:(Y, \mathcal{T}) \hra (X, \mathcal{S})$ be a closed embedding
  of cell-stratified varieties.
  Then there exists a natural isomorphism 
  as indicated in the diagram
  \begin{equation*}
    \xymatrix{
      {\dcb(\Perv(X, \mathcal{S}))} \ar[rr]_-\sim^-{\real_{X,\mathcal{S}}}
      \ar[d]^{L \p i^*}      
      & & {\dcb(X, \mathcal{S})}
      \ar[d]^{i^*}\\
      {\dcb(\Perv(Y, \mathcal{T}))} \ar[rr]_-\sim^-{\real_{Y,\mathcal{T}}}
      \ar@{=>}[urr]^\sim
      & & {\dcb(Y, \mathcal{T}).}
    }
  \end{equation*}
\end{proposition}

\begin{proof}
  Both realization functors are equivalences of categories (Theorem
  \ref{t:bgsequi}) and up to these equivalences the functors 
  $\p i_*$ and $i_*$ coincide (Proposition \ref{p:ils-real}).
  Now the statement is a consequence of the adjunctions
  adjunctions $(L\p i^*, \p i_*)$ and $(i^*, i_*)$. 
\end{proof}

\subsection{Tensor Product with a DG Bimodule}
\label{sec:tensor-product-with}

By \cite[8.1.1, 8.1.2, 8.2.5]{Keller-construction-of-triangle-equiv},
there is a fully faithful functor $p: \dgDer \ra \dgHot$ that is left
adjoint to the quotient functor $q: \dgHot \ra \dgDer$ and has image
in $\dgHotproj$. If we consider $p$ as a functor $\dgDer \ra
\dgHotproj$, it is quasi-inverse to the equivalence
\eqref{eq:dgHotp-equiv-dgDer}. 

Let $\mathcal{A}$ and $\mathcal{B}$ be dg algebras and $X$ a dg
$\mathcal{A}$-$\mathcal{B}$-bimodule (with $\mathcal{A}$ acting on the
left and $\mathcal{B}$ on the right). This yields a triangulated functor
\begin{equation*}
  (? \otimes_{\mathcal{A}} X) : \dgHot(\mathcal{A}) \ra \dgHot(\mathcal{B}).  
\end{equation*}
Its left derived functor is the pair 
$((? \Lotimes_{\mathcal{A}} X), \sigma)$ 
(we use the definition of derived functors from 
\cite[C.D.II.2.1.2, p. 301]{SGA4einhalb}), where 
\begin{equation*}
  (? \Lotimes_{\mathcal{A}} X) := q \comp (? \otimes_{\mathcal{A}} X)
\comp p : \dgDer(\mathcal{A}) \ra \dgDer(\mathcal{B})
\end{equation*}
and $\sigma$ is the natural transformation 
\begin{equation}
  \label{eq:tensor-trafo}
  \sigma:  (? \Lotimes_{\mathcal{A}} X) \comp q \ra q \comp (?
  \otimes_{\mathcal{A}} X)
\end{equation}
coming from the adjunction $(p,q)$.

\subsection{Passage from Geometry to DG Modules}
\label{sec:pass-from-geom}

Let $I: \mathcal{A} \ra \mathcal{B}$ be a right exact functor between
abelian categories. We denote the induced functor 
$\Hot^\bd(\mathcal{A}) \ra \Hot^\bd(\mathcal{B})$ by the same symbol. 
Assume that each object of $\mathcal{A}$ has a projective resolution
of finite length.
Then $I$ has a left derived functor $LI: \Der^\bd(\mathcal{A}) \ra
\Der^\bd(\mathcal{B})$. 
Let $P$ and $Q$ be bounded complexes of projective objects in $\mathcal{A}$
and $\mathcal{B}$ respectively. 
Then $\cHom(Q, I(P))$ is obviously a dg $\cEnd(P)$-$\cEnd(Q)$-bimodule.
It induces (see subsection \ref{sec:tensor-product-with}) the lower
horizontal arrow in diagram
\begin{equation}
  \label{eq:LI-Hom}
  \xymatrix{
    {\Der^\bd(\mathcal{A})} \ar[rrr]^{LI} \ar[d]^{\cHom(P,?)} 
    &&& {\Der^\bd(\mathcal{B})}  \ar[d]^{\cHom(Q,?)} \\
    {\dgDer(\cEnd(P))} \ar[rrr]_{?\Lotimesover{\cEnd(P)}\cHom(Q,I(P))} 
    \ar@{=>}[rrru]^\pi
    &&& {\dgDer(\cEnd(Q)).} 
  }
\end{equation}
The vertical functors are restrictions of the functors \eqref{eq:chomp} explained in 
Remark \ref{rem:hom-p-dot-iso}.
We construct now a natural transformation 
$\pi$
as indicated in the diagram.
  Since $\cHom(Q,?) \comp LI$ is the left derived functor of
  \begin{equation*}
    \cHom(Q, ?) \comp I: \Hot^\bd(\mathcal{A}) \ra \dgHot(\cEnd(Q)),
  \end{equation*}
  it is enough, by the universal property of left derived functors, to
  construct a natural transformation
  $\hat \pi$
  as indicated by the diagram
  \begin{equation*}
    \xymatrix{
      {\Hot^\bd(\mathcal{A})} \ar[rrr]^{I} \ar[d]_{q} 
      &&& {\Hot^\bd(\mathcal{B})}   \ar[d]^{\cHom(Q,?)}\\
      {\Der^\bd(\mathcal{A})}  \ar[d]_{\cHom(P,?)} 
      &&& {\dgHot(\cEnd(Q))}  \ar[d]^{q}\\
      {\dgDer(\cEnd(P))} \ar[rrr]_{?\Lotimesover{\cEnd(P)}\cHom(Q,I(P))} 
      \ar@{=>}[rrruu]^{\hat\pi}
      &&& {\dgDer(\cEnd(Q)).}
    }
  \end{equation*}
  Thus we define $\hat\pi$ to be the composition
  \begin{align}
    \label{eq:trafo-for-pi}
    \notag \cHom(P,q(?)) \Lotimesover{\cEnd(P)} & \cHom(Q, I(P)) \\
    & = q(\cHom(P,?)) \Lotimesover{\cEnd(P)} \cHom(Q, I(P)) \\
    \notag & \xrightarrow{\text{$\sigma$ from \eqref{eq:tensor-trafo}}} q\Big(\cHom(P,?) \otimesover{\cEnd(P)} \cHom(Q, I(P))\Big) \\
    \notag & \xra{\text{$I$ and composition}} q\big(\cHom(Q, I(?))\big).
  \end{align}
\begin{proposition}
  \label{p:LI-tensor}
  Assume in addition to the above assumptions that $[\decal]I(P) \cong Q$ in
  $\Der^\bd(\mathcal{B})$ for some integer $\decal$. Then diagram \eqref{eq:LI-Hom} induces the diagram
  \begin{equation}
    \label{eq:LI-Hom-Prae-modified}
    \xymatrix{
      {\PraeVerd(P, {\Der^\bd(\mathcal{A})})} \ar[rr]^{[\decal]LI} \ar[d]^{\cHom(P,?)} 
      && {\PraeVerd(Q, {\Der^\bd(\mathcal{B})})}  \ar[d]^{\cHom(Q,?)} \\
      {\dgPraeDer_{\cEnd(P)}(\cEnd(P))} \ar[rr]_{?\Lotimesover{\cEnd(P)}\cHom(Q,[\decal]I(P))} 
      \ar@{=>}[rru]^{\tilde{\pi}}_{\sim}
      && {\dgPraeDer_{\cEnd(Q)}(\cEnd(Q))} 
    }
  \end{equation}
  and $\tilde \pi$ is a natural isomorphism.
\end{proposition}
\begin{proof}
  We have $LI(P) \cong I(P)$
  in $\Der^\bd(\mathcal{B})$ and 
  $\cHom(Q,[\decal]I(P)) \cong
  \cEnd(Q)$ in $\dgDer(\cEnd(Q))$. 
  Thus, after replacing
  the horizontal functors in diagram \ref{eq:LI-Hom} by their
  composition with the shift $[\decal]$, this diagram restricts to 
  \eqref{eq:LI-Hom-Prae-modified}.

  In order to show that $\tilde{\pi}$ is a natural isomorphism, it is
  sufficient to check that $\tilde{\pi}_P$ or equivalently $\pi_P$ is
  an isomorphism. Since $\pi_P$ is
  obtained by plugging in $P$ in \eqref{eq:trafo-for-pi} ($P$ is a
  complex of projective objects), this follows
  from the obvious isomorphism 
  \begin{align*}
    \cHom(P,q(P)) \Lotimesover{\cEnd(P)} & \cHom(Q, I(P)) \\
    & = q\big(\cHom(P,P)) \otimes_{\cEnd(P)} \cHom(Q, I(P))\big) \\
    & \sira q(\cHom(Q, I(P))).
  \end{align*}
\end{proof}

\begin{corollary}
  \label{c:LI-tensor}
  Under the assumptions of subsection \ref{sec:goal-section}, 
  the second square in diagram \eqref{eq:the-big} commutes up to
  natural isomorphism.
\end{corollary}
\begin{proof}
  Take as $\mathcal{A}$ the category $\Perv(X, \mathcal{S})$ (using 
  Theorem \ref{t:bgsequi}), as
  $I: \mathcal{A} \ra \mathcal{B}$ the right exact
  functor
  $\p i^*: \Perv(X, \mathcal{S}) \ra \Perv(Y, \mathcal{T})$
  and as $P$ and $Q$ the complexes denoted by the same symbols in
  subsection \ref{sec:goal-section}. 
  By assumption, we have an isomorphism
  $[\decal]i^*(\tilde{M}) \sira \tilde{N}$ in $\dcb(\MHM(Y))$. 
  Hence $[\decal]i^*(\ul{M}) \sira \ul{N}$ in $\dcb(Y, \mathcal{T})$ or
  equivalently $[\decal]L\p i^*(M) \sira N$ in $\dcb(\Perv(Y,
  \mathcal{T}))$, by Proposition \ref{p:real-ius} and Theorem
  \ref{t:bgsequi}.
  But $N$ is isomorphic to $Q$ and $L\p i^*(M)$ is isomorphic to $\p i^*(P)$ since $P \ra M$ is a
  projective resolution in $\Perv(X, \mathcal{S})$.
\end{proof}

\subsection{DG Bimodules and Transformations}
\label{sec:dg-bimod-transf}

\begin{lemma}
  \label{l:prod-hopro}
  Let $B \ra A$ be a dga-morphism and $P$ a homotopically projective
  dg $B$-module. Then $P \otimes_{B} A$ is homotopically
  projective.
\end{lemma}
\begin{proof}
  Since $? \otimes_{B}A$ is left adjoint to the restriction of
  scalars functor, we see that
  $\Hom_{\dgHot(A)}(P \otimes_{B}A, ?)$ vanishes on acyclic dg
  $A$-modules.
\end{proof}

Assume that we are given dga-morphisms $\phi:B \ra A$ and $\psi:S \ra
R$. Let $M$ be a dg $A$-$R$-bimodule. 
By restriction of scalars we view $M$ as a dg
$B$-$S$-module. 
Let $\chi:N \ra M$ be a morphism of dg $B$-$S$-bimodules.
We denote this situation as follows.
\begin{equation}
  \label{eq:bimod-diagram}
  {\mathovalbox{B \curvearrowright {N} \curvearrowleft S}}
  \xra{\;(\phi, \; \chi, \;\psi)\;}
  {\mathovalbox{A \curvearrowright {M} \curvearrowleft R}}
\end{equation}
We get the following diagram
  \begin{equation}
    \label{eq:axrbys}
    \xymatrix{
      {\dgDer(B)} \ar[r]^-{? \Lotimesover{B}A} \ar[d]_-{? \Lotimesover{B}N}
      & {\dgDer(A)} \ar[d]^-{? \Lotimesover{A}M} \\
      {\dgDer(S)} \ar[r]_-{? \Lotimesover{S}R} \ar@{=>}[ur]^{\theta}
      & {\dgDer(R)}
    }
  \end{equation}
and construct now the indicated natural transformation
$\theta$.
  From Lemma \ref{l:prod-hopro} we see that the obvious transformation
  \begin{equation*}
    (? \Lotimesover{A}M) \comp (? \Lotimesover{B}A) \sira (? \Lotimesover{B}M)
  \end{equation*}
  is an isomorphism. So it is sufficient to define a transformation
  \begin{equation*}
    \tilde{\theta}:
    (? \Lotimesover{S}R) \comp (? \Lotimesover{B}N)
    \ra 
    (? \Lotimesover{B}M)
  \end{equation*}
  But there is an obvious natural transformation
  \begin{align}
    \label{eq:p}
    (? \Lotimesover{S}R) \comp (? \Lotimesover{B}N) \comp q
    & \xrightarrow{\text{{$\sigma$ from \eqref{eq:tensor-trafo}} twice}} 
    q \comp (? \otimesover{S}R) \comp (? \otimesover{B}N)\\
    \notag 
    & \xrightarrow{\quad\chi\quad} 
    q \comp (? \otimesover{B}M)
  \end{align}
  that induces, by the universal property of left derived functors,
  the transformation we want.

\begin{proposition}
  \label{p:dgprae-bimod}
  Keep the assumptions from above. Let $n \in N$ be an element such
  that the maps $f: S \ra N$, $s \mapsto ns$ and 
  $g: R \ra M$, $r \mapsto \chi(n)r$ are quasi-isomorphisms of dg modules
  (so $n \in \Cy(N)^0:=N^0\cap \Kern d_N$). Then
  diagram \eqref{eq:axrbys} restricts to 
  \begin{equation}
    \label{eq:axrbys-prae}
    \xymatrix{
      {\dgPraeDer_{B}(B)} \ar[r]^-{? \Lotimesover{B}A} \ar[d]_-{? \Lotimesover{B}N}
      & {\dgPraeDer_A(A)} \ar[d]^-{? \Lotimesover{A}M} \\
      {\dgPraeDer_S(S)} 
      \ar[r]_-{? \Lotimesover{S}R} \ar@{=>}[ur]^{{\theta}|}_{\sim}
      & {\dgPraeDer_R(R)}
    }
  \end{equation}
  and ${\theta}|$ is a natural isomorphism.
\end{proposition}
\begin{proof}
  Since $S \cong N$ and $R \cong M$ in $\dgDer$, diagram
  \eqref{eq:axrbys} restricts to \eqref{eq:axrbys-prae}.
  If $N$ is a dg $B$-module, then ${\theta}_N$ is obtained 
  by plugging in $p(N)$ in \eqref{eq:p}
  (up to an
  isomorphism coming from the adjunction isomorphism $N \sira q(p(N)$).
  In order to show that ${\theta}|$ is a natural isomorphism, it is
  enough to check that ${\theta}_B$ is an isomorphism.
  Since $B$ is homotopically projective, we may assume $p(B)=B$.
  Then ${\theta}_B$ is given by
  \begin{equation}
    \label{eq:pi-B-test}
    \big((B \Lotimesover{B}N) \Lotimesover{S}R\big)
    \sira 
    (q(N) \Lotimesover{S}R)  
    \ra 
    q \big(N \otimesover{S}R\big) 
    \xrightarrow{\chi} 
    q (M)
  \end{equation}
  We may assume that the adjunction morphism $p(q(N)) \ra N$ is given
  by $f:S \ra N$. Then the composition of the last two maps
  in \eqref{eq:pi-B-test} is identified with the isomorphism
  $q(g):q(R) \ra q(M)$ in $\dgDer(R)$.
\end{proof}

\subsection{DGG Bimodules}
\label{sec:dgg-bimod}

In subsection \ref{sec:diff-grad-grad} we considered dgg modules over a
dgg algebra $R$. We defined 
a functor 
$\Gamma: \dggMod(R) \ra \dggMod(\Gamma(R))$ (see \eqref{eq:def-gamma},
\eqref{eq:functor-gamma}) and used it to show that
dgg algebras with pure cohomology are formal. 

The construction of $\Gamma$ is easily extended to bimodules.
If $A$ and $B$ are dgg algebras and $M$ is a dgg $A$-$B$-bimodule,
then $\Gamma(M)$ becomes a dgg $\Gamma(A)$-$\Gamma(B)$-bimodule.
We get the following situation similar to \eqref{eq:bimod-diagram}.
\begin{equation}
  \label{eq:bimod-diagram-dgg-void-gamma}
  {\mathovalbox{\Gamma(A) \curvearrowright {\Gamma(M)} \curvearrowleft \Gamma(B)}}
  \;\; \subset \;\;
  {\mathovalbox{A \curvearrowright {M} \curvearrowleft B}}
\end{equation}
Here the inclusion $\Gamma(M) \subset M$ is a morphism of dgg
$\Gamma(A)$-$\Gamma(B)$-bimodules.
The cohomology $\Ho(M)$ is a dgg $\Ho(A)$-$\Ho(B)$-bimodule.
Assume now that the cohomologies of $A$, $B$ and $M$ vanish
in degrees $(i,j)$ with $i < j$. Then componentwise projection defines
the following morphisms of dgg algebras and dgg bimodules:
\begin{equation}
  \label{eq:bimod-diagram-dgg-gamma-coho}
  {\mathovalbox{\Gamma(A) \curvearrowright \Gamma(M) \curvearrowleft \Gamma(B)}}
  \xra{\;\;}
  {\mathovalbox{\Ho(A) \curvearrowright \Ho(M) \curvearrowleft \Ho(B)}}
\end{equation}
We would like to apply Proposition \ref{p:dgprae-bimod} to the
situations sketched in 
\eqref{eq:bimod-diagram-dgg-void-gamma} and
\eqref{eq:bimod-diagram-dgg-gamma-coho},
i.\,e.\ we need an element $m \in \Gamma(M)$ inducing quasi-isomorphisms
$\Gamma(B) \ra \Gamma(M)$, $B \ra M$ and $\Ho(B) \ra \Ho(M)$. 

\begin{lemma}
  \label{l:the-three-qisos}
  Let $B$ be a dgg algebra, $M$ a dgg $B$-module
  and $f:B \ra M$ 
  a quasi-isomorphism of (right) dgg $B$-modules.
  Then $m := f(1) \in \Gamma(M)^{00}$ and the multiplication maps
  $(m\cdot):B \ra M$, $(m\cdot):\Gamma(B)  \ra \Gamma(M)$, and
  $([m]\cdot):\Ho(B)   \ra \Ho(M)$
  are quasi-isomorphisms of dgg modules (over $B$, $\Gamma(B)$ and
  $\Ho(B)$), where we denote by $[m]$ the 
  class of $m$ in $\Ho(M)$.
\end{lemma}

\begin{proof}
  Since $1 \in B^{00}\cap \Kern d_B$, 
  we have $m=f(1) \in M^{00}\cap \Kern d_M =\Gamma(M)^{00}$.
  The functor $\Gamma$ is a ``truncation functor'' and maps
  quasi-isomorphisms to quasi-isomorphisms, so $\Gamma(f)$ is a
  quasi-isomorphism. But $f=(m\cdot)$, $\Gamma(f)=(m\cdot)$ and
  $\Ho(f)=([m]\cdot)$.
\end{proof}

\subsection{Triangulated Functors on Objects}
\label{sec:triang-funct-objects}

Let $\mathcal{A}$ be an abelian category.
The stupid truncation functors 
$\sigma_{\leq i}$, $\sigma_{\geq i}: \Ket(\mathcal{A}) \ra
\Ket(\mathcal{A})$, for $i \in \DZ$, are defined as follows:
$\sigma_{\leq i}$ preserves all components in degrees $\leq
i$ and replaces all components in degrees $>i$ by zero;
similarly for $\sigma_{\geq i}$.
There are obvious transformations $\id \ra \sigma_{\leq i}$ and
$\sigma_{\geq i} \ra \id$.

\begin{proposition}\label{p:compute-tri-fun}
  Let $\mathcal{A}$, $\mathcal{B}$ be abelian categories and
  $F:\Derb(\mathcal{A}) \ra \Derb(\mathcal{B})$ a triangulated
  functor. Let 
  \begin{equation*}
    \dots \ra 0 \ra P^{-n} \xra{d^{-n}} P^{-n+1} \ra
    \dots \ra P^{-1} \xra{d^{-1}} P^0 \xra{p} M \ra 0 \ra \dots
  \end{equation*}
  be a bounded exact complex in $\mathcal{A}$ (a resolution of
  $M$). Assume that $F(P^i)$ is an object of $\mathcal{B}$, 
  for all $i =-n, \dots, 0$, where we
  identify $\mathcal{B}$ with the heart of the standard t-structure on
  $\Derb(\mathcal{B})$. Let $\hat{F}(P)$ be the complex 
  \begin{equation*}
    \dots \ra 0 \ra F(P^{-n}) \xra{F(d^{-n})} F(P^{-n+1}) \ra
    \dots  \xra{F(d^{-1})} F(P^0) \ra 0 \ra \dots
  \end{equation*}
  in $\mathcal{B}$. Then $F(M)$ and $\hat{F}(P)$ 
  are isomorphic in $\Derb(\mathcal{B})$.
\end{proposition}
\begin{proof}
  We write $\Hom$ instead of $\Hom_{\Der(\mathcal{B})}$.
  The transformation $\sigma_{\geq 0} \ra \id$ yields 
  an inclusion $s:F(P^0) = \sigma_{\geq 0}(\hat{F}(P)) \ra
  \hat{F}(P)$ in $\Ket^\bd(\mathcal{B})$.
  By induction on $n$, we prove the
  following more precise statement: There is an isomorphism
  $\alpha \in \Hom(\hat{F}(P), F(M))$ such that 
  $\alpha \comp s = F(p)$ in $\Der(\mathcal{B})$.

  For $n=0$, this is obvious. Assume that $n \geq 1$.
  Consider the morphism 
  $f: [-1]\sigma_{\leq -1}(\hat{F}(P)) \ra \sigma_{\geq 0}
  (\hat{F}(P))=F(P^0)$ in $\Ket(\mathcal{B})$, given by $F(d^{-1})$ in degree zero.
  Its mapping cone is $\hat{F}(P)$, and we get a
  distinguished triangle 
  \begin{equation}
    \label{eq:ad-1}
    [-1]\sigma_{\leq -1}(\hat{F}(P)) \xra{f} F(P^0) \xra{s}
    \hat{F}(P) \xra{[1]} { }
  \end{equation}
  in $\Derb(\mathcal{B})$.
  Similarly, we get a distinguished triangle
  \begin{equation}
    \label{eq:ad-2}
    [-2]\sigma_{\leq -2}(\hat{F}(P)) \xra{g} F(P^{-1}) \xra{t} 
    [-1]\sigma_{\leq -1}(\hat{F}(P)) \xra{[1]} \quad
  \end{equation}
  in $\Derb(\mathcal{B})$, where $t$ is the obvious inclusion, defined similarly as $s$ above,
  and $g$ is a morphism in
  $\Ket(\mathcal{A})$, given by $F(d^{-2})$ in degree zero. 

  We factorize $d^{-1}:P^{-1}\ra P^0$ as 
  $P^{-1}\xra{a} K \xra{b} P^0$,    
  where $K= \Ker p = \Bild d^{-1}$. By induction, applied to the
  exact complex
  \begin{equation*}
    (\dots \ra 0 \ra P^{-n} \xra{d^{-n}} P^{-n+1} \ra
    \dots \ra P^{-2} \xra{d^{-2}} P^{-1} \xra{a} K \ra 0 \ra \dots),
  \end{equation*}
  there is an isomorphism 
  $\beta \in \Hom([-1]\sigma_{\leq -1}\hat{F}(P), F(K))$ such that
  $\beta \comp t = F(a)$ in $\Der(\mathcal{B})$.

  Consider now the diagram
  \begin{equation}
    \label{eq:partial-morph}
    \xymatrix{
      {[-1]\sigma_{\leq -1}(\hat{F}(P))} \ar[r]^-{f} \ar[d]^{\beta}_\sim & {F(P^0)} \ar[r]^s &
      {\hat{F}(P)} \ar[r]^-{[1]} & { }\\
      {F(K)} \ar[r]^{F(b)} & {F(P^0)} \ar[r]^{F(p)} \ar@{=}[u] & {F(M)} \ar[r]^-{[1]} &
      { }
    }
  \end{equation}
  Both rows are distinguished triangles, 
  the upper one is \eqref{eq:ad-1}, and
  the lower one comes
  from the short exact sequence $K \xra{b} P^0 \xra{p} M$. 
  We claim that this diagram is commutative. If this is the case, we
  can complete the partial morphism $(\beta, \id)$ of
  distinguished triangles in \eqref{eq:partial-morph} by some 
  morphism $\alpha \in \Hom(\hat{F}(P), F(M))$ to a
  morphism of distinguished triangles, and any such $\alpha$ is an isomorphism. 

  So let us show that $f = F(b) \comp \beta$. 
  If we apply $\Hom(?, F(P^0))$ to \eqref{eq:ad-2} and
  use $\Hom([-1]\sigma_{\leq -2}(\hat{F}(P)),
  F(P^0))=0$, we get an injection
  \begin{equation*}
    (? \comp t): \Hom([-1]\sigma_{\leq -1}(\hat{F}(P)), F(P^0)) 
    \hra 
    \Hom(F(P^{-1}), F(P^0)).
  \end{equation*}
  Hence it is enough to check the equality
  $f \comp t = F(b) \comp \beta \comp t$. 
  But 
  $F(b) \comp \beta \comp t = F(b) \comp F(a)=F(d^{-1}) = f \comp t$.
\end{proof}

\subsection{Restriction of Projective Objects}
\label{sec:restr-proj-objects}

Let $i:(Y, \mathcal{T}) \ra (X, \mathcal{S})$ be a closed embedding of cell-stratified varieties.
\begin{lemma}
  \label{l:i-perverse}
  If $V$ is an object of $\Perv(X, \mathcal{S})$
  with a finite filtration with standard subquotients, then
  $i^*(V)$ is in $\Perv(Y, \mathcal{T})$, so 
  $i^*(V)=\p i^*(V)$. 
  If $V$ is a projective object of $\Perv(X, \mathcal{S})$, then the
  restriction $i^*(V)$ is a projective object of $\Perv(Y, \mathcal{T})$.
\end{lemma}
\begin{proof}
  For standard objects $\Delta_S=l_{S!}([d_S]\ul{S})$,
  we have
  $i^*(\Delta_S) = \Delta_S$ if $S \in \mathcal{T}$ and 
  $i^*(\Delta_S)= 0$ otherwise.
  Thus $i^*(\Delta_S)$ is in $\Perv(Y, \mathcal{T})$.
  Since $\Perv(Y, \mathcal{T})$ is
  stable by extensions (\cite[1.3.6]{BBD}) this proves the first
  statement.
  The second statement follows from 
  Theorem \ref{t:bgsequi} and the fact that
  $\p i^*$ is left adjoint to the exact functor $\p i_*$.
\end{proof}

\begin{corollary}
  \label{c:i-perverse}
  If $\tilde V \in \MHM(X, \mathcal{S})$ is perverse-projective, then
  $i^*(\tilde{V})$ is an object of $\MHM(Y, \mathcal{T})$ and
  perverse-projective,
  where we consider $\MHM(Y, \mathcal{T})$ as the heart of the
  standard t-structure on $\dcb(\MHM(Y), \mathcal{T})$.
\end{corollary}

\begin{proof}
  Lemma \ref{l:i-perverse} shows that $\rera(i^*(\tilde{V}))$
  is a projective object of $\Perv(Y, \mathcal{T})$. This implies that
  $i^*(\tilde{V})$ is in $\MHM(Y, \mathcal{T})$, since 
  $\rat:\MHM(X) \ra \Perv(X)$ is exact and faithful.
\end{proof}

\begin{remark}
  \label{rem:p-restriction}
  We define
  $\p i^*(\tilde{V}) := i^*(\tilde{V})$
  for perverse-projective $\tilde{V}$ in $\MHM(X, \mathcal{S})$.
  This notation is justified by 
  $\rat(\p i^*(\tilde{V})) \cong \p i^*(\rat(\tilde{V}))$
  (cf.\ Corollary \ref{c:i-perverse}, Proposition \ref{p:real-ius}).
\end{remark}

In subsection \ref{sec:perverseprojective} we have defined
a mixed Hodge structure $\Ext^i_{\Perv(X)}(\tilde{M}, \tilde{N})$
on $\Ext^i_{\Perv(X)}(M,N)$, for 
objects 
$\tilde{M}$, $\tilde{N}$ of $\MHM(X)$.
The same definition also works for objects $\tilde{M}$, $\tilde{N}$ of
$\dcb(\MHM(X))$. 
Consider the obvious composition in $\dcb(\MHM(X))$ provided by several
adjunction morphisms
\begin{equation*}
  \sHom(\tilde{M}, \tilde{N}) 
  \ra i_* i^* \sHom(\tilde{M}, \tilde{N}) 
  \ra i_* \sHom(i^*(\tilde{M}), i^*(\tilde{N})).
\end{equation*}
We take hypercohomology and obtain morphisms of (polarizable) mixed
Hodge structures
\begin{equation*}
  \Ext^j_{\Perv(X)}(\tilde{M}, \tilde{N})
  \ra 
  \Ext^j_{\Perv(Y)}(i^*(\tilde{M}), i^*(\tilde{N})).
\end{equation*}
These morphisms are natural in $\tilde{M}$ and $\tilde{N}$.
In particular, if $\tilde{P}$ and $\tilde{Q}$ are smooth
perverse-projective Hodge sheaves, then $i^*(\tilde{P})$ and
$i^*(\tilde{Q})$ are smooth perverse-projective Hodge sheaves and we
get a morphism 
\begin{equation}
  \label{eq:mhs-hom-rest}
  \Hom_{\Perv(X,\mathcal{S})}(\tilde{P}, \tilde{Q})
  \ra 
  \Hom_{\Perv(Y,\mathcal{T})}(\p i^*(\tilde{P}), \p i^*(\tilde{Q}))
\end{equation}
of mixed Hodge structures (see Corollary \ref{c:i-perverse} and Remark
\ref{rem:p-restriction}).

\subsection{Passage to Cohomology Algebras}
\label{sec:pass-cohom-algebr}

We combine our results in order to prove that the third and fourth
square in diagram \eqref{eq:the-big} commute up to natural
isomorphism.

Assume that we are in the setting described in subsection
\ref{sec:goal-section} (with $I$ a singleton).
So we are given a perverse-projective resolution of finite length
$\tilde{P} \ra \tilde{M}$.
Let 
\begin{equation*}
  \p i^* (\tilde{P}) := (\dots \ra \dots  \ra \p i^*(\tilde{P}^{-1}) \ra \p i^*(\tilde{P}^0) \ra 0 \ra \dots)
\end{equation*}
be the complex obtained by applying $\p i^*$ to $\tilde{P}$
(cf. Corollary \ref{c:i-perverse} and Remark
\ref{rem:p-restriction}). We may and will assume that this complex 
$\p i^*(\tilde P)$ is a complex in $\MHM(Y, \mathcal{T})$.
The underlying complex of smooth projective perverse sheaves is
denoted by $\p i^*(P)$.

The definition of the complex 
$\tilde{A} = \cEnd(\tilde{P})$
in subsection \ref{sec:comparing} and the comments around
\eqref{eq:mhs-hom-rest} at the end of subsection
\ref{sec:restr-proj-objects} show that there is a morphism of dg
algebras ``of mixed Hodge structures''
\begin{equation}
  \label{eq:atildeaufV}
  \tilde{A} =
  \cEnd(\tilde{P})
  \ra
  \cEnd(\p i^*(\tilde{P}))
  =\cEnd([\decal]\p i^*(\tilde{P})).
\end{equation}
Recall that $\tilde{B}=\cEnd(\tilde{Q})$.
Hence the dg $\cEnd([\decal]\p i^*(\tilde{P})$-$\cEnd(\tilde{Q})$-bimodule
\begin{equation*}
  \tilde{V}:=\cHom(\tilde{Q}, [\decal]\p i^*(\tilde{P}))  
\end{equation*}
becomes
a dg $\tilde A$-$\tilde B$-bimodule.
Note that $\tilde{A}$, $\tilde{B}$ and 
$\tilde{V}$
are complexes of mixed Hodge
structures, and the differentials, multiplications and operations are
morphisms of mixed Hodge structures.

We apply the tensor functors $\omega_0$, $\gr^W_{\DR}$, and
$\omega_W=\eta \comp \gr^W_{\DR}$ from subsection
\ref{sec:formality-of-some-dga} (cf.\ diagram \eqref{eq:omega-eta})
to $\tilde A$, $\tilde B$ and the $\tilde{A}$-$\tilde{B}$-bimodule
$\tilde{V}$ and call the obtained dg(g) algebras and bimodules as
shown here: 
\begin{equation}
  \label{eq:the-bimodules}
  \xymatrix{
    {\mathovalbox{{\tilde A} \curvearrowright {\tilde{V}} \curvearrowleft {\tilde B}}}
    \ar@{|->}[d]^{\omega_0} \ar@{|->}[rd]^{\omega_W} \ar@{|->}[r]^{\gr^W_{\DR}}
    &
    {\mathovalbox{{\tilde R} \curvearrowright {\tilde{W}} \curvearrowleft {\tilde S}}}
    \ar@{|->}[d]^{\eta}\\
    {\mathovalbox{{A} \curvearrowright {{V}} \curvearrowleft {B}}}
    \ar[r]^a_{\sim}
    & 
    {\mathovalbox{{R} \curvearrowright {{W}} \curvearrowleft {S}}.}
  }
\end{equation}
The isomorphism of dg algebras and bimodules indicated by the
lower horizontal arrow comes from the natural isomorphism
\eqref{eq:a-omega-iso}.

\begin{proposition}
  \label{p:tildef}
  There is a quasi-isomorphism $\tilde f: \tilde{S} \ra \tilde{W}$ of (right)
  dgg $\tilde{S}$-modules.
\end{proposition}

\begin{proof}
  By Proposition \ref{p:compute-tri-fun} (using Corollary \ref{c:i-perverse}) and assumption
  \eqref{eq:assumption} we have 
  isomorphisms
  \begin{equation*}
    [\decal]\p i^*(\tilde P) \sira [\decal]i^*(\tilde M) \sira \tilde N
  \end{equation*}
  in $\dcb(\MHM(Y))$; recall $\tilde N \in \MHM(Y,
  \mathcal{T})$. Hence $\Ho^j([\decal]\p i^*(\tilde P))$ 
  vanishes for $j \not= 0$ and we get a sequence 
  \begin{equation*}
    \tilde{Q} \ra \tilde{N} \sila 
    \Ho^0([\decal]\p i^*(\tilde P) 
    \la \tau_{\leq 0}([\decal]\p i^*(\tilde P)) 
    \ra [\decal]\p i^*(\tilde P) 
  \end{equation*}
  of quasi-isomorphisms
  in $\Ket^\bd(\MHM(X, \mathcal{S}))$, 
  where $\tau_{\leq 0}$ is the intelligent truncation functor as
  defined, for example, in \cite[1.3]{KS}. We apply $\cHom(\tilde{Q},?)$
  to this sequence and obtain, using Lemma \ref{l:qiso-hom}, a sequence
  of quasi-isomorphisms of 
  $\tilde B$-modules connecting
  \begin{equation*}
    \tilde B = \cHom(\tilde{Q},\tilde{Q}) \text{ and }
    \tilde V = \cHom(\tilde{Q},[\decal]\p i^*(\tilde P)).
  \end{equation*}
  Hence $\Ho(\tilde S)$ and $\Ho(\tilde W)$ are isomorphic as dgg
  $\Ho(\tilde S)$-modules. Choose $w \in \Cy(\tilde W)^{00}$
  such that $\Ho(\tilde S) \ra \Ho(\tilde W)$, $s \mapsto [w]s$, is an
  isomorphism. Then
  $\tilde f: \tilde S \ra \tilde{W}$, $s \mapsto ws$,
  is a quasi-isomorphism of dgg $\tilde S$-modules.
\end{proof}

Lemma \ref{l:the-three-qisos} shows that multiplication by $w:=\tilde{f}(1)\in
\Gamma(\tilde W)^{00}$ 
defines quasi-isomorphisms $\tilde S \ra \tilde W$, $\Gamma(\tilde S)
\ra \Gamma(\tilde W)$ and $\Ho(\tilde S) \ra \Ho(\tilde W)$ of dgg
modules.

We apply $\eta$ to the morphisms of dgg algebras and bimodules
\begin{equation*}
  {\mathovalbox{\tilde{R} \curvearrowright {\tilde{W}} \curvearrowleft \tilde{S}}}
  \supset 
  {\mathovalbox{{\Gamma(\tilde R)} \curvearrowright {{\Gamma(\tilde W)}}
      \curvearrowleft {\Gamma(\tilde S)}}}
\end{equation*}
and denote the resulting situation by
\begin{equation}
  \label{eq:bimod-diagram-dgg-void-gamma-application}
  {\mathovalbox{{R} \curvearrowright {{W}} \curvearrowleft {S}}}
  \supset 
  {\mathovalbox{{\Gamma(R)} \curvearrowright {{\Gamma(W)}}
      \curvearrowleft {\Gamma(S)}}}.
\end{equation}
Multiplication by $w\in \Cy(\Gamma(W))^0$ 
still defines quasi-isomorphisms $S \ra W$ and $\Gamma(S) \ra
\Gamma(W)$ of dg modules, hence we can apply 
Proposition \ref{p:dgprae-bimod} and obtain a natural isomorphism
\begin{equation}
  \label{eq:axrbys-prae-one}
  \xymatrix{
    {\dgPraeDer_{\Gamma(R)}(\Gamma(R))} \ar[r]^-{? \Lotimesover{\Gamma(R)}R} \ar[d]_-{? \Lotimesover{\Gamma(R)}\Gamma(W)}
    & {\dgPraeDer_R(R)} \ar[d]^-{? \Lotimesover{R}W} \\
    {\dgPraeDer_{\Gamma(S)}(\Gamma(S))} 
    \ar[r]_-{? \Lotimesover{\Gamma(S)}S} \ar@{=>}[ur]_\sim
    & {\dgPraeDer_S(S).}
  }
\end{equation}

Since the cohomologies $\Ho(\tilde{R})$, $\Ho(\tilde{S})$ 
and hence $\Ho(\tilde{W})$ are
pure of weight zero (as shown in the proof of Theorem
\ref{t:formality-of-endo-complex}), componentwise projection defines well-defined morphisms of
dgg algebras and modules
\begin{equation*}
  {\mathovalbox{{\Gamma(\tilde R)} \curvearrowright {{\Gamma(\tilde W)}}
      \curvearrowleft {\Gamma(\tilde S)}}}
  \ra
  {\mathovalbox{{\Ho(\tilde R)} \curvearrowright {{\Ho(\tilde W)}}
      \curvearrowleft {\Ho(\tilde S)}}}
\end{equation*}
with underlying morphisms of dg algebras and modules
\begin{equation*}
  {\mathovalbox{{\Gamma(R)} \curvearrowright {{\Gamma(W)}}
      \curvearrowleft {\Gamma(S)}}}
  \ra
  {\mathovalbox{{\Ho(R)} \curvearrowright {{\Ho(W)}}
      \curvearrowleft {\Ho(S)}}}.
\end{equation*}
Multiplication by $w\in \Cy(\Gamma(W))^0$ and its class $[w] \in
\Ho(W)^0$ defines quasi-iso\-mor\-phisms $\Gamma(S) \ra \Gamma(W)$ and
$\Ho(S) \ra \Ho(W)$ of dg modules, so application of
Proposition \ref{p:dgprae-bimod} yields
a natural isomorphism
\begin{equation}
  \label{eq:axrbys-prae-two}
  \xymatrix{
    {\dgPraeDer_{{\Gamma(R)}}({\Gamma(R)})} \ar[r]^-{? \Lotimesover{{\Gamma(R)}}{\Ho(R)}} \ar[d]_-{? \Lotimesover{{\Gamma(R)}}{\Gamma(W)}}
    & {\dgPraeDer_{\Ho(R)}({\Ho(R)})} \ar[d]^-{? \Lotimesover{{\Ho(R)}}{\Ho(W)}} \\
    {\dgPraeDer_{\Gamma(S)}({\Gamma(S)})} 
    \ar[r]_-{? \Lotimesover{{\Gamma(S)}}{\Ho(S)}} \ar@{=>}[ur]_{\sim}
    & {\dgPraeDer_{\Ho(S)}({\Ho(S)})}
  }
\end{equation}

Let 
\begin{equation*}
  {\mathovalbox{{A} \curvearrowright {{V}} \curvearrowleft {B}}}
  \supset 
  {\mathovalbox{{\Sub(A)} \curvearrowright {{\Sub(V)}}
      \curvearrowleft {\Sub(B)}}}
\end{equation*}
be the inverse image of \eqref{eq:bimod-diagram-dgg-void-gamma-application} 
under
the isomorphism $a$ in \eqref{eq:the-bimodules}
(cf.\ diagram \eqref{eq:sub-def} in the proof of Theorem
\ref{t:formality-of-endo-complex}).
Then diagrams \eqref{eq:axrbys-prae-one} and
\eqref{eq:axrbys-prae-two}
get transformed in the third and forth square in diagram
\eqref{eq:the-big}.

\subsection{Passage to Extension Algebras}
\label{sec:pass-geom-extens}

We prove now that the fifth square in diagram \eqref{eq:the-big}
commutes up to natural isomorphism. The setting is as in subsection
\ref{sec:pass-cohom-algebr}.
Recall that $M=\rat(\tilde M)$ and $\ul{M}=\real(M)=\rera(\tilde M)$
and similarly for $\tilde N$.

We define in the following isomorphisms of dg algebras and bimodules 
(with all differentials equal to zero)
\begin{align}
  \label{eq:hee}
  \notag
  {\mathovalbox{{\Ho(A)} \curvearrowright {{\Ho(V)}}
      \curvearrowleft {\Ho(B)}}} 
  \xra{(\phi, \chi, \psi)}
  & \; {\mathovalbox{{\Ext(M)} \curvearrowright {{\Ext(N)}}
      \curvearrowleft {\Ext(N)}}}\\
  \xra{(\real, \real, \real)}  
  & \; {\mathovalbox{{\Ext(\ul M)} \curvearrowright {{\Ext(\ul N)}}
      \curvearrowleft {\Ext(\ul N)}}},
\end{align}
where we omit some indices ${\Perv(X, \mathcal{S})}$ and ${\Perv(Y,
  \mathcal{T})}$ in the second box.
The right module structures on these bimodules and the isomorphisms
$\real$ 
are the obvious ones. 
For the definition of the left module structures and the morphisms
$\phi$, $\chi$ and $\psi$ we have to recall and
establish several isomorphisms.
(The left module structures on $\Ho(V)$ and $\Ext(\ul M)$ were already
defined, but we repeat the definition.)

There is an isomorphism
$\tilde\sigma:[\decal]i^*(\tilde{M}) \sira \tilde{N}$
in $\dcb(\MHM(Y))$ 
by assumption 
(see \eqref{eq:assumption}). Define $\ul{\sigma}:=\rera(\tilde\sigma)$ and let
$\tau$ be the natural isomorphism from Proposition
\ref{p:real-ius}. We obtain isomorphisms
\begin{equation*}
  [\decal]\ul{L\p i^*(M)} \xra{[\decal]\tau_M} [\decal]i^*(\ul M)
  \xra{\ul \sigma} \ul N.
\end{equation*}
The equivalence $\real:\dcb(\Perv(X, \mathcal{S})) \ra \dcb(X,
\mathcal{S})$ shows that there is a unique isomorphism
$\lambda: [\decal]{L\p i^*(M)} \sira N$
in $\dcb(\Perv(X, \mathcal{S}))$ such that $\real(\lambda)= \ul{\sigma}
\comp [\decal]\tau_M$.

We denote the perverse-projective resolutions $\tilde P
\ra \tilde M$ and $\tilde Q \ra \tilde N$ by $\tilde \pi$ and $\tilde
\rho$, and their underlying projective resolutions as $\pi:P \ra M$
and $\rho:Q \ra N$.
Since $\pi$ is a projective resolution, we may assume that 
$L\p i^*(P)$ and $L\p i^*(M)$ are identical to $\p i^*(P)$ and that 
$L\p i^*(\pi)$ is the identity. Hence we may consider $\lambda$
also as an isomorphism
\begin{equation*}
  \lambda:[\decal]\p i^*(P) \sira N.
\end{equation*}

Instead of
$\bigoplus \Hom^n_{\Hot(\Perv(X, \mathcal{S}))}$ and
$\bigoplus \Hom^n_{\Der(\Perv(X, \mathcal{S}))}$
we write $\Hom_{\Hot}$ and $\Hom_{\Der}$, and similarly for $(Y, \mathcal{T})$.
The above isomorphisms give rise to dga-mor\-phisms
\begin{gather*}
  \Ho(A)=\Hom_{\Hot}(P) 
  \xra{[\decal]\p i^*}
  \Hom_{\Hot}([\decal]\p i^*(P)), \\
  \Ext(M) 
  \xra{[\decal]L\p i^*}
  \Ext([\decal]L\p i^*(M)
  \xrightarrow[\sim]{\lambda}
  \Ext(N),\\
  \Ext(\ul M) \xra{[\decal]i^*} \Ext([\decal]i^*(\ul M))
  \xrightarrow[\sim]{\ul \sigma} \Ext(\ul N).
\end{gather*}
The first morphism comes from \eqref{eq:atildeaufV}, the last one coincides with \eqref{eq:extmn}.
These morphisms and the obvious left multiplications define the left
operations shown in \eqref{eq:hee}. 
The morphism $\phi$ is the composition 
\begin{equation*}
  \Ho(A)=\Hom_{\Hot}(P) 
  {=}
  \Hom_{\Der}(P) 
  =\Ext(P)
  \xrightarrow[\sim]{\pi}
  \Ext(M),
\end{equation*}
$\psi$ is defined analogously, and $\chi$ is given by
\begin{equation*}
  \Ho(V)=\Hom_{\Hot}(Q, [\decal]\p i^*P)
  = \Ext(Q, [\decal]\p i^*(P))
  \xrightarrow[\sim]{\rho, \lambda}
  \Ext(N).
\end{equation*}

It is easy to check that all morphisms in \eqref{eq:hee} are isomorphisms of dg algebras and dg
bimodules respectively. We apply Proposition \ref{p:dgprae-bimod} to the situation 
\eqref{eq:hee} and the element $n \in \Ho(V)$ corresponding to $\id
\in \Ext(\ul{N})$ and obtain 
the commutativity (up to natural isomorphism) of the fifth square in
diagram \eqref{eq:the-big}.

\section{Inverse Limits}
\label{sec:inverse-limits}

\subsection{Inverse Limits of Categories}
\label{sec:inverse-limits-categ}

We exhibit a 
definition of inverse limit of a sequence of categories that
will enable us to consider inverse limits of dg categories (subsection
\ref{sec:limits}) and to obtain a description of the equivariant
derived category (subsection \ref{sec:equiv-deriv-categ}).

Let $\mathcal{C}_0 \xla{F_0}
\mathcal{C}_1 
\la
\dots \la
\mathcal{C}_n \xla{F_n}
\mathcal{C}_{n+1} \la
\dots$
or in short $((\mathcal{C}_n), (F_n))$
be a sequence of categories (and functors).
We call the following category the inverse limit of this sequence and
denote it by $\invlim \mathcal{C}_n$:
\begin{itemize}
\item Objects are sequences 
  $((M_n)_{n \in \DN}, (\phi_n)_{n \in \DN})$
  of objects $M_n$ in $\mathcal{C}_n$ and isomorphisms
  $\phi_n:F_n(M_{n+1})\sira M_n$.
\item Morphisms $\alpha: ((M_n), (\phi_n)) \ra ((N_n), (\psi_n))$ are sequences $(\alpha_n)_{n \in
    \DN}$ of morphisms $\alpha_n:M_n \ra N_n$ such that
  $\psi_n \comp F_n(\alpha_{n+1}) = \alpha_n \comp \phi_n$, for all $n \in \DN$.
\end{itemize}

\begin{lemma}
  \label{l:limit-equiv}
  Let $N \in \DN$ and 
  assume that $F_n: \mathcal{C}_{n+1} \ra \mathcal{C}_n$ is an
  equivalence for all $n \geq N$. Then the obvious projection functor
  $    \pr_N: \invlim \mathcal{C}_n \ra \mathcal{C}_N$
  is an equivalence.
\end{lemma}

\begin{proof}
  Obvious.
\end{proof}

A morphism of sequences 
$((\mathcal{C}_n), (F_n))$ and 
$((\mathcal{D}_n), (G_n))$ of categories is a sequence
$\nu=(\nu_n)$ of functors $\nu_n: \mathcal{C}_n \ra \mathcal{D}_n$ such
that
$\nu_n \comp F_n$ and $G_n \comp \nu_{n+1}$ coincide (up to natural
isomorphism) for each $n \in \DN$.
Any such
morphism $\nu$ obviously defines a functor 
$\invlim \nu_n : \invlim \mathcal{C}_n \ra \invlim \mathcal{D}_n$.

In the following, we describe a setting in which this functor is an
equivalence.
Let $(\mathcal{I}, \leq)$ be a directed (i.\,e.\ for all $I$, $J
\in \mathcal{I}$ there is $K \in \mathcal{I}$ with $I\leq K$, $J \leq
K$) partially ordered set
(e.\ g.\ the set of segments in $\DZ$, partially ordered by inclusion).
An $\mathcal{I}$-filtered category
is a category $\mathcal{C}$ together with strict full subcategories
$(\mathcal{C}^I)_{I \in \mathcal{I}}$ such that 
$\mathcal{C}^I \subset \mathcal{C}^J$ for $I \leq J$.
We say that $\mathcal{C}$ is the union of the $\mathcal{C}^I$ if 
any object of $\mathcal{C}$ is contained in some
$\mathcal{C}^I$. 
A morphism 
$\mathcal{C} \ra \mathcal{D}$
of $\mathcal{I}$-filtered categories ($\mathcal{I}$-filtered
morphism) is a functor $F:\mathcal{C} 
\ra \mathcal{D}$ inducing functors $F^I: \mathcal{C}^I \ra
\mathcal{D}^I$ for all $I \in \mathcal{I}$.

If $((\mathcal{C}_n), (F_n))$ is a sequence of
$\mathcal{I}$-filtered categories (and $\mathcal{I}$-filtered
morphisms), the inverse limit $\invlim 
\mathcal{C}_n$ is filtered by the $\invlim \mathcal{C}_n^I$.
We will use the following conditions on a sequence of $\mathcal{I}$-filtered
categories $((\mathcal{C}_n), (F_n))$.
\begin{enumerate}[label={(F\arabic*)}]
\item
\label{enum:becomes-equiv}
 For each $I \in \mathcal{I}$ there is $N_I \in \DN$ such
  that, for all $n \geq N_I$, $F_n^I:\mathcal{C}_{n+1}^I \ra
  \mathcal{C}_n^I$ is an equivalence.
\item
\label{enum:is-union}
 $\invlim \mathcal{C}_n$ is the union of the 
  $\invlim \mathcal{C}_n^I$.
\end{enumerate}

Any morphism $(\nu_n): ((\mathcal{C}_n), (F_n))
\ra ((\mathcal{D}_n), (G_n))$ of 
sequences of $\mathcal{I}$-filtered categories induces an
$\mathcal{I}$-filtered morphism
$  \invlim \nu_n:
  \invlim \mathcal{C}_n
  \ra
  \invlim \mathcal{D}_n$. 

\begin{proposition}
  \label{p:limit-filtfun}
  Let $(\nu_n): ((\mathcal{C}_n), (F_n))
  \ra ((\mathcal{D}_n), (G_n))$ be a morphism of sequences of
  $\mathcal{I}$-filtered categories and 
  assume that both sequences satisfy 
  condition \ref{enum:is-union} and that $((\mathcal{C}_n), (F_n))$ satisfies
  condition \ref{enum:becomes-equiv}. 
  If for all $I \in \mathcal{I}$ 
  there is $N \in \DN$ such that, for
  all $n \geq N$, 
  $\nu_n^I: \mathcal{C}^I_n \ra \mathcal{D}^I_n$ is an
  equivalence, then $\invlim \nu_n:\invlim \mathcal{C}_n \ra
  \invlim\mathcal{D}_n$ is an equivalence and $((\mathcal{D}_n), (G_n))$
  also satisfies condition \ref{enum:becomes-equiv}.
\end{proposition}

\begin{proof}
  Obviously $((\mathcal{D}_n), (G_n))$ also satisfies
  condition \ref{enum:becomes-equiv}.
  By condition \ref{enum:is-union} it is sufficient to show that each $\invlim
  \nu_n^I$ is an equivalence. 
  But this follows from condition \ref{enum:becomes-equiv} and Lemma
  \ref{l:limit-equiv}.
\end{proof}

Let 
$ \mathcal{T}_0 \xla{F_0}
  \mathcal{T}_1 \xla{F_1}
  \dots \xla{F_{n-1}}
  \mathcal{T}_n \xla{F_n}
  \dots$
be a sequence of triangulated categories and triangulated
functors.
Then $\invlim \mathcal{C}_n$ is obviously additive and the shift
functors of the various $\mathcal{T}_n$  
induce an obvious shift functor $[1]$ on $\invlim
\mathcal{T}_n$.
Assume that each $\mathcal{T}_n$ is an $\mathcal{I}$-filtered category
and that all functors $F_n$ are $\mathcal{I}$-filtered morphisms
(the $\mathcal{T}_n^I$ are not assumed to be stable under the
shift).

\begin{proposition}
  \label{p:limit-triangulated}
  Let $((\mathcal{T}_n), (F_n))$ as above satisfy conditions
  \ref{enum:becomes-equiv} and \ref{enum:is-union} and assume that each
  $\mathcal{T}_n^I$ is closed 
  under extensions in $\mathcal{T}_n$.
  Then there is a unique class $\mathcal{D}$ of triangles in $\invlim
  \mathcal{T}_n$ (considered as an additive category with
  shift functor $[1]$) such that $(\invlim \mathcal{T}_n, \mathcal{D})$ is
  a triangulated category and all projections 
  $\pr_i:\invlim \mathcal{T}_n \ra \mathcal{T}_i$ are
  triangulated ($i \in \DN$).
  A triangle $\Sigma$ is in $\mathcal{D}$ if and only if all 
  $\pr_i(\Sigma)$ are distinguished ($i \in \DN$). 
\end{proposition}

\begin{proof} (Cf.\ {\cite[2.5.2]{BL}}.)
  We denote by $\mathcal{E}$ the class of triangles $\Sigma$ in $\invlim
  \mathcal{T}_n$ such that all $\pr_i(\Sigma)$ are distinguished and
  prove that $(\invlim \mathcal{T}_n, \mathcal{E})$ is a 
  triangulated category.
  In all axioms of a
  triangulated category (\cite{verdier-these}), only a finite set
  $F$
  of
  objects is involved, and the existence of some objects and morphisms
  is asserted. So we may check these axioms in a suitable full
  subcategory $\invlim \mathcal{T}_n^I$ of $\invlim \mathcal{T}_n$
  containing all $[k]X$, for $X \in F$ and $k = -1, 0, 1$.
  But this subcategory 
  is equivalent to $\mathcal{T}_{N_I}^I$ by Lemma \ref{l:limit-equiv}.
  (The condition that $\mathcal{T}_{N_I}^I$ is closed under extensions
  is used for constructing a distinguished triangle with a given base.)

  If a class $\mathcal{D}$ of triangles satisfies the conditions of
  the proposition, then obviously $\mathcal{D} \subset \mathcal{E}$.
  If $\Sigma: X \xra{f} Y \ra Z \ra [1]X$ is in $\mathcal{E}$, there is a
  triangle $\Sigma': X\xra{f} Y \ra Z' \ra [1]X$ in $\mathcal{D}$. 
  All objects are in some $\invlim \mathcal{T}_n^I$. Since $\Sigma$
  and $\Sigma'$ become isomorphic under $\pr_{N_I}: \invlim
  \mathcal{T}_n^I \sira \mathcal{T}_{N_I}^I$, they are isomorphic in
  $\invlim \mathcal{T}_n$ and hence $\Sigma \in \mathcal{D}$.
\end{proof}

\begin{remark}
  \label{rem:generalize}
  We omit the 
  obvious generalization of Proposition \ref{p:limit-filtfun} 
  to $\mathcal{I}$-filtered
  triangulated categories.
\end{remark}

\subsection{Filtered DG Modules}
\label{sec:filtered-dg-modules}

Let $\mathcal{A}$ be a dg algebra satisfying the conditions
\ref{enum:form-pg}-\ref{enum:form-sdga}. 
We recall the definition of a certain equivalent subcategory of
$\dgPerDer(\mathcal{A})$. This subcategory will enable us to prove the
concise statement of Proposition \ref{p:equi-on-dgprae-I}.

Recall the $\mathcal{A}$-modules $\{\hat L_x\}_{x \in W}$ from
subsection \ref{sec:perfect-dg-modules}.
We consider the following full subcategory $\dgFiltDer(\mathcal{A})$
of $\dgPerDer(\mathcal{A})$:
Its objects are $\mathcal{A}$-modules
$M$ admitting a finite filtration 
$0 = F_0(M) \subset F_1(M) \subset \dots \subset
F_n(M) =M$ by dg submodules with subquotients
$F_{i}(M)/F_{i-1}(M) \cong \{l_i\}\hat L_{x_i}$ in
$\dgMod(\mathcal{A})$
for suitable $l_1 \geq l_2 \geq \dots \geq l_n$ and $x_i \in W$.

\begin{theorem}[\cite{OSdiss-perfect-dg-arXiv}]
  \label{t:form-filt-iso-perfect}
  If $\mathcal{A}$ is a dg algebra satisfying \ref{enum:form-pg}-\ref{enum:form-sdga},
  then the inclusion $\dgFiltDer(\mathcal{A}) \subset
  \dgPerDer(\mathcal{A})$ is an equivalence of categories.
  Any object of $\dgFiltDer(\mathcal{A})$ is homotopically projective.
  An object $M$ of $\dgFiltDer(\mathcal{A})$
  lies in
  $\dgPerDer^{\leq n}$ (in $\dgPerDer^{\geq n}$) if and only if $M$ is
  generated in degrees $\leq n$ (in degrees $\geq n$)
  as a graded $A$-module.
\end{theorem}

\subsection{Inverse Limits of Categories of DG Modules}
\label{sec:limits}
Let $\mathcal{A}_{\infty}$ be a positively graded dg algebra with
differential zero and $A_\infty^0$ isomorphic to a finite product of
division rings. If $\mathcal{A}_\infty$ is the inverse limit
of a sequence of dg algebras 
$(\mathcal{A}_n)_{n \in \DN}$ of the same type that stabilizes in
each degree,
we show that 
$\dgPerDer(\mathcal{A}_\infty)$ is the inverse limit of the categories
$\dgPerDer(\mathcal{A}_i)$. We first study the special case where only
two dg algebras are involved, and generalize afterwards.

\subsubsection{Special Case}
\label{sec:special-case}

Let $\mathcal{A}=(A=\bigoplus_{i \geq 0} A^i, d=0)$ be a positively
graded dg algebra with differential zero and $A^0=\prod_{x \in W}
e_xA^0$ a finite product
of division rings (here $e_x$ is the unit element of $e_xA^0$).
In particular, $\mathcal{A}$ satisfies the conditions
\ref{enum:form-pg}-\ref{enum:form-sdga}.
The $e_xA^0$ are up to isomorphism the simple $A^0$-modules, so
$\dgPerDer(\mathcal{A}) =\dgPraeDer_{\mathcal{A}}(\{e_x
\mathcal{A}\}_{x \in W})$ 
thanks to Theorem~\ref{t:t-structure-auf-perf}.
Let $\mathcal{B}$ be a dg algebra of the same type and 
$\phi: \mathcal{A} \ra \mathcal{B}$ a dga-morphism.
We assume that $\phi^0:A^0 \ra B^0$ is an isomorphism.
Hence $B^0=\prod e_xB^0$, where we write $e_x$ instead of $\phi(e_x)$.
The extension of scalars functor 
induces a triangulated functor 
\begin{equation*}
  \pro_{\mathcal{A}}^\mathcal{B}=(?\Lotimes_{\mathcal{A}}\mathcal{B}) :
  \dgPerDer(\mathcal{A}) \ra \dgPerDer(\mathcal{B}). 
\end{equation*}
Since every object of $\dgFiltDer(\mathcal{A})$ is homotopically
projective (Theorem \ref{t:form-filt-iso-perfect}) (and hence
``homotopically flat'') we can and will
assume in the following that this functor is given by
$M \mapsto M \otimes_{\mathcal{A}}\mathcal{B}$ on $\dgFiltDer(\mathcal{A})$.
Using Theorem \ref{t:form-filt-iso-perfect} it is then easy to see that
this functor is t-exact with respect to the
t-structures from Theorem \ref{t:t-structure-auf-perf}.

By a segment we mean in the following a non-empty bounded interval in $\DZ$.
If $I=[a,b]$ is a segment ($a$, $b \in \DZ$, $a \leq b$),
we define $|I|=b-a$ and  
$\dgPerDer^I=\dgPerDer^{\geq a}\cap \dgPerDer^{\leq b}$. 
If $\mathcal{I}$ is the poset of all segments in $\DZ$, partially ordered by
inclusion, 
then $\dgPerDer$ is an $\mathcal{I}$-filtered category and the union of the $\dgPerDer^I$.
The functor $\pro_{\mathcal{A}}^{\mathcal{B}}$ is 
a morphism of $\mathcal{I}$-filtered categories
(as defined in subsection \ref{sec:inverse-limits-categ}).

\begin{lemma}
  \label{l:I-iff-I}
  Let $\mathcal{X}$ be in $\dgPerDer(\mathcal{A})$. If $I$ is any
  segment, then
  $\mathcal{X}$ is in $\dgPerDer^I(\mathcal{A})$ if and only if
  $\pro_{\mathcal{A}}^{\mathcal{B}}(\mathcal{X})$ is in $\dgPerDer^I(\mathcal{B})$.
\end{lemma}
\begin{proof}
  This is a direct consequence of Theorem \ref{t:form-filt-iso-perfect} 
  and the assumption that $\phi^0:A^0 \ra B^0$ is
  an isomorphism.
\end{proof}

\begin{proposition}
  \label{p:equi-on-dgprae-I}
  Let
  $\phi:\mathcal{A} \ra \mathcal{B}$ be a morphism of dg algebras as
  above.  If $I$ is a segment and $\phi$ an isomorphism up to degree
  $|I|+2$
  (i.\,e.\ $\phi^i:A^i \ra B^i$ is an isomorphism for all $i \leq |I|+2$), then
  $\pro_\mathcal{A}^\mathcal{B}:
    \dgPerDer^I(\mathcal{A}) \ra \dgPerDer^I(\mathcal{B})$
  is an equivalence of categories.
\end{proposition}

\begin{proof}
  We may assume that $I=[0,b]$ for some $b \in \DN$.
  The statement of the proposition is true since
  morphisms, homotopies and differentials are defined and can be
  defined in small degrees. 
  But let us give the details.
  
  \textbf{$\pro_\mathcal{A}^\mathcal{B}$ is fully faithful:}
  Let $\mathcal{X}$, $\mathcal{Y}$ be in
  $\dgPerDer^I(\mathcal{A})$. We have to show that
  \begin{equation*}
    \pro_\mathcal{A}^\mathcal{B}: \Hom_{\dgDer(\mathcal{A})}(\mathcal{X}, \mathcal{Y})
    \ra \Hom_{\dgDer(\mathcal{B})}(\phi(\mathcal{X}), \phi(\mathcal{Y})) 
  \end{equation*}
  is an isomorphism,
  where we abbreviate $\phi(\mathcal{X})=
  \pro_\mathcal{A}^\mathcal{B}(\mathcal{X})$ and 
  $\phi(\mathcal{Y})=\pro_\mathcal{A}^\mathcal{B}(\mathcal{Y})$.
  We may assume 
  that $\mathcal{X}$ and $\mathcal{Y}$ 
  are in $\dgFiltDer$ (Theorem \ref{t:form-filt-iso-perfect}) and that
  \begin{align}
    \label{eq:form-x-y}
    X & = \{l_1\}e_{v_1}A  \oplus \{l_2\}e_{v_2}A  \oplus \dots
    \oplus \{l_s\}e_{v_s}A,\\
    Y & = \{m_1\}e_{w_1}A  \oplus \{m_2\}e_{w_2}A  \oplus \dots
    \oplus \{m_t\}e_{w_t}A 
    \notag
  \end{align}
  as graded $A$-modules, with $v_i$, $w_i \in W$, 
  $0 \leq -l_i \leq b$ and 
  $0 \leq -m_i \leq b$.
  Both $\phi(\mathcal{X})=\mathcal{X}\otimes_{\mathcal{A}}\mathcal{B}$ and $\phi(\mathcal{Y})$ are given by the
  right hand side of \eqref{eq:form-x-y}, if we replace $A$ by
  $B$ there. Since objects of $\dgFiltDer$ are homotopically
  projective (Theorem \ref{t:form-filt-iso-perfect}) it is sufficient
  to show that
  \begin{equation}
    \label{eq:pro-ff}
    \pro_\mathcal{A}^\mathcal{B}: \Hom_{\dgHot(\mathcal{A})}(\mathcal{X}, \mathcal{Y})
    \ra \Hom_{\dgHot(\mathcal{B})}(\phi(\mathcal{X}), \phi(\mathcal{Y})) 
  \end{equation}
  is an isomorphism.

  We have 
  \begin{equation*}
    \Hom_{\gMod(A)}(\{l\}e_vA, \{m\}e_wA)
    =e_wA^{m-l}e_v
  \end{equation*}
  for $v$, $w \in W$, $l$, $m \in \DZ$; here $\gMod(A)$ is the
  category of graded $A$-modules. 
  Since the differentials $d_X:X \ra \{1\}X$, $d_Y:Y \ra \{1\}Y$ are
  morphisms of graded $A$-modules (the differential of $\mathcal{A}$
  is zero), they are given by matrices $x$ and $y$ with entries in $A$.
  Similarly, morphisms $f \in \Hom_{\dgMod}(\mathcal{X},\mathcal{Y})$
  are matrices satisfying $yf=fx$, and homotopies $h:X \ra \{-1\}Y$
  are matrices.  
  Each entry of these matrices is homogeneous, more precisely, we have
  $x_{ij} \in e_{v_i}A^{l_i+1-l_j}e_{v_j}$,
  $y_{ij} \in e_{w_i}A^{m_i+1-m_j}e_{w_j}$,
  $f_{ij} \in e_{w_i}A^{m_i-l_j}e_{v_j}$, and
  $h_{ij} \in e_{w_i}A^{m_i-1-l_j}e_{v_j}$.
  The differential of $\phi(\mathcal{X})$ is given by the matrix
  $\phi(x)$.
  The functor $\pro_\mathcal{A}^\mathcal{B}:\dgMod(\mathcal{A}) \ra
  \dgMod(\mathcal{B})$ maps the matrix 
  $f=(f_{ij})$ to $\phi(f)=(\phi(f_{ij}))$,
  and similarly for homotopies. 

  Surjectivity of \eqref{eq:pro-ff}: Let $\tilde{f}$ be in
  $\Hom_{\dgMod}(\phi(\mathcal{X}), 
  \phi(\mathcal{Y}))$. 
  Since all entries of $\tilde{f}$ are of degree
  $\leq b \leq |I|+2$, there is a unique matrix $f$ such that
  $\phi(f)=\tilde{f}$ and
  $\deg(f_{ij})=\deg(\tilde{f}_{ij})$ for all $i$, $j$ (and $f_{ij}=0$ if $\tilde f_{ij}=0$).
  This $f$ defines an element of 
  $\Hom_{\dgMod}(\mathcal{X}, \mathcal{Y})$ if and only if the matrix
  equation $yf=fx$ holds. All summands in every entry of this
  equation have degree $\leq b+1 \leq |I|+2$, and $\phi$ is an isomorphism
  up to degree $|I|+2$. So it is enough to show that
  $\phi(y)\tilde{f}=\tilde{f}\phi(x)$. But this is true by assumption
  on $\tilde{f}$. 
  
  Injectivity of \eqref{eq:pro-ff}: Assume that $f$ in
  $\Hom_{\dgMod}(\mathcal{X}, \mathcal{Y})$ is 
  mapped to $\phi(f)=0$ in 
  $\Hom_{\dgHot(\mathcal{B})}(\phi(\mathcal{X}), \phi(\mathcal{Y}))$. 
  Then there is a homotopy $\tilde{h}:\phi(X)\ra\{-1\}\phi(Y)$ alias
  a matrix with entries in $B$, such that
  $\phi(f)=\tilde{h}\phi(x)+\phi(y)\tilde{h}$.
  Since all entries of $\tilde{h}$ are homogeneous of degree $\leq
  b-1\leq |I|+2$, there is a unique matrix $h$ with
  $\phi(h)=\tilde{h}$ and $\deg(h_{ij})=\deg(\tilde{h}_{ij})$.
  This matrix $h$ defines a homotopy between $f$ and $0$, because
  $\phi$ is an isomorphism up to degree $|I|+2$.

  \textbf{$\pro_\mathcal{A}^\mathcal{B}$ is dense:}
  Let $\widetilde{\mathcal{X}}$ be an object of $\dgPerDer^I(\mathcal{B})$.
  We may assume that $\widetilde{\mathcal{X}}$ is in $\dgFiltDer$ and has, as
  a graded $B$-module, the form
  \begin{equation*}
    \widetilde{X} = \{l_1\}e_{v_1} B \oplus \{l_2\}e_{v_2} B \oplus \dots
    \oplus \{l_s\}e_{v_s}B,
  \end{equation*}
  with $0 \leq -l_1 \leq -l_2 \leq \dots \leq -l_s \leq b$.
  The differential $d_{\widetilde{X}}$ is a matrix $\tilde{x}$, with
  all entries in $B$ of
  degree $\leq b+1\leq |I|+2$. Let $x$ be the unique matrix with entries
  in $A$ such that $\phi(x)=\tilde{x}$ and
  $\deg(x_{ij})=\deg(\tilde{x}_{ij})$ for all $i$, $j$.
  Define 
  \begin{equation*}
    X = \{l_1\}e_{v_1} A \oplus \{l_2\}e_{v_2} A \oplus \dots
    \oplus \{l_s\}e_{v_s}A.
  \end{equation*}
  The matrix $x$ defines a differential on $X$ if and only if $x^2=0$. 
  In this matrix equation, all summands have degree $\leq b+2 \leq |I|+2$.
  But $\phi(x^2)=\tilde{x}^2=0$ holds, and $\phi$ is an isomorphism in
  degrees $\leq |I|+2$. Hence $(X,x)$ is in $\dgFiltDer$ and the $\mathcal{A}$-module we are
  searching for.
\end{proof}

\subsubsection{General Case}
\label{sec:general-case}

Let 
$\mathcal{A}_0 \xla{\phi_0} \mathcal{A}_1 \la \dots
  \la \mathcal{A}_n \xla{\phi_n} \mathcal{A}_{n+1} \la \dots$
be a sequence of dg
algebras and dga-morphisms. Assume that 
\begin{enumerate}[label={(S\arabic*)}]
\item 
  \label{enum:pgdz}
  Each $\mathcal{A}_n=(A_n=\bigoplus_{i \geq 0} A_n^i, d=0)$ is a positively
  graded dg algebra with differential zero. 
\item 
  \label{enum:pdivr}
  $A^0_0=\prod_{x \in W} e_xA^0_0$ is a finite product of division rings.
\item 
  \label{enum:higher-iso}
  There is an increasing sequence $0 \leq r_0 \leq r_1 \leq \dots$ of
  non-negative integers $(r_n)$ with $r_n \ra \infty$ for $n \ra \infty$, such
  that each $\phi_{n}$ is an isomorphism up to degree $r_n$.
  (In particular, all $\phi_{n}^0:A_{n+1}^0 \ra A_n^0$ are isomorphisms.)
\end{enumerate}

The morphisms $\phi_{n}$ induce extension of scalars functors
$\phi_{n}^*:=\pro_{\mathcal{A}_{n+1}}^{\mathcal{A}_n}$ and
we obtain a sequence
$((\dgPerDer(\mathcal{A}_n)), ({\phi_{n}^*}))$
of categories or even of $\mathcal{I}$-filtered categories, where
$\mathcal{I}$ is the poset of segments in $\DZ$. 

\begin{proposition}
  \label{p:limit-dgprae-triang}
  Under the above assumptions, $\invlim \dgPerDer(\mathcal{A}_n)$ has a
  natural structure of triangulated category with the following
  class of distinguished triangles:  
  A triangle 
  $\Sigma$ is distinguished if and only if all 
  $\pr_i(\Sigma)$ are distinguished ($i \in \DN$). 
\end{proposition}
\begin{proof}
  We want to deduce this from Proposition \ref{p:limit-triangulated}.
  It follows from \ref{enum:higher-iso} and
  Proposition \ref{p:equi-on-dgprae-I} that our sequence
  $((\dgPerDer(\mathcal{A}_n)), ({\phi_{n}^*}))$
  satisfies condition \ref{enum:becomes-equiv}.
  Condition \ref{enum:is-union} is fulfilled by Lemma \ref{l:I-iff-I}.
  It is clear that each $\dgPerDer^I$ is closed under extensions in
  $\dgPerDer$.
\end{proof}

Let  
$\mathcal{A}_\infty$
be a dg algebra with dga-morphism $\nu_n:\mathcal{A}_\infty \ra
\mathcal{A}_n$ ($n \in \DN$) such that
$\nu_n = \phi_n \comp \nu_{n+1}$ for all $n \in \DN$.
Assume that 
  there is an increasing sequence $0 \leq s_0 \leq s_1 \leq \dots$ of
  non-negative integers $(s_n)$ with $s_n \ra \infty$ for $n \ra \infty$, such
  that each $\nu_{n}$ is an isomorphism up to degree $s_n$.
Equivalently, we could say that $\mathcal{A}_\infty$ is the inverse
limit of our sequence 
$(\mathcal{A}_n)_{n \in \DN}$
of dg algebras, i.\,e.\
$\mathcal{A}_\infty = \invlim \mathcal{A}_n$.

\begin{proposition}
  \label{p:limit-commutes}
  Under the above assumptions, the obvious functor
  $\dgPerDer(\mathcal{A}_\infty) \ra
    \invlim \dgPerDer(\mathcal{A}_n)$
  is a triangulated equivalence.
\end{proposition}

\begin{proof}
 This is a consequence of 
 Proposition \ref{p:limit-filtfun}
 and Remark
 \ref{rem:generalize} since
 $\dgPerDer(\mathcal{A}_\infty)$ is equivalent to the inverse limit of 
 the constant sequence $((\dgPerDer(\mathcal{A}_\infty)),(\id))$.
\end{proof}

\section{Formality of Equivariant Derived Categories}
\label{cha:form-equiv-deriv}

\subsection{Equivariant Derived Categories of Topological Spaces}
\label{sec:equiv-deriv-categ}

We introduce the equivariant derived category, following \cite{BL}.

If $Y$ is a topological space, we denote by $\Sh(Y)$ the category of
sheaves of real vector spaces on $Y$ and by $\dc^+(Y)$ and $\dcb(Y)$
its bounded below and bounded derived category.

Let $f: X \ra Y$ be a continuous map
and $n \in \DNinfty$. Given a  base change $\tilde Y
\ra Y$ we denote the induced map $X \times_{Y} \tilde{Y} \ra
\tilde{Y}$ by $\tilde f$. We say that $f$ is \define{\acyclic{n}} if
for any base change $\tilde Y \ra Y$ and any sheaf $B \in \Sh(\tilde
Y)$ the truncated adjunction morphism
${B} \ra \tau_{\leq n}\tilde{f}_*\tilde{f}^* {B}$ is an isomorphism.
Here $\tau_{\leq n}$ is the truncation functor for $n \in \DN$ and $\tau_{\leq \infty}= \id$.
The composition of \acyclic{n} maps is \acyclic{n}.
A topological space $X$ is called \define{\acyclic{n}} if the constant map
$X \ra \point$ is \acyclic{n}.

Let $G$ be a topological group. A $G$-\define{space} is a topological
space $X$ with a continuous $G$-action $G\times X \ra X$. A $G$-\define{map}
of $G$-spaces is a continuous $G$-equivariant map.
A $G$-space $X$ is \define{free} if it has a covering
by open $G$-stable subspaces $G$-isomorphic to $G$-spaces
of the form 
$G \times Y$ (for a suitable topological space $Y$) with 
$G$-action $g.(h,y)=(gh,y)$. 

A \define{resolution} of a $G$-space $X$ is a $G$-map from a free
$G$-space to $X$. Morphisms of resolutions are $G$-maps over $X$.
Let $p: P \ra X$ be a resolution of $X$ and
$q: P \ra G\bl P$ the
quotient map. The category $\dbc_G(X, P)$ is defined as follows:
\begin{itemize}
\item Objects $M$ are triples $(M_X, \ol M, \mu)$
  where $M_X \in \dcb(X)$, $\ol M$ in $\dcb(G\bl P)$ and
  $\mu: p^*(M_X) \sira q^*(\ol M)$
  is an isomorphism in $\dcb(P)$.
\item Morphisms $\alpha:M \ra N$ (where $M=(M_X, \ol M, \mu)$ and
  $N=(N_X, \ol N, \nu)$) are pairs $(\alpha_X, \ol \alpha)$ where 
  $\alpha:M_X \ra N_X$ and $\ol \alpha: \ol M \ra \ol N$ are morphisms
  in $\dcb(X)$ and $\dcb(G\bl P)$ respectively such that
  $\nu \comp p^*(\alpha_X) = q^*(\ol \alpha) \comp \mu$.
\end{itemize}
We have two obvious functors: The forgetful functor
$\Forget: \dcb_G(X, P) \ra \dcb(X)$, $M \mapsto M_X$, and the functor
$\gamma: \dcb_G(X, P) \ra \dcb(G \bl P)$, $M \mapsto \ol{M}$.
If $I$ is a segment in $\DZ$, we let $\dc^I_G(X, P)$ be the
full subcategory of $\dcb_G(X,P)$ consisting of objects $M$ 
with $\Forget(M)$ in $\dc^I(X)$.
If $p$ is surjective, this is equivalent to the condition $\gamma(M) \in
\dc^I(G\bl P)$.

A resolution $p:P \ra X$ is \define{\acyclic{n}} if the continuous map
$p$ is \acyclic{n}. Note that any \acyclic{n} map is surjective.

\begin{proposition}
  \label{p:nustarI}
  Let $\nu: P \ra R$ be a morphism of \acyclic{n} resolutions
  $p:P \ra X$ and $r:R \ra X$, where $n \in \DNinfty$. If $I$ is a
  segment with $n > |I|$, the 
  obvious functor $\nu^*: \dcb_G(X, R) \ra \dcb_G(X, P)$
  restricts to an equivalence
  $\nu^*: \dc^I_G(X, R) \ra \dc^I_G(X, P)$.
\end{proposition}

\begin{proof}
  Let $S= P \times_X R$ be the fiber product of $P$ and $R$ over $X$
  with projections $\pi_P:S \ra P$ and $\pi_R:S \ra R$.
  Then $\pi_P^*: \dc^I_G(X,P) \ra \dc^I_G(X,S)$
  and $\pi_R^*: \dc^I_G(X,R) \ra \dc^I_G(X,S)$
  are equivalences of categories by \cite[2.2.2]{BL} (but with $n>|I|$).
  Let $(\id_P, \nu):P \ra S=P\times_X R$ be the unique morphism of
  resolutions with $\pi_P \comp (\id_P, \nu) = \id_P$ and $\pi_R \comp
  (\id_P, \nu) = \nu$. Then $(\id_P, \nu)^*$
  is inverse to $\pi_P^*$ and $\nu^*=(\id_P, \nu)^* \comp
  \pi_R^*$ is an equivalence on $\dc^I_G$.
\end{proof}

If $P\ra X$ and $R\ra X$ are \acyclic{\infty} resolutions 
there is an
equivalence of $\dcb_G(X,P)$ and $\dcb_G(X, R)$ that is defined up to a
canonical natural isomorphism.
The $G$-equivariant derived
category $\dcb_G(X)$ of $X$ is defined to be $\dcb_G(X,P)$, for $p:P\ra X$ an
\acyclic{\infty} resolution (\cite[2.7.2]{BL}). It is a triangulated
category (cf.\ \cite[2.5.2]{BL}).

We give a description of the equivariant derived
category as an inverse limit.
Let 
$  P_0 
\ra
  \dots
  \ra
  P_n \xra{f_n}
  P_{n+1} \ra
  \dots$
be a sequence of morphisms of resolutions $p_n:P_n \ra X$. It gives
rise to a sequence of categories and functors 
$((\dcb_G(X, P_n)), (f_n^*))$. 
We consider the inverse limit $\invlim \dcb_G(X, P_n)$ 
of these categories as defined
in subsection \ref{sec:inverse-limits-categ}. It is an additive
category and has an obvious shift functor.
We denote by $\gamma_i$ the composition
\begin{equation*}
  \invlim \dcb_G(X, P_n) \xra{\pr_i} \dcb_G(X, P_i) \xra{\gamma} \dcb(G\bl P_i).
\end{equation*}

\begin{proposition}
  \label{p:equider-als-limit}
  Keep the above assumptions and assume that $p_n:P_n \ra X$ is
  \acyclic{n}, for each $n \in \DN$. Then $\invlim \dcb_G(X, P_n)$
  carries a natural structure of 
  triangulated category
  with the following
  class of distinguished triangles:  
  A triangle 
  $\Sigma$ is distinguished if and only if all 
  $\gamma_i(\Sigma)$ are distinguished ($i \in \DN$). 
  Moreover, the categories 
  $\dcb_G(X)$ and $\invlim \dcb_G(X, P_n)$ are equivalent as
  triangulated categories.
\end{proposition}

\begin{proof}
  This follows from \cite[2]{BL} (more details in \cite[5.2]{OSdiss}).
\end{proof}

\subsection{Equivariant Derived Categories of Varieties}
\label{sec:equiv-deriv-categ-1}

Let $G$ be an affine algebraic group (defined over $\DC$, as all
varieties in the following). The definitions of a $G$-variety and of a
$G$-morphism between $G$-varieties are the obvious generalizations
from the topological category. 

A \define{Zariski principal fiber bundle (Zpfb) with structure group $G$}
(or \define{$G$-Zpfb}) is a surjective $G$-morphism $q:E \ra B$ between
$G$-varieties with the following property:
For every point in $B$, 
there is a Zariski open neighborhood $U$ in $B$ and a $G$-isomorphism
$\tau: G \times U \sira q\inv(U)$ 
(here 
the $G$-action on $G \times U$ is given by $g.(h,u)= (gh, u)$)
such that 
$q \comp \tau$ is equal to the 
projection $\pr_U: G \times U \ra U$.
A map $\tau$ as above 
is called a 
\define{local trivialization} over $U$.
By abuse of notation we often say that $E$ is a Zpfb.
A morphism $f:E\ra E'$ of $G$-Zpfbs is a $G$-morphism $f:E\ra E'$.
It induces a morphism $f: B \ra B'$ on quotient spaces.

Let $X$ be a $G$-variety. A 
\define{Zariski resolution} 
of $X$ is a datum $(B \xla{q} E \xra{p} X)$ consisting of a Zpfb $q:E
\ra B$ together with a $G$-morphism $p:E \ra X$. 
We often omit $q:E \ra B$ from the notation and say that $p:E \ra X$
is a Zariski resolution or even that $E$ is a Zariski resolution of
$X$. 
Morphisms $E \ra E'$ of Zariski resolutions of $X$ are $G$-morphisms over $X$.

Let $n \in \DNinfty$. A variety (morphism of varieties) is called
\define{\acyclic{n}},  
if it is \acyclic{n} with respect to the classical topology.
A Zariski resolution $(B \xla{q} E \xra{p} X)$ is \define{\acyclic{n}} if $p$
is \acyclic{n}.

Let $E_0 \xra{f_0} E_1 \ra \dots \ra E_n \xra{f_n} E_{n+1} \ra
\dots$ be a sequence of Zariski resolutions of a $G$-variety
$X$. If $p_n:E_n \ra X$ is \acyclic{n},
Proposition
\ref{p:equider-als-limit}
provides a
triangulated equivalence
\begin{equation}
  \label{eq:edc-resoinvlim}
  \dcb_G(X) \cong \invlim \dcb_G(X, E_n).
\end{equation}

Let $(X, \mathcal{S})$ be a
stratified $G$-variety (we do not assume that the strata are
$G$-stable).
We denote the full subcategory of $\dcb_G(X)$ consisting of objects
$M$ with $\Forget(M) \in \dcb(X, \mathcal{S})$ by $\dcb_G(X,
\mathcal{S})$.
Similarly, we define $\dcb_G(X,E, \mathcal{S}) \subset \dcb_G(X, E)$,
if $E$ is a Zariski resolution of $X$. 
If $I$ is a segment, we use the obvious definitions for
$\dc^I_G(X,\mathcal{S})$ etc..
Equivalence \eqref{eq:edc-resoinvlim} restricts to a triangulated
equivalence 
\begin{equation}
  \label{eq:edc-resoinvlim-s}
  \dcb_G(X, \mathcal{S}) \cong \invlim \dcb_G(X, E_n, \mathcal{S}).
\end{equation}

If $X$ is a $G$-variety, we denote by $\dcb_{G,c}(X)$ the full
subcategory of $\dcb_G(X)$ consisting of objects $F$ such that
$\Forget(F)$ is constructible for some stratification of $X$.

\begin{proposition}
  \label{p:orbit-constructible}
  Let $(X, \mathcal{S})$ be a stratified $G$-variety. If each stratum
  is a $G$-orbit, then $\dcb_G(X, \mathcal{S})= \dcb_{G,c}(X)$.
\end{proposition}

\begin{proof}
  Let $M$ be in $\dcb_{G,c}(X)$ and $l:S \ra X$ the inclusion of an
  orbit. We have to prove that
  $\Ho^i(l^*(\Forget(M)))=\Forget(\Ho^i(l^*(M)))$ 
  is a local system on $S$.   
  But $\Ho^i(l^*(M))$ is in the heart of the standard t-structure on $\dcb_{G,c}(S)$
  and hence a $G$-equivariant constructible sheaf on the orbit $S$ (\cite[2.5.3]{BL}).
\end{proof}

A $G$-\define{stratification} of a $G$-variety is a stratification by $G$-stable
strata. A $G$-\define{stratified} variety is a $G$-variety $X$ with a
$G$-stratification.

Let $(X, \mathcal{S})$ be a $G$-stratified variety.
An \define{\asystem{}} $(E, f)$ of $(X, \mathcal{S})$ is a sequence
$E_0 
\ra
\dots \ra E_n \xra{f_n} E_{n+1} \ra \dots$
of Zariski resolutions $(B_n \xla{q_n} E_n \xra{p_n} X)$ of $X$. 
An \asystem{} $(E, f)$ is called an \define{A-\asystem} if the following
conditions hold.
\begin{enumerate}[label={(A\arabic*)}, ref={A\arabic*}]
\item 
  \label{enum:acyclic}
  Each $(B_n \xla{q_n} E_n \xra{p_n} X)$ is \acyclic{n}.
\item 
  \label{enum:Sn-strati}
  $\mathcal{S}_n := \{q_n(p_n\inv(S)) \mid S \in
  \mathcal{S}\}$
  is a stratification of $B_n$,
  for all $n \in \DN$. (In particular, each stratum of $\mathcal{S}_n$
  is irreducible.)
\item 
  \label{enum:simply-conn}
  Every stratum in $\mathcal{S}_n$ is simply connected.
\end{enumerate}

Any A-\asystem{} $E_0 \xra{f_0} E_1 \xra{f_1} \dots$ yields
a sequence $B_0 \xra{f_0} B_1 \xra{f_1} \dots$ of varieties, and
a sequence $((\dcb(B_n, \mathcal{S}_n)), (f_n^*))$ of triangulated
categories. 

\begin{proposition}
  \label{p:go-to-limit}
  Let $(E,f)=(\dots \ra E_n \xra{f_n} E_{n+1} \ra \dots)$ be an A-\asystem{} of a $G$-stratified 
  variety $(X, \mathcal{S})$.
  Then there is a canonical triangulated equivalence
  \begin{equation}
    \label{eq:limes}
    \invlim \dcb_G(X, E_n, \mathcal{S}) 
    \sira \invlim \dcb(B_n, \mathcal{S}_n) 
  \end{equation}
\end{proposition}

\begin{proof}
  The functor $\dcb_G(X, E_n) \ra
  \dcb(B_n)$, $(M_X, \ol M, \mu) \mapsto \ol M$, restricts to a functor
  $\nu_n:\dcb_G(X, E_n, \mathcal{S}) \ra \dcb(B_n, \mathcal{S}_n)$,
  since $q_n$ is locally trivial. The inverse limit $\invlim
  \nu_n$ is the functor in \eqref{eq:limes}.

  All categories $\dcb_G(X, E_n, \mathcal{S})$ are $\mathcal{I}$-filtered by the
  $\dc^I_G(X, E_n, \mathcal{S})$ 
  (where $\mathcal{I}$ is the poset of segments in $\DZ$), 
  and are the union of these
  subcategories. Similarly for $\dcb(B_n, \mathcal{S}_n)$.
  For $n > |I|$, restriction 
  $f_n^{*}: \dc^I_G(X, E_{n+1}, \mathcal{S})
    \ra \dc^I_G(X, E_n, \mathcal{S})$
  is an equivalence (Proposition
  \ref{p:nustarI}).
  
  We claim that 
  $\nu_n^I:\dc^I_G(X, E_n, \mathcal{S}) \ra \dc^I(B_n, \mathcal{S}_n)$
  is an equivalence for $n > |I|$.
  By \cite[Lemma 2.3.2]{BL} (but with $n>|I|$),
  $\nu_n^I$ is fully faithful and its essential 
  image is closed under extensions in $\dcb(B_n, \mathcal{S}_n)$. 
  Let $S \in \mathcal{S}$ and $S_n := q_n(p_n\inv(S))$. 
  The inclusions $l_S$ and $l_{S_n}$ and proper base change give rise
  to an object ``extension by zero of the constant sheaf on $S$'' in 
  $\dc^{[0,0]}_G(X, E_n)$ that is mapped to $l_{S_n!}\ul{S_n}$ under
  $\nu_n: \dcb_G(X, E_n, \mathcal{S}) \ra \dcb(B_n, \mathcal{S}_n)$.
  It follows from Lemma \ref{p:erzeuger} below that $\nu_n^I$ is
  dense.

  Proposition \ref{p:limit-filtfun} shows that \eqref{eq:limes} is
  an equivalence and that the $\mathcal{I}$-filtered category $\dcb(B_n,
  \mathcal{S}_n)$ satisfies condition \ref{enum:becomes-equiv}. 
  Proposition \ref{p:limit-triangulated} equips 
  $\invlim \dcb(B_n, \mathcal{S}_n)$ with the structure of a
  triangulated category, and it is obvious that \eqref{eq:limes} is triangulated.
\end{proof}

\begin{corollary}
  \label{c:go-to-limit}
  Let $(E, f)$ be an A-\asystem{}, $I$ a segment and $N \in \DN$.
  If $N > |I|$, the obvious functor 
  $\invlim \dc^I(B_n, \mathcal{S}_n) \ra \dc^I(B_N)$
  is fully faithful.
\end{corollary}

\begin{proof}
  The proof of Proposition \ref{p:go-to-limit}
  shows that
  $f_n^{*}: \dc^I(B_{n+1}, \mathcal{S}_{n+1}) \ra \dc^I(B_n, \mathcal{S}_n)$
  is an equivalence for $n>|I|$. Hence Lemma \ref{l:limit-equiv} shows that
  $\pr_N:\invlim \dc^I(B_n, \mathcal{S}_n) \ra \dc^I(B_N, \mathcal{S}_N)$
  is an equivalence for $N > |I|$.  
\end{proof}

\begin{lemma}
  \label{p:erzeuger}
  Let $(X, \mathcal{S})$ be a stratified variety with simply connected
  strata and $I$ a segment in $\DZ$. Then every $A \in \dc^I(X,
  \mathcal{S})$ is an iterated extension of objects $l_{S!} \ul{S}[-i]$, for $S
  \in \mathcal{S}$ and $i \in I$.
\end{lemma}

\begin{proof}
  The shift $[1]$ and the truncation functors for the standard t-structure
  allow us to assume that $A \in \Sh(X, \mathcal{S})$.
  If $j$ is the inclusion of an open stratum $U$ and $i$ the inclusion of
  its closed complement, we get a distinguished triangle $(j_!j^*(A),
  A, i_*i^*(A))$. Since $j^*(A)$ is a finite direct sum of constant
  sheaves $\ul U$, an induction on the number of strata finishes the
  proof.
\end{proof}

\subsection{Equivariant Intersection Cohomology Complexes}
\label{sec:equiv-inters-cohom}

Let $G$ be an affine algebraic group of complex dimension $d_G$ and
$(X, \mathcal{S})$ a $G$-stratified variety.
On $\dcb_G(X,
\mathcal{S})$, there is the perverse t-structure whose heart is the category
$\Perv_G(X, \mathcal{S})$ of equivariant perverse sheaves (smooth along
$\mathcal{S}$) (see \cite[5]{BL}).
If $\mathcal{L}$ is a $G$-equivariant local system on $S$, we have the
equivariant intersection cohomology complex $\IC_G(\ol S,
\mathcal{L})=l_{S!*}([d_S]\mathcal{L})$ in $\Perv_G(X, \mathcal{S})$. 
We are mainly interested in the case of the constant $G$-equivariant
local system $\ul S_G$ on $S$ and define 
$\IC_G(S):= \IC_G(\ol S, \ul S_G)$.
We will describe this object precisely using the following type of \asystem{}.

An \define{AB-\asystem} is an A-\asystem{} $(E, f)$ of $(X,
\mathcal{S})$ such that the following conditions are satisfied. 
\begin{enumerate}[label={(B\arabic*)}, ref={B\arabic*}]
\item 
  \label{enum:p-smooth}
  Each morphism $p_n:E_n \ra X$ is smooth of relative complex dimension $d_{p_n}$.
\item 
  \label{enum:closed-embedding}
  Each $f_n:B_n \ra B_{n+1}$ is a
  closed embedding of varieties.
\item
  \label{enum:stratified-nn} 
  For all $n \in \DN$ and $S \in \mathcal{S}$, 
  $f_n: \ol S_n \ra \ol S_{n+1}$ is a
  normally nonsingular inclusion of codimension $c_n$, where
  $S_n:=q_n(p_n\inv(S))$; here $c_n$ only depends on $n$. 
\end{enumerate}

Let $(E, f)$ be an AB-\asystem{} of 
$(X, \mathcal{S})$. 
If we consider $\IC_G(S)$ as an object of $\invlim \dcb(B_n, \mathcal{S}_n)$, 
using equivalences \eqref{eq:edc-resoinvlim-s} and
\eqref{eq:limes}, we have
(cf.\ \cite[5.1]{BL}, using \ref{enum:p-smooth})
\begin{equation*}
  \pr_n(\IC_G(S)) = [d_G-d_{p_n}]\IC(S_n) \in \dcb(B_n, \mathcal{S}_n).
\end{equation*}
In order to avoid to many shifts, we replace $\invlim \dcb(B_n,
\mathcal{S}_n)$ by an equivalent category as follows.
Consider the following morphism of sequences of triangulated categories
\begin{equation}
  \label{eq:shift-sequence}
  \xymatrix{
    {\dcb(B_0, \mathcal{S}_0)}                 \ar[d]^{[d_{p_0}-d_G]} &
    {\dcb(B_1, \mathcal{S}_1)} \ar[l]_-{f_0^*} \ar[d]^{[d_{p_1}-d_G]} &
    {\dcb(B_2, \mathcal{S}_2)} \ar[l]_-{f_1^*} \ar[d]^{[d_{p_2}-d_G]} &
    {\dots}  \ar[l]_-{f_2^*}\\
    {\dcb(B_0, \mathcal{S}_0)} &
    {\dcb(B_1, \mathcal{S}_1)} \ar[l]_-{[-c_0]f_0^*} &
    {\dcb(B_2, \mathcal{S}_2)} \ar[l]_-{[-c_1]f_1^*} &
    {\dots.} \ar[l]_-{[-c_2]f_2^*}
  }
\end{equation}
If we denote the inverse limit of the second row by $\invlim \dcb[B_n,
\mathcal{S}_n]$, this morphism induces a triangulated equivalence
\begin{equation}
  \label{eq:rund-eckig}
  \invlim \dcb(B_n, \mathcal{S}_n) \sira
  \invlim \dcb[B_n, \mathcal{S}_n].
\end{equation}

Conditions 
\ref{enum:Sn-strati},
\ref{enum:closed-embedding}
and
\ref{enum:stratified-nn} show 
that each $f_n:B_n \ra B_{n+1}$ meets the assumptions 
made before 
\eqref{eq:nnpervstrat} and \eqref{eq:nnmhmstrat} in subsection
\ref{sec:form-norm-nons}.
So we obtain
isomorphisms
\begin{align}
  \label{eq:nnICperv}
  \iota_{S,n}:[-c_n]f_n^*(\IC(S_{n+1})) & \sira \IC(S_n) && \text{in
    $\Perv(B_n, \mathcal{S}_n)$ and}\\
  \notag
  \tilde{\iota}_{S,n}:[-c_n]f_n^*(\tilde{\IC}(S_{n+1})) & \sira \tilde{\IC}(S_n)
  && \text{in $\MHM(B_n, \mathcal{S}_n)$.}
\end{align}
As an object of $\invlim \dcb[B_n, \mathcal{S}_n]$,
the equivariant intersection cohomology complex $\IC_G(S)$ is 
$((\IC(S_n)), (\iota_{S,n}))$.
Note that 
$\tilde{\IC}_G(S) = ((\tilde{\IC}(S_n)), (\tilde{\iota}_{S,n}))$
is a natural ``Hodge lift'' of $\IC_G(S)$.
The same argument as in the proof of Theorem 
\ref{t:formality-ic-closed-embedding} 
shows that all functors $[-c_n]f_n^*$ in
\eqref{eq:shift-sequence}
are t-exact with respect to the perverse t-structures.

\subsection{Better Approximations and Formality}
\label{sec:better-approximations-formality}

Let $(E, f)$ be an \asystem{} of a $G$-stratified variety $(X,
\mathcal{S})$ and assume that we are given
stratifications $\mathcal{T}_n$ of $B_n$, for each $n \in \DN$. 
The triple $(E, f, \mathcal{T}=(\mathcal{T}_n))$ is called an
\define{ABC-\asystem} if $(E,f)$ is an AB-\asystem{} and the following
conditions hold.
\begin{enumerate}[label={(C\arabic*)}, ref={C\arabic*}]
\item
  \label{enum:finer-strat}
  Each $\mathcal{T}_n$ is a cell-stratification of $B_n$ that is finer
  than the stratification $\mathcal{S}_n$.
\item 
  \label{enum:closed-strat-embedding}
  Each $f_n:(B_n, \mathcal{T}_n) \ra (B_{n+1}, \mathcal{T}_{n+1})$ is a
  closed embedding of (cell-)stratified varieties.
\item 
  \label{enum:T-pure}
  The Hodge sheaf $\tilde{\IC}(S_n)$ is
  $\mathcal{T}_n$-pure of weight $d_{S_n}$,
  for all $n \in \DN$ and $S \in \mathcal{S}$.
  (By Remark \ref{rem:T-pure} we can equivalently require 
  $\mathcal{T}_n$-$*$-purity of weight $d_{S_n}$.)
\end{enumerate}
In subsection 
\ref{sec:exist-exampl-abc}
we show how to construct ABC-\asystem{}s.

For $M$, $N$ in $\dcb_G(X)$, define
$\Ext^n(M,N) :=\Hom_{\dcb_G(X)}(M,[n]N)$ and 
\begin{equation*}
  \Ext(M, N) := \bigoplus_{n \in \DZ}\Ext^n(M, N).  
\end{equation*}
The (equivariant) extension algebra of $M$ is 
$\Ext(M) := \Ext(M,M)$.

\begin{theorem}
  \label{t:das-ziel}
 Let $G$ be an affine algebraic group and $(X, \mathcal{S})$ a
  $G$-stratified variety. If $(X, \mathcal{S})$ has an ABC-\asystem,
  there is a triangulated equivalence
  \begin{equation}
    \label{eq:das-ziel}
    \dcb_G(X, \mathcal{S}) \cong 
    \dgPerDer(\Ext(\IC_G(\mathcal{S}))),
  \end{equation}
  where $\IC_G(\mathcal{S})$ is the direct sum of the 
  $(\IC_G(S))_{S \in \mathcal{S}}$.
  This equivalence is t-exact with respect to the perverse t-structure
  on $\dcb_G(X, \mathcal{S})$ and the t-structure from 
  Theorem~\ref{t:t-structure-auf-perf} on
  $\dgPerDer(\Ext(\IC_G(\mathcal{S})))$. 
  By restriction to the heart, it induces an equivalence of abelian categories
  \begin{equation}
    \label{eq:das-ziel-herz}
    \Perv_G(X, \mathcal{S}) \cong \dgFilMod(\Ext(\IC_G(\mathcal{S}))).
  \end{equation}
\end{theorem}

\begin{proof}
  Let $(E, f, \mathcal{T})$ be an ABC-\asystem{} of $(X, \mathcal{S})$. 
  By \eqref{eq:edc-resoinvlim-s}, 
  \eqref{eq:limes}
  and
  \eqref{eq:rund-eckig}
  we have equivalences
  \begin{equation}
    \label{eq:limes-combined}
    \dcb_G(X, \mathcal{S}) \cong 
    \invlim \dcb_G(X, E_n, \mathcal{S}) 
    \sira \invlim \dcb(B_n, \mathcal{S}_n) 
    \sira \invlim \dcb[B_n, \mathcal{S}_n].
  \end{equation}
  of triangulated categories.

  Properties \ref{enum:simply-conn}, \ref{enum:finer-strat} and
  \ref{enum:T-pure} allow to apply 
  Theorem \ref{t:formality-ic} (cf.\ Remark
  \ref{rem:formy-simplyconn}), and we obtain equivalences 
  \begin{equation*}
    \Formy_n:= \Formy_{\tilde{P}_n \ra
      \tilde{\IC}(\mathcal{S}_n)}^{\mathcal{T}_n}: \dcb(B_n,
    \mathcal{S}_n) \sira \dgPerDer(\mathcal{E}_n) 
  \end{equation*}
  of triangulated categories, where we fixed perverse-projective
  resolutions $\tilde{P}_{n,S_n}\ra \tilde{\IC}(S_n)$ and where
  $\mathcal{E}_n:=\Ext(\IC(\mathcal{S}_n))$.
  The isomorphisms \eqref{eq:nnICperv} induce dga-mor\-phisms
  $\phi_n: \mathcal{E}_{n+1} \ra \mathcal{E}_n$.
  Properties 
  \ref{enum:Sn-strati},
  \ref{enum:simply-conn},
  \ref{enum:closed-embedding}, 
  \ref{enum:stratified-nn}, 
  \ref{enum:finer-strat},
  \ref{enum:closed-strat-embedding},
  \ref{enum:T-pure} and
  Theorem \ref{t:formality-ic-closed-embedding} 
  yield the following commutative (up to natural isomorphism)
  diagram with triangulated and t-exact functors:
  \begin{equation*}
    \xymatrix{
      {\dbc(B_{n+1}, \mathcal{S}_{n+1})}
      \ar[rr]^{[-c_n]f_n^*}
      \ar[d]_{\Formy_{n+1}}^\sim
      &&
      {\dbc(B_n, \mathcal{S}_n)}
      \ar[d]_{\Formy_n}^\sim
      \\
      {\dgPerDer(\mathcal{E}_{n+1})}
      \ar[rr]^-{? \Lotimesover{\mathcal{E}_{n+1}} \mathcal{E}_n}
      &&
      {\dgPerDer(\mathcal{E}_n).}
    }
  \end{equation*}
  Hence the sequence $(\Formy_n)_{n \in \DN}$ defines a morphism
  between sequences of triangulated categories. Its inverse limit
  establishes an equivalence
  \begin{equation}
    \label{eq:form-limit}
    \invlim \dcb[B_n, \mathcal{S}_n]
    \sira 
    \invlim \dgPerDer(\mathcal{E}_n).
  \end{equation}
 
  Define 
  $\mathcal{E}_\infty:=\Ext(\IC_G(\mathcal{S}))$.
  This
  is a positively graded
  dg algebra with differential zero and has as degree zero part the
  product of $|\mathcal{S}|$ copies of $\DR$.
  From 
  \eqref{eq:limes-combined} we obtain dga-morphisms
  $\nu_n:\mathcal{E}_\infty \ra \mathcal{E}_n$ such that $\nu_n =
  \phi_n \comp \nu_{n+1}$.
  Let $J$ be a segment such that $\IC_G(\mathcal{S}) \in \dc^J_G(X)$.
  We deduce from Corollary \ref{c:go-to-limit} that $\nu_n$ is an
  isomorphism up to degree $n-|J|-1$. 
  So our sequence of dg algebras $((\mathcal{E}_n), (\phi_n))$
  satisfies (if we forget the first $|J|+1$ members and renumerate,
  which is harmless for the following)
  the conditions \ref{enum:pgdz}-\ref{enum:higher-iso} considered
  in subsection \ref{sec:general-case}, and $\mathcal{E}_\infty$ is
  the inverse limit of this sequence.
  Proposition \ref{p:limit-dgprae-triang} 
  shows 
  that $\invlim \dgPerDer(\mathcal{E}_n)$
  carries a natural structure of 
  triangulated category.
  Since all functors $\Formy_n$ are triangulated, it follows from
  Proposition \ref{p:equider-als-limit} that equivalence
  \eqref{eq:form-limit} is triangulated.
  Finally, Proposition \ref{p:limit-commutes} provides a triangulated
  equivalence
  $\dgPerDer(\mathcal{E}_\infty) \sira \invlim \dgPerDer(\mathcal{E}_n)$.
  This establishes \eqref{eq:das-ziel}.
  Since $\IC_G(S)$ is mapped to
  $e_S \mathcal{E}_{\infty}$, equivalence \eqref{eq:das-ziel} is
  t-exact, and we obtain \eqref{eq:das-ziel-herz}.
\end{proof}

\subsection{Existence of ABC-Approximations}
\label{sec:exist-exampl-abc}

Let $G$ be an affine algebraic group. An \define{ABCD-\asystem} of a
$G$-stratified variety $(X, \mathcal{S})$ is an ABC-\asystem{} 
$(E, f, \mathcal{T})$ of $(X, \mathcal{S})$ satisfying the following
condition. 
\begin{enumerate}[label={(D)}, ref={D}]
\item 
  \label{enum:local-around}
  For each $n \in \DN$, the Zpfb $q_n:E_n \ra B_n$ can be
  trivialized around each stratum $T \in \mathcal{T}_n$ (this means
  that there is an open subvariety $U$ of $B_n$ containing $T$ such
  that $q_n$ has a local trivialization over $U$). 
\end{enumerate}
An \define{ABCD-\asystem}
for $G$ is an ABCD-\asystem{} of the $G$-stratified variety 
$(\point, \{\point\})$. 

\begin{proposition}\label{p:T-B-approxi}
  Every torus and any connected solvable affine algebraic group has an ABCD-\asystem{}. 
\end{proposition}

Let us remark that ABCD-\asystem{}s also exist for $\GL_n(\DC)$ and
parabolic subgroups of $\GL_n(\DC)$ (see \cite[5.6]{OSdiss}).

\begin{proof}
  Let $q_i:E_i:=\DC^{i+1}\setminus \{0\} \ra B_i:=\DP^i(\DC)$ be the
  obvious \mbox{$\DC^*$-Zpfb}. 
  The standard closed embeddings
  $\DC^{i+1}\hra \DC^{i+2}$, $x \mapsto (x,0)$ induce morphisms of
  Zpfbs $f_i:E_i \ra E_{i+1}$.
  Let $\mathcal{T}_i$ be the standard cell-stratification of
  $B_i=\DP^i(\DC)$, the strata being the orbits of the standard Borel subgroup of $\GL_{i+1}(\DC)$
  under the natural action.
  Note that $E_i$ is \acyclic{2i} (cf.\ \cite[3.1]{BL}).
  Thus $(B \xla{q} E, f)$ is an ABCD-\asystem{} for $\DC^*$
  (for details see \cite[5.6]{OSdiss}). 
  Taking the obvious product of this construction shows that
  any torus has an ABCD-\asystem.

  Now let $G$ be a connected solvable group. Choose a maximal torus 
  $T \subset G$ and let $U \subset G$ be the unipotent radical.
  Let $(B \xla{q^T} E^T,\mathcal{T}, f^T)$ be an ABCD-\asystem{} for
  $T$.
  Define $E_i^G := \ind_T^G E_i^T = G \times_T E_i^T$, and let $f_i^G
  := \ind_T^G f_i^T$. The morphisms
  $G \times E_i^T \ra B_i$, $(p,e) \mapsto q_i^T(e)$ induce
  $G$-Zpfbs $q_i^G:E_i^G \ra B_i$.
  Since multiplication $U\times T \sira G$ is an isomorphism we get an isomorphism of
  varieties 
  $E_i^G = G \times_T E_i^T \sila U \times
  E_i^T$. Since $U$ is \acyclic{\infty}
  and $E_i^T$ is \acyclic{i},
  $E_i^G$ is \acyclic{i}.
  So $(B \xla{q^G} E^G,\mathcal{T}, f^G)$ is an ABCD-\asystem{} for $G$.
\end{proof}

\begin{proposition}\label{p:approx-for-variety}
  Let $G$ be an affine algebraic group and $(X, \mathcal{S})$ a
  $G$-stratified variety.
  Assume that 
  \begin{enumerate}[label={(R\arabic*)}]
  \item $\mathcal{S}$ is a $G$-stratification into cells,
  \item 
    \label{enum:XSrein}
    $\tilde{\IC}(S)$ is $\mathcal{S}$-pure, for every $S \in \mathcal{S}$,
  \item $G$ has an ABCD-\asystem{},
  \item 
    \label{enum:G-closed-embedding}
    there is a $G$-stratified variety $(Y, \hat{\mathcal{S}})$ together with
    \begin{enumerate}[label={(\alph*)}]
    \item
      \label{enum:lokalabgeschlosssen}
      a $G$-equivariant locally closed embedding $v:X \ra Y$ satisfying
      $v(\mathcal{S}):= \{v(S) \mid S \in \mathcal{S}\} \subset
      \hat{\mathcal{S}}$, and
    \item 
      \label{enum:abgeschlossenemanifold}
      a $G$-equivariant closed embedding of $Y$ in a smooth
      $G$-manifold $M$.
    \end{enumerate}
  \end{enumerate}
  Then $(X, \mathcal{S})$ has an ABCD-\asystem{}.
\end{proposition}

\begin{remark}
  Condition \ref{enum:G-closed-embedding} will be used for the proof
  of \ref{enum:stratified-nn}. Possibly it is redundant.
  It is satisfied for Schubert varieties and more generally for
  unions of Borel-orbits that are locally closed in the flag variety
  (take as $Y$ and $M$ the flag variety).
  For a normal projective $G$-variety $Y$ and connected $G$, 
  \ref{enum:G-closed-embedding} 
  \ref{enum:abgeschlossenemanifold}
  is satisfied by \cite{sumihiro} or \cite{mumford-git}.
\end{remark}

\begin{proof}
  If $(B \xla{r} E, f, \mathcal{T})$ is an ABCD-\asystem{} for $G$,
  its $i$-th subdatum $(B_i \xla{r_i} E_i, \mathcal{T}_i)$ satisfies an
  obvious subset of the conditions
  \ref{enum:acyclic}-\ref{enum:local-around}.
  
  So let $(B \xla{r} E, \mathcal{T})$ be the $i$-th subdatum of an
  ABCD-\asystem{} for $G$. 
  Consider
  the commutative diagram
  \begin{equation}
    \label{eq:ithpart}
    \xymatrix{
      {E \times_G X} \ar[d]^{\ol{\pi}} 
      & {E \times X}
      \ar[l]_-{q} \ar[r]^-{p} \ar[d]^{\pi} 
      & X \ar[d] \\ 
      B  & E \ar[l]_-{r} \ar[r]^c & {\point.}
    }
  \end{equation}
  Here $q$ is the quotient map for the
  diagonal action of $G$ on $E \times X$, $p$ is the second
  projection, $\pi$ the first projection and $\ol \pi$ the induced
  map on quotient spaces. 
  We claim that the upper row of diagram \eqref{eq:ithpart} together
  with
  \begin{equation*}
    r\inv(\mathcal{T})\times_G \mathcal{S} 
    := \{r\inv(T)\times_G S \mid T \in \mathcal{T}, S \in \mathcal{S} \}
  \end{equation*}
  satisfies the conditions imposed on the $i$-th subdatum of an
  ABCD-\asystem{} of $(X, \mathcal{S})$.

  The square on the left in \eqref{eq:ithpart} is cartesian, $q$ is a Zpfb and $\ol{\pi}$ is
  a (Zariski locally trivial) fiber bundle with fiber $X$. 
  These statements are also true in the classical topology.
  They can be deduced from local trivializations of $r$.
  If $\tau: G \times U \sira r\inv(U)$ is such a trivialization 
  (cf.\ subsection \ref{sec:equiv-deriv-categ-1}) 
  over an open subvariety $U$ of $B$, diagram
  \eqref{eq:ithpart} restricts to
  \begin{equation*}
    \xymatrix{
      {U \times X} \ar[d]^{\pr_U} 
      &&& {G \times U \times X}
      \ar[lll]_-{(u,g\inv x) \mapsfrom(g,u,x)}
      \ar[r]^-{p} \ar[d]^{\pr_{G\times U}} 
      & X \ar[d] \\ 
      U  &&& G \times U \ar[lll]_-{\pr_U} \ar[r]^c & {\point.}
    }
  \end{equation*}
  Here we consider $U\times X$ as an open subvariety of $E \times_G X$, the
  inclusion given by $(u,x) \mapsto [\tau(1,u), x]$.

  Since $c$ is \acyclic{i} and smooth, the same holds
  for $p$ (\ref{enum:acyclic}, \ref{enum:p-smooth}).
  If $S \in \mathcal{S}$ is a stratum, the intersection of 
  $E \times_G S = q(p\inv(S))$ with $U \times X \subset E \times_G X$
  is $U \times S$. Hence 
  $E \times_G \mathcal{S} := \{E \times_G S \mid S \in \mathcal{S}\}$
  is a stratification of $E \times_G X$ (\ref{enum:Sn-strati}) 
  (the irreducibility of the strata will be established below). 
  The long exact sequence of homotopy groups for the
  fiber bundle $\ol{\pi}:E\times_G S \ra B$ with fiber $S$ shows that $E
  \times_G S$ is simply connected (\ref{enum:simply-conn}), since $S$
  is a cell and $B$ is simply connected by assumption.
  
  Let $S \in \mathcal{S}$ and $T \in \mathcal{T}$ be strata. 
  The intersection of $r\inv(T) \times_G S$ with $U \times X$ is
  $(T\cap U) \times S$, so
  $r\inv(\mathcal{T}) \times_G \mathcal{S}$ is a stratification of
  $E\times_G X$. 
  By \ref{enum:local-around}, we find a local trivialization $\tau$ of $r$ as
  above with $T \subset U$. Then $\tau \times \id_X$ is a local
  trivialization of $q$ over $U \times X$, and $r\inv(T) \times_G S =
  T \times S \subset U \times X$ is a cell (\ref{enum:local-around}, \ref{enum:finer-strat}).
  
  By \ref{enum:Sn-strati}, $B$ is irreducible. Let $T \in \mathcal{T}$
  be dense in $B$. Then $r\inv(T)$ is dense in $E$ and the cell $r\inv(T) \times_G S$ is dense in $E \times_G S$, for all 
  $S \in \mathcal{S}$, showing 
  the irreducibility of each stratum $E \times_G S$ (\ref{enum:Sn-strati}).
 
  Let $S \in \mathcal{S}$. We prove \ref{enum:T-pure}. By Remark
  \ref{rem:T-pure} is is sufficient to show that $\tilde{\IC}(E\times_G S)$ is
  $(r\inv(\mathcal{T}) \times_G \mathcal{S})$-$*$-pure of weight
  $d_B+d_S$. Let $R \in \mathcal{S}$, $T \in \mathcal{T}$. 
  We choose $U \subset B$ open as above containing $T$. Then the inclusion
  $r\inv(T)\times_G R \os{l}{\hra} E \times_G X$ looks like
  $T \times R \os{t\times l_R}{\hra} U \times X \os{j}{\hra} E \times_G
  X$. Since $j$ is an open embedding, $j^*$ preserves weights and we obtain
  \begin{equation}
    \label{eq:IC-pure-ST}
    j^*(\tilde{\IC}(E\times_G S)) \cong \tilde{\IC}(U \times S) \cong \tilde{\IC}(U)\boxtimes
    \tilde{\IC}(S)
  \end{equation}
  Let $l_T:T\hra B$ be the inclusion and $\tilde{\IC}(B)$ the Hodge
  intersection cohomology sheaf on $B$. Then 
  $ l_T^*(\tilde{\IC}(B)) \cong t^*(\tilde{\IC}(U))$, and restriction of \eqref{eq:IC-pure-ST}
  yields
  \begin{equation*}
    l^*(\tilde{\IC}(E\times_G S)) 
    \cong t^*(\tilde{\IC}(U))\boxtimes l_R^*(\tilde{\IC}(S))
    \cong l_T^*(\tilde{\IC}(B)) \boxtimes l_R^*(\tilde{\IC}(S))
  \end{equation*}
  which is pure of weight $d_B + d_S$ by assumptions \ref{enum:T-pure} and \ref{enum:XSrein}.

  Let $(B \xla{r} E, f, \mathcal{T})$ be an ABCD-\asystem{} for $G$.
  Let 
  \begin{equation*}
    (E_i \times_G  X \xla{q_i} E_i \times X \xra{p_i} X,
    r_i\inv(\mathcal{T}_i) \times_G \mathcal{S})
  \end{equation*}
  be the datum constructed from its $i$-th subdatum by the above method. We claim
  that the sequence of these data together with the sequence of
  morphisms 
  $f_i \times \id_X: E_i \times X \ra E_{i+1} \times X$
  defines an
  ABCD-\asystem{} of $(X, \mathcal{S})$. Conditions
  \ref{enum:closed-embedding} and \ref{enum:closed-strat-embedding}
  are obviously satisfied. 
  A slight modification of the above arguments 
  shows that $(E_{i}\times_G Y, E_i \times_G
  \hat{\mathcal{S}})$ is a stratified variety. Consider the diagram
  \begin{equation*}
    \xymatrix{
      {E_{i+1} \times_G X}
      \ar[rr]^{\id \times_G v} 
      &&
      {E_{i+1} \times_G Y}
      \ar[r] &
      {E_{i+1} \times_G M}
      \\
      {E_{i} \times_G X}
      \ar[rr]^{\id \times_G v} 
      \ar[u]_{f_i \times_G \id_X} 
      &&
      {E_{i} \times_G Y}
      \ar[r] 
      \ar[u]_{f_i \times_G \id_Y} 
      &
      {E_{i} \times_G M,}
      \ar[u]_{f_i \times_G \id_M} 
    }
  \end{equation*}
  where both squares are cartesian.
  In the smooth manifold $E_{i+1} \times_G M$, the smooth submanifold
  $E_i \times_G M$ is transverse to each stratum of the closed stratified
  variety $(E_{i+1} \times_G Y, E_{i+1} \times_G \hat{\mathcal{S}})$. It follows from
  \cite[I.1.11]{gormac-stratified} that condition
  \ref{enum:stratified-nn} is satisfied.
  (For each $S \in \mathcal{S}$, 
  $f_i\times \id_{\ol S}$ obviously is a normally nonsingular
  inclusion of the same codimension as $f_i$. If this implies the
  same statement on quotient spaces, we can do without
  \ref{enum:G-closed-embedding}.)
\end{proof}

\subsection{Formality of Equivariant Flag Varieties}
\label{sec:form-equiv-flag}

Let $G \supset P \supset B$ be respectively a complex connected
reductive affine algebraic group, a parabolic and a Borel
subgroup.

\begin{theorem}
  \label{t:formality-equivariant-partial-flag-variety}
  If $\mathcal{S}$ is the stratification of $G/P$ into
  $B$-orbits,  
  there is a t-exact equivalence of triangulated categories
  \begin{equation*}
    \dcb_{B,c}(G/P)
    =
    \dcb_B(G/P, \mathcal{S}) 
    \cong 
    \dgPerDer(\Ext(\IC_B(\mathcal{S}))).
  \end{equation*}
  Restriction to the hearts induces an equivalence 
  \begin{equation*}
    \Perv_{B}(G/P)
    =
    \Perv_B(G/P, \mathcal{S}) 
    \cong 
    \dgFilMod(\Ext(\IC_B(\mathcal{S}))).
  \end{equation*}
\end{theorem}

\begin{proof}
  The $B$-stratified variety $(G/P, \mathcal{S})$ has an ABCD-\asystem{}
  by Theorem \ref{t:bptgp} and Propositions \ref{p:T-B-approxi} and
  \ref{p:approx-for-variety}.
  Hence we can apply Theorem \ref{t:das-ziel}. 
  The equalities follow from Proposition \ref{p:orbit-constructible}.
\end{proof}

\begin{remark}
  \label{rem:equivariant-complex}
  The strategy from subsection
  \ref{sec:complex-coefficients} also shows that the complexified
  version of diagram \eqref{eq:the-goal} commutes.
  This implies that Theorem
  \ref{t:formality-equivariant-partial-flag-variety} is also true for
  complex coefficients. 
\end{remark}

\begin{remark}
    \label{rem:formality-equivariant-partial-flag-variety}
    We claim that the diagram 
    \begin{equation}
      \label{eq:forgetful-compatible}
      \xymatrix{
        {\dcb_B(G/P, \mathcal{S})} \ar[r]^-\sim
        \ar[d]^{\Forget}
        &
        {\dgPerDer(\Ext(\IC_B(\mathcal{S})))}
        \ar[d]^{?\Lotimesover{\Ext(\IC_B(\mathcal{S}))} \Ext(\IC(\mathcal{S})) }
        \\
        {\dcb(G/P, \mathcal{S})} \ar[r]^-\sim
        &
        {\dgPerDer(\Ext(\IC(\mathcal{S})))}
      }
    \end{equation}
    is commutative (up to natural isomorphism).
    The horizontal equivalences are those from 
    Theorems \ref{t:formality-equivariant-partial-flag-variety} 
    and \ref{t:formality-partial-flag-variety}, the vertical functors
    are the forgetful functor and the induced extension of scalars
    functor.
    
    The upper horizontal equivalence was established as the inverse
    limit of a sequence of equivalences $(\Formy_n)_{n \in \DN}$ of categories
    (cf.\ proof of Theorem \ref{t:das-ziel}). The lower
    horizontal equivalence can be chosen equal to $\Formy_0$. (If we use our method for constructing the
    ABCD-\asystem{} for $(G/P, \mathcal{S})$ we have 
    $E_0=B \times_T T \times G/P = B \times G/P$ for $T\subset B$ a
    maximal torus, and $B\bl E_0= G/P$ has the stratification
    $\mathcal{S}_0=\mathcal{S}$.) Hence diagram
    \eqref{eq:forgetful-compatible} is commutative.

    Since all functors in diagram 
    \eqref{eq:forgetful-compatible} are triangulated and t-exact, 
    we obtain by restriction a commutative diagram relating the hearts.
\end{remark}

\bibliographystyle{alpha}
\def\weg#1{} \def\cprime{$'$} \def\cprime{$'$} \def\cprime{$'$}

\end{document}